\documentclass[11pt]{amsart}

\usepackage[all]{xy}
\usepackage{fullpage}
\usepackage{latexsym}
\usepackage{amsmath}
\usepackage{amsfonts}
\usepackage{amssymb}
\usepackage{amsthm}
\usepackage{eucal}
\usepackage{enumerate,yfonts}
\usepackage{mathrsfs}
\usepackage{graphicx}
\usepackage{graphics}
\usepackage{epstopdf}
\usepackage{amscd}
\usepackage{bbm}
\usepackage{hyperref}
\usepackage{url}
\usepackage{color}
\usepackage{bbm}
\usepackage{cancel}
\usepackage{enumerate}
\usepackage{amsmath,amsthm}
\usepackage{amssymb}
\usepackage{epsfig}
\usepackage{pstricks}
\usepackage{xy}
\usepackage{xypic}

\newtheorem{thm}{Theorem}[subsection]
\newtheorem{corollary}[thm]{Corollary}
\newtheorem{lemma}[thm]{Lemma}
\newtheorem{lem}[thm]{Lemma}
\newtheorem{proposition}[thm]{Proposition}
\newtheorem{prop}[thm]{Proposition}

\newtheorem{thm-dfn}[thm]{Theorem-Definition}

\theoremstyle{definition}
\newtheorem{definition}[thm]{Definition}
\newtheorem{remark}[thm]{Remark}
\newtheorem{example}[thm]{Example}

\numberwithin{equation}{subsection}

\newcommand{\fg}{{\mathfrak g}}
\newcommand{\ft}{{\mathfrak t}}
\newcommand{\fl}{{\mathfrak l}}

\newcommand{\fC}{{\mathfrak C}}

\newcommand{\fc}{{\mathfrak c}}

\newcommand{\rW}{{\mathrm W}}

\newcommand{\bC}{{\mathbb C}}
\newcommand{\bX}{{\mathbb X}}
\newcommand{\bG}{{\mathbb G}}
\newcommand{\bZ}{{\mathbb Z}}

\newcommand{\mD}{\mathcal{D}}
\newcommand{\mS}{\mathcal{S}}

\newcommand{\mF}{\mathcal{F}}

\newcommand{\mA}{\mathcal{A}}
\newcommand{\mM}{\mathcal{M}}
\newcommand{\mT}{\mathcal{T}}
\newcommand{\mO}{\mathcal{O}}
\newcommand{\mL}{\mathcal{L}}
\newcommand{\mH}{\mathcal{H}}
\newcommand{\mP}{\mathcal{P}}

\newcommand{\mG}{\mathcal{G}}

\newcommand{\mC}{\mathcal{C}}

\newcommand{\sX}{\mathscr{X}}
\newcommand{\sY}{\mathscr{Y}}

\newcommand{\sG}{\mathscr{G}}

\newcommand{\sZ}{\mathscr{Z}}
\newcommand{\sP}{\mathscr{P}}
\newcommand{\sT}{\mathscr{T}}
\newcommand{\sA}{\mathscr{A}}

\newcommand{\scP}{\mathscr{P}}
\newcommand{\sH}{\mathscr{H}}
\newcommand{\sL}{\mathscr{L}}
\newcommand{\sQ}{\mathscr{Q}}
\newcommand{\sD}{\mathscr{D}}

\newcommand{\on}{\operatorname}

\newcommand{\tU}{\widetilde U}

\newcommand{\ra}{\rightarrow}

\newcommand{\s}{\textbackslash}

\newcommand{\is}{\simeq}

\newcommand{\Loc}{\on{LocSys}}
\newcommand{\Bun}{\on{Bun}}
\newcommand{\Sect}{\on{Sect}}
\newcommand{\tC}{\widetilde C}

\newcommand{\quash}[1]{}  
\newcommand{\nc}{\newcommand}

\newcommand{\frakc}{{\mathfrak c}}

\newcommand{\frakl}{{\mathfrak l}}

\newcommand{\fraks}{{\mathfrak s}}

\newcommand{\frakA}{{\mathfrak A}}

\newcommand{\frakC}{{\mathfrak C}}
\newcommand{\frakD}{{\mathfrak D}}

\newcommand{\bbF}{{\mathbb F}}
\newcommand{\bbG}{{\mathbb G}}

\newcommand{\bbQ}{{\mathbb Q}}

\newcommand{\bbT}{{\mathbb T}}

\newcommand{\bbZ}{{\mathbb Z}}
\newcommand{\calA}{{\mathcal A}}

\newcommand{\calE}{{\mathcal E}}
\newcommand{\calF}{{\mathcal F}}

\newcommand{\calK}{{\mathcal K}}
\newcommand{\calL}{{\mathcal L}}

\newcommand{\calO}{{\mathcal O}}

\newcommand{\calT}{{\mathcal T}}

\newcommand{\calV}{{\mathcal V}}

\nc{\al}{{\alpha}} \nc{\be}{{\beta}} \nc{\ga}{{\gamma}}
\nc{\ve}{{\varepsilon}} \nc{\Ga}{{\Gamma}} 
\nc{\La}{{\Lambda}}

\nc{\ad }{{\on{ad }}} \nc{\Ad }{\on{Ad}}

\nc{\aff}{{\on{aff}}} \nc{\Aff}{{\mathbf{Aff}}}
\newcommand{\Aut}{{\on{Aut}}}
\newcommand{\cha}{{\on{char}}}
\nc{\der}{{\on{der}}}

\nc{\diag}{{\on{diag}}}

\nc{\Fl}{{\calF\ell}}

\newcommand{\Gr}{{\on{Gr}}}
\nc{\Hg}{{\on{Higgs}}}
\newcommand{\Hom}{{\on{Hom}}}
\newcommand{\id}{{\on{id}}}
\nc{\Id}{{\on{Id}}}

\nc{\Ind}{{\on{Ind}}}
\newcommand{\Lie}{{\on{Lie}\!\ }}
\nc{\Op}{{\on{Op}}}
\newcommand{\Pic}{{\on{Pic}\!\ }}
\newcommand{\pr}{{\on{pr}}}
\newcommand{\Res}{{\on{Res}}}
\nc{\res}{{\on{res}}}

\newcommand{\Spec}{{\on{Spec}}}

\nc{\tr}{{\on{tr}}}

\newcommand{\GL}{{\on{GL}}}
\nc{\GSp}{{\on{GSp}}} \nc{\GU}{{\on{GU}}} \nc{\SL}{{\on{SL}}}
\nc{\SU}{{\on{SU}}} \nc{\SO}{{\on{SO}}}

\nc{\nh}{{\Loc_{J^p}(\tau')}}
\nc{\bnh}{{\Loc_{\breve J^p}(\tau')}}

\nc{\bU}{{\overline{U}}} \nc{\IC}{{\on{IC}}}
\newcommand{\Nm}{{\on{Nm}}}

\def\xcoch{\mathbb{X}_\bullet}
\def\xch{\mathbb{X}^\bullet}

\newcommand{\AJ}{{\on{AJ}}}

\begin{document}
\title{Geometric Langlands in prime characteristic}
        \author{Tsao-Hsien Chen, Xinwen Zhu}
        \address{Department of Mathematics, Northwestern University,
        2033 Sheridan Road, Evanston, IL 60208, USA}
        \email{chenth@math.northwestern.edu}
        \address{Department of Mathematics, California Institute of Techonology,
        1200 E. California Blvd, Pasadena, CA 91125, USA}
         \email{xzhu@caltech.edu}
  
         
\begin{abstract}
Let $G$ be a semisimple algebraic group over an algebraically closed
field $k$, whose characteristic is positive and does not divide the
order of the Weyl group of $G$, and let $\breve G$ be its Langlands
dual group over $k$. Let $C$ be a smooth projective curve over $k$ of genus at least two. Denote by
$\Bun_G$ the moduli stack of $G$-bundles on $C$ and $ \Loc_{\breve
G}$ the moduli stack of $\breve G$-local systems on $C$. Let
$D_{\Bun_G}$ be the sheaf of crystalline differential operators on
$\Bun_G$. In this paper we construct an equivalence between the
bounded derived category $D^b(\on{QCoh}(\Loc_{\breve G}^0))$ of
quasi-coherent sheaves on some open subset $\Loc_{\breve
G}^0\subset\Loc_{\breve G}$ and bounded derived category
$D^b(D_{\Bun_G}^0\on{-mod})$ of modules over some localization
$D_{\Bun_G}^0$ of $D_{\Bun_G}$. This generalizes the work of
Bezrukavnikov-Braverman  in the $\GL_n$ case.
\end{abstract}
\maketitle

\setcounter{tocdepth}{1} \tableofcontents

\section{Introduction}
\subsection{Geometric Langlands conjecture in prime characteristic}
Let $G$ be a reductive algebraic group over $\bC$ and let
$\breve G$ be its Langlands dual group. Let $C$ be a smooth
projective curve over $\bC$. Let $\Bun_G$ be the stack of
$G$-bundles on $C$ and $\Loc_{\breve G}$ be the stack of de Rham
$\breve G$-local systems on $C$. The geometric Langlands conjecture (GLC), as proposed by Beilinson and Drinfeld, is a conjectural
equivalence between certain appropriately defined category of
quasi-coherent sheaves on $\Loc_{\breve G}$ and certain appropriately
defined category of $\mD$-modules on $\Bun_G$. A precise formulation
of this conjecture (over $\bC$) can be found in the recent work of Dima
Arinkin and Dennis Gaitsgory \cite{AG,Ga}.

The geometric Langlands duality has a classical limit which amounts to the duality of
Hitchin fibrations. The classical duality is established
``generically" by Donagi and Pantev in \cite{DP} over $\bC$.
 
In this paper, we establish a ``generic" characteristic $p$ version
of the geometric Langlands conjecture. Namely, let $G$ be
a semi-simple algebraic group over an algebraically closed field $k$
of characteristic $p$ that does not divide the order of the
Weyl group of $G$, and $\breve G$ be its Langlands dual group, defined over
$k$. Let $C$ be a smooth projective curve over $k$ of genus at least two\footnote{The assumptions on the genus of $C$ and on the semisimplicity of $G$ should not be essential. We impose them to avoid the DG structure on moduli spaces.}. Then we establish an equivalence of
bounded derived category
\begin{equation}\label{GL}
D^b(\mD\on{-mod}(\Bun_G)^0)\simeq D^b(\on{QCoh}(\Loc_{\breve
G})^0),\end{equation} where $\mD\on{-mod}(\Bun_G)^0$ (resp.
$\on{QCoh}(\Loc_{\breve G})^0$) is a certain localization of the
category of $\mD$-modules on $\Bun_G$ (resp. a localization of the
category of quasi-coherent sheaves on $\Loc_{\breve G}$). We call
\eqref{GL} a ``generic" version of the GLC.

One remark is in order. Recall that over a field of positive
characteristic, there are different objects that can be called
$\mD$-modules. In this paper, we use the notion of crystalline
$\mD$-modules, i.e. $\mD$-modules are quasi-coherent sheaves with
a flat connection. Likewise, the stack $\Loc_{\breve G}$ is the
stack of $\breve G$-bundles on $C$ with a flat connection.

\subsection{Summary of the construction}
The case $G=\GL_n$ has been considered by R. Bezrukavnikov and A.
Braverman in \cite{BB} (see \cite{Groe,Trav} for various extensions). The main observation is that, the geometric Langlands duality in characteristic $p$
formulated in the above form can be thought as a twisted version of
its classical limit. Since the classical duality holds
``generically", they proved a ``generic" version of the GLC in the
case when $G=\GL_n$.

Our generalization to any semisimple group $G$ is based on the same
observation, but some new ingredients are needed in this general situation.

One of the main difficulties for general $G$ is that the classical
duality is more complicated. For $G=\GL_n$, the generic fibers of
the Hitchin fibration are the Picard stacks of line bundles on
the corresponding spectral curves and the duality of Hitchin
fibrations in this case essentially amounts to the self-duality of the Jacobian of an
algebraic curve. However, for general $G$, the fibers of the Hitchin
fibration involve more general Picard stacks, such as the Prym
varieties, etc., and the duality of the Hitchin fibrations for $G$ and
$\breve G$ over $\bC$ are the main theme of \cite{DP} (see \cite{HT} for the case $G=\SL_n$). As commented
by the authors, the arguments in \cite{DP} use transcendental
methods in an essential way and therefore cannot be applied to our
situation directly.

Our first step is to extend the classical duality to any reductive
group $G$ over any algebraically closed field $k$ whose characteristic
does not divide the order of the Weyl group of $G$. Let us first
give its statement, and leave the details to
\S\ref{classical duality}. For a reductive group $G$ and a smooth projective curve
$C$ over $k$, and a positive line bundle $\mL$ on $C$,  let $\Hg_{G,\mL}\to B$ denote the corresponding Hitchin fibration, on
which the Picard stack $\sP_\mL\to B$ acts (see \S \ref{Hf} for a review). There is an open subset $B^0\subset B$
such that $\sP_\mL|_{B^0}$ is a Beilinson 1-motive (a Picard stack that
is essentially an abelian variety, see Appendix \ref{A}). Fix a
nondegenerate bilinear form on the Lie algebra $\fg$ of $G$, one can
identify the Hitchin base $B$ and the corresponding open subset
$B^0$ for $G$ and $\breve G$. The classical duality is the following assertion.
\begin{thm}\label{intro:classical}
For a positive line bundle $\mL$ on $C$,
there is a canonical isomorphism of Picard stacks
\begin{equation}\label{intro:classical equ}
\frakD_{cl}:(\sP_\mL|_{B^0})^\vee\simeq \breve \sP_\mL|_{B^0},
\end{equation}
where $(\sP_\mL|_{B^0})^\vee$ is the dual Picard stack of $\sP_\mL|_{B^0}$
(as defined in Appendix \ref{A}).
\end{thm}

\medskip

Now assume that the characteristic of $k$ is positive. In addition, assume that the genus of $C$ is at least two and that $\mL=\omega_C$ is the canonical bundle. 
We will omit the subscript $\omega_C$ and write 
$\sP=\sP_{\omega_C}$ etc.
The second step then is
to construct a twisted version of the above classical duality in this
situation. To explain its meaning, let us first introduce a notation: If
$X$ is a stack over $k$, we denote by $X'$ its Frobenius twist,
i.e., the pullback of $X$ along the absolute Frobenius endomorphism
of $k$. Let $F_X:X\to X'$ denote the relative Frobenius morphism. We will
replace both sides of \eqref{intro:classical equ} by certain torsors
under ${\sP'}^\vee$ and $\breve \sP'$.

We begin to explain the $\breve \sP'$-torsor $\breve\sH$, which was introduced in \cite{CZ}. There is a smooth
commutative group scheme $\breve J'$ on $C'\times B'$ and $\breve
\sP'$ in fact classifies $\breve J'$-torsors. Let us denote by $\breve
J^p$ the pullback of $\breve J'$ along the relative
Frobenius $F_{C'\times B'/B'}:C\times B'\to C'\times B'$. This is a group scheme with a canonical connection along $C$,
and therefore it makes sense to talk
about $\breve J^p$-local systems on $C\times B'$ and their
$p$-curvatures (see \cite[Appendix]{CZ}  for generalities). Let
$\breve\sH$ be the stack of $\breve
J^p$-local systems with some specific $p$-curvature
$\breve\tau'$. This is a $\breve \sP'$-torsor.

Next we explain the ${\sP'}^\vee$-torsor $\sT_{\sD(\theta_m)}$.
According to general nonsense (Appendix \ref{A}), such a torsor
gives a multiplicative $\bG_m$-gerbe $\sD$ on $\sP'$ and
vice versa. So it is enough to explain this multiplicative
$\bG_m$-gerbe $\sD(\theta_m)$ on $\sP'$. First recall that the sheaf of crystalline differential operators on $\sP$ can be regarded as a $\bG_m$-gerbe $\sD_{\sP}$ on the cotangent bundle $T^*\sP'$.
We will construct a 1-form $\theta_m$ on $\sP'$, which is multiplicative (in the sense of \S \ref{appen:var}). Now, $\sD=\sD(\theta_m)$ is the gerbe on $\sP'$ obtained via pullback of $\sD_{\sP}$ along the map $\theta_m: \sP'\to T^*\sP'$.

The twisted version of the classical duality is the following assertion
\begin{thm} \label{intro:twist} Over $B^{'0}$, there is a canonical isomorphism of ${\sP'}^\vee\simeq \breve\sP'$-torsors
\[\frakD:\sT_{\sD(\theta_m)}|_{B^{'0}}\simeq \breve\sH|_{B^{'0}}.\]
\end{thm}

\medskip

The final step towards \eqref{GL} is to establish two abelianization
theorems. Another difference between the geometric Langlands correspondence for $\GL_n$ and for a general group $G$ is that in the latter case, there is no canonical equivalence in general. As is widely known to experts (e.g. see \cite{FW}), the geometric Langlands correspondence for general $G$ should depend on a choice of theta characteristic of the curve $C$.

Let us fix a square root $\kappa$ of $\omega_C$. 
Then
the Kostant section of $\Hg'_{G}\to B'$ induces a map
$\epsilon_{\kappa'}:\sP'\to\Hg'_G$. The first abelianization theorem
asserts a canonical isomorphism
$$\epsilon_{\kappa'}^*\sD_{\Bun_G}\simeq \sD(\theta_m),$$
where $\sD_{\Bun_G}$ is the $\bG_m$-gerbe (on $\Hg'_G=T^*\Bun'_G$) of
crystalline differential operators on $\Bun_G$ and $\sD(\theta_m)$ is
the $\bG_m$-gerbe on $\sP'$ mentioned above. 

On the dual side, we
constructed a canonical morphism in \cite{CZ}
\[\fC: \breve\sH\times^{\breve\sP'}\Hg'_{\breve G}\to \Loc_{\breve G},\]
and the Kostant section of $\Hg'_{\breve G}\to B'$ induces an
isomorphism
\[\fC^{}_{\kappa}:\breve\sH\is\Loc_{\breve G}^{reg},\]
where $\Loc_{\breve G}^{reg}$ is a certain open substack of
$\Loc_{\breve G}$ (see \cite[Remark 3.14]{CZ}).

Combining the above three steps and a general version of the
Fourier-Mukai transform (Appendix \ref{A}) will give the desired
equivalence \eqref{GL}.

Let us mention that the morphism $\fC$ was obtained in \cite{CZ} as a
version of Simpson correspondence for smooth projective curves in positive characteristic. 

Finally in \S \ref{a_chi and b_chi} and \S \ref{tensoring action}, we discuss how the equivalence constructed above depends on the choice of the theta characteristic. This can be regarded as a verification of the predictions of \cite[\S 10]{FW} in our settings.

\subsection{The Langlands transform}
To claim that the above equivalence is the conjectural geometric
Langlands transform, one needs to verify several properties that it is supposed to satisfy.  We will only briefly discuss these properties (see \cite{Ga} for more details), and leave the verifications to our next work.

The first property is that the equivalence should intertwine the action of  the Hecke operators on the automorphic side and the action of the Wilson operators on the spectral side. Recall that in the case $k=\bC$, both categories
$D(\mD\on{-mod}(\Bun_G))$ and $D(\on{Qcoh}(\Loc_{\breve G}))$ admit
actions of a family of commuting operators, labeled by points $x$ on
the curve and representations $V$ of the group $\breve G$. Namely,
for $x\in C$ and $V\in \on{Rep}(\breve G)$, there is the so-called
Wilson operator $W_{V,x}$ acting on $\on{Qcoh}(\Loc_{\breve G})$ by
tensoring with the locally free sheaf
$V_{E_{univ}}|_{\Loc_{\breve G}\times \{x\}}$. On the other side,
there is the Hecke operator $H_{V,x}$ acting on
$\mD\on{-mod}(\Bun_G)$ via certain integral transform (e.g. see \cite[\S 5]{BD}). The second property is that the equivalence should satisfy the Whittaker normalization. Namely, the 
Whittaker $\mD$-module $\mF_\Psi$ on $\Bun_G$ is supposed to transformed to the structure sheaf $\mO_{\Loc_{\breve G}}$. 

In the positive characteristic, it is yet not clear how to define Hecke operators (except those corresponding to minuscule coweights) due to lack of the notion of intersection cohomology $\mD$-modules. Our observation is that by the geometric Casselman-Shalika formula (\cite{FGV}), the two properties together will imply that the Whittaker coefficients of $\mD$-modules on $\Bun_G$ can be calculated by applying the Wilson operators on their Langlands transforms and then taking the global sections. This is a well formulated statement in characteristic $p$ and we will verify in the future work that this is satisfied by the equivalence constructed here.

The third property is that the equivalence should be compatible with Beilinson-Drinfeld's construction of automorphic $\mD$-modules via opers (\cite{BD}). 
In the case $G=\GL_n$ this property has been verified in \cite{BT}. We plan to return to this in the future work.

\subsection{Structure of the article}

Let us now describe the contents of this paper in more detail.

In \S \ref{Hf} we collect some facts about Hitchin fibrations  that
are used in this paper. Main references are \cite{N1,N2}.

In \S \ref{classical duality} we prove the classical duality, i.e.,
the duality of Hitchin fibrations. This extends the work of \cite{DP}
(over $\bC$) to any algebraically closed field whose characteristic
does not divide the order of the Weyl group of $G$. In \S \ref{l_J}, we discuss the compatibility of the classical duality with twisting by $Z(\breve G)$-torsors. This is used to study the dependence of the equivalence \eqref{GL} on the choice of the theta characteristic in \S \ref{a_chi and b_chi} and \S \ref{tensoring action}.

In \S \ref{mult one forms} we construct a canonical multiplicative 1-form $\theta_m$ on $\sP'$.

In  \S \ref{Main} we deduce our main Theorem \ref{main} from the twisted duality (see \S\ref{tw}) and the two abelianization theorems (see \S\ref{ab thm}).

There are three appendices at the end of the paper. 

In \S\ref{A} we collect some basic facts about Beilinson's 1-motive
and duality on Beilinson's 1-motive. In particular, we state a
general version of Fourier-Mukai transforms for Beilinson's 1-motives.

In \S\ref{B} we recall the basic theory of $\mD$-modules over varieties and stacks in positive characteristic, following \cite{BMR,BB,OV,Trav}.

In \S\ref{C} we prove the abelian duality for good Beilinson's 1-motives. It asserts that the derived category of $\mD$-modules on a ``good" Beilinson's 1-motive $\sA$ is equivalent to the derived category of quasi-coherent sheaves on the universal extension $\sA^\natural$ by vector groups of its dual $\sA^\vee$. 

\subsection{Notations}
\subsubsection{Notations related to algebraic stacks} Our terminology
of algebraic stacks follows the book \cite{LB}. Let $k$ be an
algebraically closed field and let $p$ be the characteristic
component of $k$. Let $S$ be a Noetherian scheme over $k$. In this
paper, an algebraic stack $\sX$ over $S$ is a stack such that the
diagonal morphism
$$\Delta_S:\sX\ra\sX\times_S\sX$$
is representable and quasi-compact and such that there exists a smooth
presentation, i.e., a smooth, surjective morphism $X\ra\sX$ from a
scheme $X$.

An algebraic stack $\sX$ is called smooth over $S$ if for every $S$-scheme
$U$ mapping smoothly to $\sX$, the structure morphism $U\ra S$ is
smooth.

For any algebraic stack $\sX$, we denote by $\sX_{Et}$ the big
\'{e}tale site of $\sX$. We denote by $\sX_{sm}$ the smooth site on
$\sX$, i.e., the site for which the underling category has
objects consisting of $S$-schemes $U$ together with a smooth morphism
$U\ra\sX$ and has morphisms $V\to U$ smooth 2-morphisms over $\sX$ and for
which covering maps are smooth surjective maps of schemes. If $\sX$
is a Deligne-Mumford stack, we denote by $\sX_{et}$ the small
\'{e}tale site of $\sX$.

Let $\sY\ra\sX$ be a
quasi-projective morphism of algebraic stacks, with $\sX$ smooth and proper over $S$. We denote by
$\on{Sect}_S(\sX,\sY)$ the stack of ``sections" of $\sY$ over
$\sX$, i.e., for any $u:U\ra S$ we have
$$\on{Sect}_S(\sX,\sY)(U)=\on{Hom}_{\sX}(\sX\times_SU,\sY).$$
If the base scheme $S=\on{Spec}(k)$, we write
$\on{Sect}(\sX,\sY)=\on{Sect}_S(\sX,\sY)$. 

If $\sX$ is a smooth algebraic stack over $S$, we define the
relative tangent stack $T(\sX/S)$ as the stack that assigns every $\on{Spec} R\ra S$, the groupoid
$$T(\sX/S)(R):=\sX(R[\epsilon]/\epsilon^2).$$
It is  algebraic and the natural inclusion $R\ra
R[\epsilon]/\epsilon^2$ induces a morphism
$$\tau_{\sX}:T(\sX/S)\ra\sX.$$
It is known that $T(\sX/S)$ is a relative Picard stack over $\sX$. Therefore, one can associate to it a complex in $D^{[-1,0]}(\sX,\bZ)$,
called the relative tangent complex:
$$T_{\sX/S}^{\bullet}=\{T_{\sX/S}\ra T_{\sX}\}.$$
The relative cotangent stack is then defined as
$$T^*(\sX/S):=\on{Spec}_{\sX}(\on{Sym}_{\mO_{\sX}}H^0(T^{\bullet}_{\sX/S})).$$

Let $f:\sX\ra\sY$ be a (representable) morphism between two
algebraic stacks over $S$. We denote the cotangent morphism as the following diagram of maps
\begin{equation}\label{cotangent morphism}
\xymatrix{T^*(\sY/S)\times_{\sY}\sX\ar[r]^{\ \ \ \ f_d}\ar[d]^{f_p}&T^*(\sX/S)\\
T^*(\sY/S)}.
\end{equation}

\subsubsection{Notations related to Frobenius morphism} Let $S$ be a
Noetherian $k$-scheme and $\sX\ra S$ be an algebraic stack over $S$. If
$p\calO_S=0$, we denote by  $Fr_S:S\ra S$ the absolute Frobenius
map of $S$. We have the following commutative diagram
$$\xymatrix{\sX\ar[r]^{F_{\sX/S}}\ar[dr]&\sX^{(S)}\ar^{\pi_{\sX/S}}[r]\ar[d]&\sX\ar[d]\\
&S\ar[r]^{Fr_S}&S}$$ where the square is Cartesian. We call
$\sX^{(S)}$ the Frobenius twist of $\sX$ along $S$, and
$F_{\sX/S}:\sX\ra\sX^{(S)}$ the relative Frobenius morphism. If the
base scheme $S$ is clear, $\sX^{(S)}$ is also denoted by $\sX'$ for
simplicity.

\subsubsection{Notation related to torsors} Let $\mG$ be a smooth
affine group scheme over $X$, and $E$ be a $\mG$-torsor on $X$. We
denote by $\Aut (E)=E\times ^{\mG}\mG$ the adjoint torsor and
$\ad (E)$ or $\fg_E=E\times^{\mG}\Lie\mG$ the adjoint bundle.

\subsection{Acknowledgement} We thank Roman Bezrukavnikov, Roman Travkin and Zhiwei Yun for useful discussions, and Uwe Weselmann for helpful comments and suggestions. The first author would like to thank his advisor 
Roman Bezrukavnikov for continuous interest in this work and many helpful advices. 
T-H. Chen is partially supported by NSF under the agreement No.DMS-1128155. X. Zhu is partially supported by NSF grant DMS-1001280/1313894 and DMS-1303296 and AMS Centennial Fellowship.

\section{The Hitchin fibration}\label{Hf}
In this section, we review some basic geometric facts of Hitchin
fibrations, following \cite{N1,N2}. Only \S \ref{tau and c} is
probably new.

\subsection{Notations related to reductive groups}\label{bilinear form}
Let $G$ be a reductive algebraic group over $k$ of rank $l$. We
denote by $\breve G$ its Langlands dual group over $k$. We denote by
$\fg$ (resp. by $\breve\fg$) the Lie algebra of $G$ (resp. $\breve G$). Let $T$ denote the abstract Cartan of $G$ with its Lie algebra $\ft$. 
The counterparts 
on the Langlands dual side are denoted by
$\breve T, \breve\ft$. We denote by $\rW$  the abstract Weyl group of $G$, which acts on $T$ and $\breve T$.
We denote by $\bX^{\bullet}(T)$ or simply by $\bX^{\bullet}$ (resp. by $\bX_{\bullet}(T)$ or simply by $\bX_{\bullet}$) the character (resp.
the cocharacter) group of $T$. Let $\Phi\subset \bX^{\bullet}(T)$ be the set of roots.
Sometimes, we also fix a set of simple roots $\{\al_1,\ldots,\al_l\}$ and an embedding $\ft\subset \fg$. Then for $\al\in\Phi$, let $\fg_\al\subset\fg$ denote the corresponding root subspace.

From now on, we assume that the $\cha\ k=p$ is zero or  $p\nmid |\rW|$. We fix a $\rW$-invariant non-degenerate bilinear form $(\ ,\
):\ft\times\ft\ra k$ and identify $\ft$ with $\breve\ft$ using $(\
,\ )$. This invariant form also determines a unique $G$-invariant
non-degenerate bilinear form $\fg\times\fg\to k$, still denoted by
$(\ ,\ )$. Let $\fg\simeq \fg^*$ be the resulting $G$-equivariant
isomorphism.

\subsection{Hitchin map}
Let $k[\fg]$ and $k[\ft]$ be the algebras of polynomial functions on
$\fg$ and on $\ft$ respectively. By Chevalley's theorem, we have an isomorphism
$k[\fg]^{G}\is k[\ft]^{\rW}$. Moreover, $k[\ft]^\rW$ is isomorphic to
a polynomial ring of $l$ variables $u_1,\ldots,u_l$ and each $u_i$ is
homogeneous in degree $e_i$. Let $\fc=\on{Spec}(k[\ft]^\rW)$. The
natural $\bG_m$ action on $\fg$ induces a $\bG_m$-action on $\fc$
and under the isomorphism
$\fc\is\on{Spec}(k[u_1,\ldots,u_l])\is\mathbb A^l$ the action is
given by
\[h\cdot(a_1,\ldots, a_l)=(h^{e_1}a_1,\ldots, h^{e_l}a_l).\]

Let  $\chi:\fg\ra\fc$ be the map induced by $k[\fc]\is
k[\fg]^G\hookrightarrow k[\fg]$. It is a $G\times\bG_m$-equivariant
map where $G$ acts trivially on $\fc$. 
Similarly, let $\boldsymbol\pi:\ft\ra\fc$ be the map induced by $k[\fc]\hookrightarrow k[\ft]$, which is also $\bG_m$-equivariant. Let $\mL$ be an invertible
sheaf on $C$ and $\mL^\times$ be the corresponding
$\bG_m$-torsor. We denote by
$\fg_{\mL}=\fg\times^{\bG_{m}}\mL^\times$, $\ft_{\mL}=\ft\times^{\bG_m}\mL^\times$,  and
$\fc_{\mL}=\fc\times^{\bG_m}\mL^\times$ the $\bG_m$-twist of $\fg$, $\ft$,
and $\fc$ with respect to the natural $\bG_m$-action.

Let $\on{Higgs}_{G,\mL}=\on{Sect}(C,[\fg_{\mL}/G])$ be the stack of
sections of $[\fg_{\mL}/G]$ over $C$. I.e., for each $k$-scheme $S$
the groupoid  $\on{Higgs}_{G,\mL}(S)$ consist of maps:
\[h_{E,\phi}:C\times S\ra[\fg_{\mL}/G],\]
or equivalently, those maps
\[h_{E,\phi}:C\times S\ra [\fg/G\times\bG_m]\] such that the composition of
$h_{E,\phi}$ with the projection $[\fg/G\times\bG_m]\ra B\bG_m$ is
given by the $\bG_m$-torsor $\mL^\times$. 
Explicitly, $\on{Higgs}_{G,\mL}(S)$ consist of pairs $(E,\phi)$ (called Higgs bundles), where
$E$ is an $G$-torsor over $C\times S$ and $\phi$ is an element in
$\Gamma(C\times S,\ad (E)\otimes \mL)$ known as the Higgs field. If the group $G$ is clear
from the context, we simply write $\on{Higgs}_{\mL}$ for
$\on{Higgs}_{G,\mL}$.

Let $B_{\mL}=\on{Sect}_{\on{Spec} k}(C,\fc_{\mL})$ be the scheme of
sections of $\fc_{\mL}$ over $C$. I.e., for each $k$-scheme $S$,
$B_{\mL}(S)$ is the set of sections
$$b:C\times S\ra\fc_{\mL},$$
or equivalently, those maps \[b:C\times S\ra [\fc/\bG_m]\]
such that the composition of $b$ with the projection $[\fc/\bG_m]\ra
B\bG_m$ is given by $\mL^\times$.
It is called the Hitchin base of $G$.

The natural $G$-invariant projection $\chi:\fg\ra\fc$  induces a map
$$[\chi_{\mL}]:[\fg_{\mL}/G]\ra\fc_{\mL},$$
or more generally 
\begin{equation}
\label{twist of chi}[\chi/G\times
\bG_m]:[\fg/G\times\bG_m]\ra [\fc/\bG_m].\end{equation}
The map $[\chi_{\mL}]$ induces a natural map
$$h_{\mL}:\on{Higgs}_{\mL}=\on{Sect}(C,[\fg_{\mL}/G])\ra\on{Sect}(C,\fc_{\mL})=B_{\mL}.$$
\begin{definition}
We call $h_{\mL}:\on{Higgs}_{\mL}\ra B_{\mL}$ the Hitchin map
associated to $\mL$.
\end{definition}

For any
$b\in B_\mL(S)$ we denote by $\on{Higgs}_{\mL,b}$ the fiber product
$S\times_{B_\mL}\on{Higgs}_{\mL}$. 

Observe that the invariant bilinear form $\ft\times \ft\to k$ induces a canonical isomorphism $\ft\simeq \ft^*=:\breve\ft$, compatible with the $\rW$-action. Therefore, there is a canonical isomorphism $\fc\simeq \breve\fc$ and $B_{\mL}\simeq \breve B_{\mL}$. In what follows, we will identify them.

Let $\omega=\omega_C$ be the canonical line bundle of $C$. We are mostly interested in the case $\mL=\omega$. For
simplicity, from now on we denote $B=B_{\omega}$,
$\on{Higgs}=\on{Higgs}_{\omega}$, $h=h_{\omega}:\on{Higgs}\ra B$, and $\Hg_b=\Hg_{\omega_C,b}$. We sometimes also write
$\on{Higgs}_G$ for $\on{Higgs}$ to emphasize the group $G$. Observe that the bilinear form as in
\S\ref{bilinear form} together with the Serre duality induces an isomorphism $\on{Higgs}\simeq
T^*\Bun_G$ (cf. \cite{H}).

\subsection{The Kostant section}\label{kostant section}
In this section, we recall the construction of the Kostant section
of the Hitchin map $h_{\mL}$. For each simple root $\alpha_i$ we choose
a nonzero vector $f_i\in\fg_{-\alpha_i}$. Let $f=\oplus_{i=1}^l
f_i\in\fg$.  We complete $f$  into an $\fraks\frakl_2$ triple $\{f,h,e\}$ and denote by $\fg^e$ the centralizer of $e$ in $\fg$. A theorem of
Kostant says that $f+\fg^e$ consist of regular elements in $\fg$ and
the restriction of $\chi:\fg\ra\fc$ to $f+\fg^e$ is an isomorphism
onto $\fc$. We denote by $$kos:\fc\is f+\fg^e$$ the inverse of
$\chi|_{f+\fg^e}$. Let $\rho(\bG_m)$ denote the following $\bG_m$-action on $\fg$: It acts trivially on $\ft$, and on $\fg_{\alpha}$
by $\rho(t)x=t^{\on{ht}(\alpha)}x$ where
$\on{ht}(\alpha)=\sum n_i$ if $\alpha=\sum n_i\alpha_i$. We have
$\rho(t)f=t^{-1}f$ and $\rho(t)e=te$, in particular $\fg^e$ is
invariant under $\rho(\bG_m)$. We define a new $\bG_m$-action on
$\fg$ by  $\rho^+(t)=t\rho(t)$. Then $\rho^+(t)f=f$ and
$\rho^+(\bG_m)$ preserves $f+\fg^e$. With respect to this action, the isomorphism $kos:\fc\is
f+\fg^e$ is $\bG_m$-equivariant.

The diagonal map $\bG_m\ra\bG_m\times\bG_m$ induces a map
\[[\fg/\rho^+(\bG_m)]\ra[\fg/\bG_m\times\rho(\bG_m)].\]
By precomposing with the map
$[\fc/\bG_m]\stackrel{kos}\is[f+\fg^e/\rho^+(\bG_m)]\ra[\fg/\rho^+(\bG_m)]$
we obtain
\[[\fc/\bG_m]\ra[\fg/\bG_m\times\rho(\bG_m)].\]

If the action of $\rho(\bG_m)$ on $\fg$  factors through the adjoint
action of $G$, for example when $G$ is adjoint, then there is a map
$[\fg/\bG_m\times\rho(\bG_m)]\ra[\fg/\bG_m\times G]$ which defines
a section
\[[\fc/\bG_m]\ra[\fg/\bG_m\times\rho(\bG_m)]\ra[\fg/\bG_m\times G]\]
of \eqref{twist of chi}, and in particular, we get a section of
$h_{\mL}$. In general, the action $\rho(\bG_m)$ does not necessarily factor
through $G$, but its square does since it is given by the co-character
$2\rho:\bG_m\ra G$ where $2\rho$ is the sum of positive coroots. So
if we denote $\bG_m^{[2]}\ra\bG_m$ the square map (so $\bG_m^{[2]}$ is isomorphic to $\bG_m$, but regarded as its the double cover), we get a map
\[\eta^{1/2}:[\fc/\bG_m^{[2]}]\ra[\fg/\bG_m^{[2]}\times\rho(\bG_m^{[2]})]\ra[\fg/\bG_m^{[2]}\times G].\]

Let $\mL^{1/2}$ be a square root of  $\mL$. Then every $b:S\times
C\ra[\fc/\bG_m]$ in $B_{\mL}(S)$  factors through a unique map
$b^{1/2}:S\times C\ra[\fc/\bG_m^{[2]}]$. Therefore, by composing
with $\eta^{1/2}$, we get a lift of $b$:
\[\eta^{1/2}(b):S\times C\stackrel{b^{1/2}}\longrightarrow[\fc/\bG_m^{[2]}]\stackrel{\eta^{1/2}}
\longrightarrow[\fg/\bG_m^{[2]}\times G]\ra[\fg/\bG_m\times G].\]
The assignment $b\ra\eta^{1/2}(b)$ defines a section
\[\eta_{\mL^{1/2}}:B_{\mL}\ra\on{Higgs}_{\mL}\]
of  the Hitchin map $h_{\mL}$.

We fix a
square root $\kappa=\omega^{1/2}$ (called a theta characteristic) of
$\omega$ and write
$\kappa=\eta_{\kappa}:B\ra\on{Higgs}$.

\subsection{Cameral curve}\label{cameral}
 For any
$b\in B_\mL(S)$, 
the cameral curve $\tC_b$ is defined as the fiber product
\[\xymatrix{\tC_b\ar[r]\ar[d]^{\pi_b}&\ft_\mL\ar[d]\\
C\times S\ar[r]^b&\fc_\mL}.\]
When $b=\id:B_\mL\ra B_\mL$, the corresponding cameral curve 
$\tC_\mL:=\tC_{b}$ is called the universal cameral curve.
For simplicity, we will write $\tC=\tC_\omega$, $\pi=\pi_{b}:\tC\ra C\times B$.

\subsection{The universal centralizer group schemes}\label{univ cent}
Consider the group scheme $I$ over $\fg$ consisting of pairs
$$I=\{(g,x)\in G\times\fg\mid \on{Ad}_g(x)=x\}.$$
We define $J=kos^{*}I$, where $kos:\fc\ra\fg$ is the Kostant
section. This is called the universal centralizer group scheme of
$\fg$ (see Proposition \ref{J}). To study it, it is convenient to
introduce two auxiliary group schemes. We define $J^1=\on{Res}_{\ft/\fc}(T)^\rW$ and let $J^0$
to be the neutral component of $J^1$. All the group schemes $J$,
$J^0$ and $J^1$  are smooth commutative group schemes over $\fc$.
The following proposition is proved in \cite{N1} (see also
\cite{DG}).
\begin{proposition}\label{J}
\
\begin{enumerate}
\item There is a unique morphism of group schemes $a:\chi^{*}J\ra I\subset G\times\fg$, which extends the
canonical isomorphism $\chi^{*}J|_{\fg^{reg}}\is I|_{\fg^{reg}}$.
.

\item There are natural inclusions $J^0\subset J\subset J^1$.

\item The inclusion $J\subset J^1=\on{Res}_{\ft/\fc}(T)^\rW$ in part (2)
defines a morphism 
\[\mathrm j^1:\boldsymbol\pi^*J\ra T\times\ft\]
of group schemes over $\ft$, which is an isomorphism over $\ft^{rs}$,

\end{enumerate}
\end{proposition}
All the above constructions can be twisted. Namely, there are
$\bG_m$-actions on $I$, $J$, $J^1$ and $J^0$. Moreover, the
$\bG_m$-action on $I$ can be extended to a $G\times\bG_m$-action
given by $(h,t)\cdot(x,g)=(t\cdot hxh^{-1},hgh^{-1})$. The natural morphisms 
$J
\ra\fc$ and $I\ra \fg$ are $\bG_m$-equivariant, and therefore we can
twist everything by the $\bG_m$-torsor $\mL^\times$ to get
$J_{\mL}\ra\fc_{\mL}$, $I_{\mL}\ra\fg_{\mL}$ where
$J_{\mL}=J\times^{\bG_m}\mL^\times$ and
$I_{\mL}=I\times^{\bG_m}\mL^\times$ . Similarly, we have
$J^0_{\mL}\ra\fc_{\mL}$ and $J^1_{\mL}\ra\fc_{\mL}$, and 
there are natural inclusions $J^0_\mL\subset J_\mL\subset J^1_\mL$.
The group
scheme $I_{\mL}$ over $\fg_{\mL}$ is equivariant  under the
$G$-action, hence it descends to a group scheme $[I_{\mL}]$ over
$[\fg_{\mL}/G]$.



\subsection{Symmetries of Hitchin fibration}\label{symmetry of Hitchin}
Let $b:S\ra B_{\mL}$ be an $S$-point of $B_{\mL}$, corresponding to a map
$b:C\times S\ra\fc_{\mL}$. Pulling back $J_\mL\ra\fc_{\mL}$
along this map, we obtain a smooth group scheme $J_b=b^{*}J$ over $C\times
S$.

Let $\sP_b$ be the Picard category of $J_b$-torsors over $C\times
S$. The assignment $b\ra\sP_b$ defines a Picard stack over $B$,
denoted by $\sP_\mL$. Let us fix $b\in B_\mL(S)$, and let $(E,\phi)\in\on{Higgs}_{\mL,b}$
corresponding to the map $h_{E,\phi}:C\times S\ra [\fg_{\mL}/G]$. Observe that the morphism $\chi^*J\to I$ in
Proposition \ref{J} induces $[\chi_{\mL}]^{*}J_\mL\ra [I_\mL]$ of group
schemes over $[\fg_{\mL}/G]$.
Pulling back to $C\times S$ using $h_{E,\phi}$, we get a map
 \begin{equation}\label{actEphi}
a_{E,\phi}:J_b\ra h_{E,\phi}^{*}[I]=\Aut(E,\phi)\subset \Aut(E),
\end{equation}
which allows us to twist
$(E,\phi)\in\on{Higgs}_{\mL,b}$ by a $J_b$-torsor. This construction defines an action
of $\sP_{\mL}$ on $\on{Higgs}_{\mL}$ over $B_{\mL}$.

Let $\on{Higgs}^{reg}_{\mL}$ be the open stack of $\on{Higgs}_{\mL}$ consisting of
$(E,\phi):C\to [\fg_\mL/G]$ that factors through $C\to
[(\fg^{reg})_\mL/G]$. If $(E,\phi)\in\on{Higgs}^{reg}_{\mL}$, then
$a_{E,\phi}$ above is an isomorphism. The Kostant section $\eta_{\mL^{1/2}}:
B_{\mL}\to \on{Higgs}_{\mL}$ factors through $\eta_{\mL^{1/2}}: B_{\mL}\to \on{Higgs}^{reg}_{\mL}$. Following \cite[\S4]{N1}, we
define $B^0_{\mL}$ as the  open sub-scheme of $B_{\mL}$ consisting of $b\in
B_\mL(k)$ such that the image of the map $b:C\ra\fc_{\mL}$
intersects the discriminant divisor transversally.  The following
proposition can be extracted from \cite{DG,DP,N1}:

\begin{proposition}\label{reg torsor}
(1) The stack $\on{Higgs}^{reg}_{\mL}$ is a $\sP_{\mL}$-torsor, which can be trivialized
by a choice of a Kostant section $\eta_{\mL^{1/2}}$. 

(2) One has $\on{Higgs}^{reg}_{\mL}\times_{B_{\mL}}B^0_{\mL}=\on{Higgs}_{\mL}\times_{B_{\mL}}B^0_{\mL}$.

(3) The restriction of the universal cameral curve $\tC_\mL|_{B^0}\to B^0_{\mL}$ to $B^0_{\mL}$ is smooth. The restriction $\sP_{\mL}|_{B^0_{\mL}}$ to $B^0_{\mL}$ is a Beilinson's 1-motive . 
\end{proposition}

\begin{remark}\label{transversal locus}
Let $\underline{Disc}:\ft\ra k$ be the discriminant function 
defined by 
$$\underline{Disc}=\prod_{\alpha\in\Phi} d\alpha,$$ where $\Phi$ is the set of roots of $G$. 
The function $\underline{Disc}$ is $\rW$-invariant, and thus descends to a function 
$Disc$ on $\fc$. Moreover, the function  $Disc:\fc\ra k$ is $\bG_m$-equivariant where 
$\bG_m$ acts on $k$ via the character $t\ra t^N$ and  $N=|\Phi|$.  Let 
$Disc_{\mL}:\fc_{\mL}\ra\mL^N$ be the twist of $Disc$.  For any $b:C\ra\fc_\mL$, we get a section  
\[s_b\in\Gamma(C,\mL^N).\]
The zeros of $s_b$ is the branch loci $\mathcal B$ of the cameral cover $\pi_b:\tC_b\ra C$. If
$b\in B^{0}_{\mL}(k)$, then $\mathcal B$ is multiplicity free.
Note that if $\text{deg}\ \mL>0$ the 
branch loci $\mathcal B$ is non-empty.

\end{remark}

\subsection{The tautological section $\tau:\fc\to \Lie J$}\label{tau and c}
Recall that by
Proposition \ref{J}, there is a canonical isomorphism
$\chi^*J|_{\fg^{reg}}\simeq I|_{\fg^{reg}}$. The sheaf of Lie
algebras $\Lie (I|_{\fg^{reg}})\subset \fg^{reg}\times \fg$ admits a
tautological section $\tilde\tau:\fg^{reg}\ra\Lie(I|_{\fg^{reg}})$
given by $x\mapsto x\in \Lie I_x$ for $x\in
\fg^{reg}$. This section descends to a tautological section
$\tau:\fc\ra\Lie J$.\footnote{Indeed, one can check that $\tau$ is 
equal to $kos^*(\tilde\tau)$, the pullback of $\tilde\tau$ along the 
Kostant section $kos:\fc\ra\fg^{reg}$.} Recall the following property of $\tau$ \cite[Lemma 2.2]{CZ}

\begin{lem}\label{taut:prolong}
Let $x\in \fg$, and $a_x: J_{\chi(x)}\to I_x\subset G$ be the homomorphism as in Proposition \ref{J} (1). Then $da_x(\tau(x))=x$, where $da_x$ denotes the differential of $a_x$.
\end{lem}

Let us regard $\Lie J$ as a scheme over $\fc$. Besides the section $\tau$, there is a canonical map $c:\Lie J\to \fc$ such that $c\tau=\id$. Namely, if we regard $\Lie
(I|_{\fg^{reg}})$ as a scheme, then there is a natural map $\Lie
(I|_{\fg^{reg}})\to \fc$ given by
$$
\Lie (I|_{\fg^{reg}})\subset \fg\times\fg^{reg} \to \fc\times \fg^{reg}\to \fc,
$$
which also descends to a morphism $c:\Lie J\to \fc$.

The morphisms $\tau$ and $c$ have global counterparts (see also \cite[\S 2.3]{CZ}).  Observe that $\bG_m$ acts on $\fg\times \fg^{reg}$ via natural
homotheties on both factors, and therefore on $\chi^*\Lie
J|_{\fg^{reg}}\simeq \Lie (I|_{\fg^{reg}})\subset \fg\times
\fg^{reg}$. This $\bG_m$-action on $\chi^*\Lie J|_{\fg^{reg}}$ descends to
a $\bG_m$-action on $\Lie J$, and for any line bundle $\mL$ on $C$
the $\mL^\times$-twist $(\Lie J)\times^{\bG_m}\mL^\times$ is $\Lie J_\mL\otimes \mL$, where $J_\mL$ is
introduced in \S \ref{univ cent}. In addition, both maps $\tau$ and
$c$ are $\bG_m$-equivariant with respect to this $\bG_m$-action on
$\Lie J$ and the natural $\bG_m$-action on $\fc$. Therefore, if we define a vector bundle
$B_{J,\mL}$ over $B_{\mL}$, whose fiber over $b\in B_{\mL}$ is
$\Gamma(C,\Lie J_b\otimes \mL)$, then by
twisting $\tau$ and $c$ by $\mL$, we obtain 
\begin{equation}\label{taut sect II}
\tau_{\mL}: B_{\mL}\to B_{J,\mL}.
\end{equation}
which is a canonical section of the  projection $\pr:B_{J,\mL}\to
B_{\mL}$, and a canonical map
\begin{equation}\label{c}
c_{\mL}:B_{J,\mL}\to B_{\mL}
\end{equation}
such that $c_{\mL}\tau_{\mL}=\id$.
As before, we omit the subscript ${_\mL}$ if $\mL=\omega$ for brevity.

\medskip

Likewise, we introduce the vector bundle $B_{J,\mL}^*$ over
$B_{\mL}$ whose fiber over $b$ is $\Gamma(C,(\Lie J_b)^*\otimes
\mL)$. Observe that $B_{J,\mL}^*$ is not the dual of $B_{J,\mL}$.
Rather, when $\mL=\omega$, it is the pullback $e^*T^*(\sP_{\mL}/B_{\mL})$ of the
cotangent bundle of $\sP_{\mL}\to B_{\mL}$ along the unit section
$e:B_{\mL}\to \sP_{\mL}$ and will also be denoted by
$\bbT_e^*(\sP_{\mL})$ interchangeably later on. We construct
a section
\begin{equation}\label{transpose II}
\tau^*_{\mL}: B_{\mL}\to B^*_{J,\mL}
\end{equation}
as follows.  The non-degenerate bilinear form $(\ ,\ )$ we fixed in
\ref{bilinear form} induces $\fg\simeq \fg^*$, which restricts to a
map $\Lie I_x\to (\Lie I_x)^*$ for every $x\in \fg^{reg}$. This map
descends to give 
\begin{equation}\label{transpose III}
\iota:\Lie J\to (\Lie J)^*,
\end{equation}
which is
$\bG_m$-equivariant. We define $\tau^*_{\mL}$ as the twist of
$\fc\stackrel{\tau}\to \Lie J\stackrel{\iota}\to (\Lie J)^*$. As before, we omit the subscript
$_{\mL}$ if $\mL=\omega$. 

We give another interpretation of this
map.
Observe that the Kostant section $\kappa$ induces the map
$$v_\kappa:\sP\to \on{Higgs}_G\to \Bun_G\times B$$
over $B$, and therefore,
$$\xymatrix{T^*(\Bun_G)\times_{\Bun_G}\sP\ar[r]^{\ \ \ \ \ \ (v_\kappa)_d}\ar[d]^{(v_\kappa)_p}&T^*(\sP/B)\\
T^*\Bun_G}.$$

\begin{lem}\label{iden with tau}
The map
$$\sP\stackrel{\kappa\times \id}{\to } T^*(\Bun_G)\times_{\Bun_G}\sP\stackrel{(v_\kappa)_d}{\to} T^*(\sP/B)\simeq \bbT^*_e\sP\times_B \sP,$$
can be identified with
\[\sP\stackrel{\on{pr}\times\id}{\to} B\times\sP\stackrel{\tau^*\times \id}{\longrightarrow}\bbT^*_e\sP\times_B \sP.\]
\end{lem}
\begin{proof}
For $b\in B$, we write the restriction of $v_\kappa$ over $b$ by
$v_{\kappa,b}:\sP_b\to \Bun_G$. We need to show that for
$x\in\sP_b$, the image of the point
\[\kappa(x)\in T^*_{v_{\kappa,b}(x)}\Bun_G\to T_x^*\sP_b\simeq (\bbT_e^*\sP)_b\]
coincides with $\tau^*(b)$. Let $E$ denote the $G$-bundle
$v_{k,b}(x)$.

Observe that there is a universal $G$-torsor $E_{univ}$ over
$[\fg/G]$ given by $\fg\to [\fg/G]$, and that
$\ad(E_{univ})\to[\fg/G]$ is canonically isomorphic to
$[\fg/G]\times_{BG}[\fg/G]\stackrel{\pr_1}{\to}[\fg/G]$. The
cotangent map
$$(v_{\kappa,b})_d: T^*_{v_{\kappa,b}(x)}\Bun_G\to T^*_{x}\sP_b$$
is induced by twisting $$kos^*(\ad(E_{univ}))^*\to (\Lie J)^*$$ by
the $(G\times\bbG_m)$-torsor $(E\times\omega^\times)$. Therefore, it
is enough to show that
$$\kappa(x)\in
T^*_{v_{\kappa,b}(x)}\Bun_G=\Gamma(C,\fg_E\otimes\omega)$$ can be
identified with the image of $b$ under $$\tau(b)\in\Gamma(C,\Lie
J_b\otimes\omega)\to \Gamma(C,\fg_E\otimes\omega).$$ Let us consider
the universal situation. Therefore, we need to show that
\[\frakc\stackrel{\tau}{\to}\Lie J\to kos^*\ad(E_{univ})\simeq \fc\times_{BG}[\fg/G]\]
is the same as $$\fc\stackrel{\id\times kos}{\to}
\fc\times_{BG}[\fg/G].$$ However,  the composition
$$[\fg/G]\stackrel{[\chi]^*(\tau)}{\to} [\chi]^*\Lie J\to
\ad(E_{univ})\simeq [\fg/G]\times_{BG}[\fg/G]$$ restricts to a map
$[\fg^{reg}/G]\to [\fg^{reg}/G]\times_{BG}[\fg/G]$, which is easily
checked to be the diagonal map using the definition of $\tau$. By
pulling back this identification along $kos:\fc\to[\fg^{reg}/G]$, we
obtain the claim.
\end{proof}

By the similar argument, we have the following lemma, which will be used in \S\ref{mult one forms}.
Let $\mathrm j^1:\boldsymbol\pi^*J\ra T\times{\ft}$ be the map in Proposition \ref{J},
and let 
\begin{equation}\label{dj^1}
d\mathrm j^1:\boldsymbol\pi^*\Lie J\ra\ft\times{\ft}
\end{equation}
denote its differential.
Consider the pull back $\boldsymbol\pi^*\tau:\ft\ra\pi^*\Lie J$
of $\tau:\fc\ra\Lie J$ along $\boldsymbol\pi:\ft\ra\fc$.

\begin{lemma}\label{iota}
The composition
\[
\delta:\ft\stackrel{\boldsymbol\pi^*\tau}\ra\boldsymbol\pi^*\Lie J_{}\stackrel{d\mathrm j^1}\ra\ft\times\ft
\]
is equal to the diagonal map $\Delta:\ft\ra\ft\times\ft$.
\end{lemma}
\begin{proof}
For an embedding $\ft\subset \fg$, the restriction of $d\mathrm j^1$ to $\ft^{reg}=\ft\cap \fg^{reg}$ is just the restriction to $\ft^{reg}$ of the isomorphism $\Lie J|_{\fg^{reg}}\simeq \Lie( I|_{\fg^{reg}})$. This follows from the construction of $\mathrm j^1$ as in \cite[Proposition 2.4.2]{N2}.
Therefore, the restriction of $\delta$ to $\ft^{reg}$ is just the diagonal map. The lemma then follows.
\end{proof}


\section{Classical duality}\label{classical duality}
In this section, we fix a smooth projective curve $C$ over $k$ and a line bundle $\mL$ on $C$ such that $\deg \mL>0$. Except \S \ref{l_J}, we also fix a connected reductive group $G$ over $k$. We assume that $p=\cha\ k$ does not divide the order of the Weyl group of $G$. 
We show that the $\breve\sP_\mL\simeq \sP_\mL^\vee$ as
Picard stacks over $B^0$. Note that this duality for $k=\mathbb C$
is the main theorem of \cite{DP} (for $G=\SL_n$, see \cite{HT}). However, as mentioned by the
authors, transcendental arguments are used in \emph{loc. cit.} in an
essential way, and therefore cannot be applied directly to our
situation. Our argument works for any algebraically closed field $k$
of characteristic zero or $p$ with $p\nmid |\rW|$.

In fact, it is not hard to construct a canonical isogeny
$\frakD_{cl}$ between $\breve \sP_\mL$ and $\sP_\mL^\vee$. If the adjoint
group of $G$ does not contain a simple factor of type $\mathrm B$ or $\mathrm C$,
then to show that $\frakD_{cl}$ is an isomorphism is relatively
easy. It is to show that $\frakD_{cl}$ is an isomorphism in the
remaining cases that some complicated calculations are needed.

Observe in this section, we do not need to assume that
$\mL=\omega_C$. We only need the assumption that $\text{deg}\ \mL$ is positive.
However, to simplify the notations, we still omit
the subscript $_{\mL}$.


\subsection{Galois description of $\scP$}\label{Gal des}

We first introduce several auxiliary Picard stacks.

Let $\tC\to B$ be the universal cameral curve. There is a natural
action of $\rW$ on $\tC$. For a $T$-torsor $E_T$ on $\tC$, and an
element $w\in \rW$, there are two ways to produce a new $T$-torsor.
Namely, the first is via the pullback $w^*E_T=\tC \times{_{w,\tC}}
E_T$, and the second is via the induction $E_T\times^{T,w}T$. We
denote
\[w(E_T)=  ((w^{-1})^*E_T)\times^{T,w}T.\]
Clearly, the assignment $E_T\mapsto w(E_T)$ defines an action of $W$
on $\Bun_T(\tC/B)$, i.e. for every $w,w'\in \rW$, there is a
canonical isomorphism $w(w'(E_T))\simeq (ww')(E_T)$ satisfying the
usual cocycle conditions. 
\begin{example}
Let us describe $w(E_T)$ more explicitly in the case $G=\SL_2$. Let
$s$ be the unique nontrivial element in the Weyl group, acting on
the cameral curve $s:\tC_b\to \tC_b$. If we identify
$T=\bbG_m$-torsors with invertible sheaves $\mL$, then
\[s(\mL)=s^*\mL^{-1}.\]
\end{example}
Let $\Bun_{T}^\rW(\tC/B)$ (or $\Bun_T^\rW$ for simplicity) denote the
Picard stack of strongly $\rW$-equivariant $T$-torsors on $\tC/B$.
By definition, for a $B$-scheme $S$, $\Bun_{T}^\rW(\tC/B)(S)$ is the
groupoid of $(E_T, \{\gamma_w, w\in W\})$, where $E_T$ is a $T$-torsor
on $\tC_S$, and $\gamma_w: w(E_T)\simeq E_T$ is an isomorphism,
satisfying the natural compatibility conditions. Another way to
formulate these compatibility conditions is provided in \cite{DG}.
Namely, for a $T$-torsor $E_T$, let $\Aut_\rW(E_T)$ be the group
consisting of $(w,\gamma_w)$, where $w\in \rW$ and
$\gamma_w:w(E_T)\simeq E_T$ is an isomorphism. Then there is a
natural projection $\Aut_\rW(E_T)\to \rW$. Then an object of
$\Bun_T^\rW(\tC/B)(S)$ is a pair $(E_T,\gamma)$, where
$\gamma:\rW\to \Aut_\rW(E_T)$ is a splitting of the projection. 

For later purpose, it is worthwhile to give another description of $\Bun_T^\rW$. Namely, there is a non-constant group scheme $\mT=\tC\times^\rW T$ on the stack $[\tC/\rW]$. Then the pullback functor induces an isomorphism from the stack $\Bun_{\mT}$ of $\mT$-torsors on $[\tC/W]$ to $\Bun_T^\rW$.

In \cite{DG}, a Galois description of $\scP$ in terms of
$\Bun_T^\rW$ is given. We here refine their description.

Let $\scP^1$ be the Picard stack over $B$ classifying $J^1$-torsors
on $C\times B$. First, we claim that there is a canonical morphism
\begin{equation}\label{p=p'}
j^{1,\scP}:\scP^1\to \Bun_T^{\rW}(\tC/B).
\end{equation}
To construct $j^{1,\sP}$, recall that $J^1=(\pi_*(T\times\tC))^W$, where $\pi:\tC\to C\times B$ is the projection, and
therefore, for any $J^1$-torsor $E_{J^1}$ on $C\times S$ (where
$b:S\to B$ is a test scheme), one can form a $T$-torsor on $\tC_S$
by 
\begin{equation}\label{ET ind}
E_T:=\pi^*E_{J^1}\times^{\pi^*J^1}T.
\end{equation} 
Clearly, $E_T$ carries on
a strongly $\rW$-equivariant structure $\gamma$, and $j^1(E_{J^1})=
(E_T,\gamma)$ defines the morphism $j^{1,\sP}$.

The morphism $j^{1,\sP}$, in general, is not an isomorphism. Let us
describe the image. Let $\alpha\in\Phi$ be a root and let
$i_\alpha:\tC_\alpha\to \tC$ be the inclusion of the fixed point
subscheme of the reflection $s_\alpha$. Let
$T_{\alpha}=T/(s_\alpha-1)$ be the torus of coinvariants of the
reflection $s_\alpha$. Then
$s_\alpha(E_T)|_{\tC_\alpha}\times^TT_\alpha$ is canonically
isomorphic to $E_T|_{\tC_\alpha}\times^TT_\alpha$ and therefore
$\gamma_{s_\alpha}|_{\tC_\alpha}$ induces an automorphism of the
$T_\alpha$-torsor $E_T\times^TT_\alpha$. In other words, there is a
natural map
\[r=\prod_{\alpha\in\Phi} r_\alpha:\Bun_T^\rW(\tC/B)\to (\prod_{\alpha\in\Phi} \Res_{\tC_\alpha/B}(T_\alpha\times\tC_\alpha))^\rW.\]
It is easy to see that $r\circ j^{1,\sP}$ is trivial, and one can show that

\begin{lem}\label{scP1 and strongly equivariant}
$\scP^1\simeq \ker r$. In other words, $\scP^1(S)$ consists of those
strongly $\rW$-equivariant $T$-torsors $(E_T,\gamma)$ such that the
induced automorphism of $E_T\times^TT_\alpha|_{\tC_\alpha}$ is
trivial for every $\alpha\in \Phi$.
\end{lem}
\begin{proof} We shall
show that every strongly $\rW$-equivariant $T$-torsor
$(E_T,\gamma)$ such that $r(E_T,\gamma)=1$ is Zariski locally on
$\tC$ isomorphic to the trivial one, i.e., the trivial $T$-torsor
together with the canonical $\rW$-equivariance structure. If this is
the case, then the inverse map from $\ker r\to \scP^1$ is given as
follows. For every strongly $\rW$-equivariant $T$-torsor
$(E_T,\gamma)$, $\pi_*E_T$ carries on an action of $\rW$. Namely,
let $x:S\to C$ be a point and $m: S\times_C\tC_b\to E_T$ be a point
of $\pi_*E_T$ over $x$. Then $w(m)$ is the point of $\pi_*E_T$ over
$x$ given by
\[S\times_C\tC_b\stackrel{1\times w^{-1}}{\to} S\times _C\tC_b \stackrel{w^{-1}(m)}{\to} (w^{-1})^*E_T\to w(E_T)\stackrel{s(w)}{\to} E_T.\]
This $\rW$-action on $\pi_*E_T$ is compatible with the action of
$\pi_*(T\times\tC)$ in the sense that $w(mt)= w(m)w(t)$. Now let
$E_{J^1}=(\pi_*E_T)^W$, then as $(E,\gamma)$ is locally isomorphic
to the trivial one, $E_{J^1}$ is locally isomorphic to $J^1$, and
therefore is a $J^1$-torsor on $C$.

To prove the local triviality, we follow the argument as in
\cite[Proposition 16.4]{DG}. One reduces to prove the statement for
a neighborhood around a point $x\in \cap_\alpha \tC_\alpha$. By
replacing $\tC$ by the local ring around $x$, one can assume that
$E_T$ is trivial. Pick up a trivialization, then the
$\rW$-equivariance structure on $E_T$ amounts to a 1-cocycle $W\to
T(\tC)$. By evaluating $T(\tC)$ at the unique closed point $x$,
there is a short exact sequence $1\to K\to T(\tC)\to T(k)\to 1$. The
condition $r(E_T,\gamma)=1$ would mean that the cocycle takes value
in $K$. Since there exists a filtration on $K$, such that the associated graded is an $\bbF_p$-vector space and $p\nmid |\rW|$, this
cocycle is trivial.
\end{proof}

\subsection{}\label{+} Recall that in \cite{DG,N1}, an open embedding
$J\to J^1$ is constructed. To describe the cokernel, we need some
notations. Let $\breve\alpha\in\breve\Phi$ be a coroot. Let
$$\mu_{\breve\alpha}:=\ker(\breve\alpha:\bbG_m\to T).$$ This is either
trivial, or $\mu_2$, depending on whether $\breve\alpha$ is
primitive or not. Let $\mu_{\breve\alpha}\times\tC_\alpha$ be the
constant group scheme over $\tC_\alpha$, regarded as a sheaf of
groups over $\tC_\alpha$, and let
$(i_\alpha)_*(\mu_{\breve\alpha}\times \tC_\alpha)$ be its push
forward to $\tC$. Now, the result of \cite[\S 11, 12]{DG} can be reformulated
as follows: there is a natural exact sequence of sheaves of groups on $C$.
\begin{equation}\label{J and J1}
1\to J\to J^1\to
\pi_*(\bigoplus_{\alpha\in\Phi}(i_\alpha)_*(\mu_{\breve\alpha}\times
\tC_\alpha))^\rW\to 1.
\end{equation}

As a result, we obtain a short exact sequence of Picard stacks 
(see \S\ref{short exact seq})
\begin{equation}\label{short exact}
1\to
(\prod_{\alpha\in\Phi}\Res_{\tC_\alpha/B}(\mu_{\breve\alpha}\times
\tC_\alpha))^\rW\to \scP\to \scP^1\to 1.
\end{equation}

To simplify the notation, we will denote
$\Res_{\tC_\alpha/B}(\mu_{\breve\alpha}\times \tC_\alpha)$ by
$\mu_{\breve\alpha}(\tC_\alpha)$ in the sequel.

Consider the composition
$$
j:\scP\to \scP^1\to \Bun_T^\rW(\tC/B).
$$
Combining Lemma \ref{scP1 and strongly equivariant} and
\eqref{short exact}, we recover a description of $\scP$ in terms of
$\Bun_T^\rW(\tC/B)$ as given in \cite[\S 16.3]{DG}. Namely, given a strongly
$\rW$-equivariant $T$-torsor $(E_T,\gamma)$, one obtains a canonical
trivialization
\begin{equation}\label{can triv}
E_T^{\breve\alpha\circ\alpha}:=(E_T|_{\tC_\alpha})\times^{T,\alpha}\bbG_m\times^{\bbG_m,\breve\alpha}T\simeq
E_T^0|_{\tC_{\alpha}},
\end{equation}
given by
$(E_T|_{\tC_\alpha})\times^{T,\alpha}\bbG_m\times^{\bbG_m,\breve\alpha}T\simeq
E_T|_{\tC_\alpha}\otimes s_\alpha(E_T^{-1})|_{\tC_\alpha}$. The
condition that $r_\alpha(E,\gamma)=1$ is equivalent to the
condition that \eqref{can triv} comes from a trivialization
\begin{equation}\label{calpha}
c_\alpha:E_T^{\alpha}:=(E_T|_{\tC_\alpha})\times^{T,\alpha}\bbG_m\simeq
\bbG_m\times\tC_\alpha.
\end{equation}
In addition, the set of all such $c_\alpha$
form a $\mu_{\breve\alpha}$-torsor. 
Consider the following Picard stack $\Bun_T^\rW(\tC/B)^+$:
For any $B$-scheme $S$, its $S$-points
form the 
Picard groupoid of triples
\begin{equation}\label{+ structure}
\Bun_T^\rW(\tC_S)^+:=(E_T,\gamma,c_\alpha, \alpha\in\Phi),
\end{equation}
where $(E_T,\gamma)$ is a strongly $\rW$-equivariant $T$-torsor on
$\tC_S$, and
$c_\alpha:(E_T|_{\tC_\alpha})\times^{T,\alpha}\bbG_m\simeq
\bbG_m\times\tC_\alpha$ is a trivialization, which induces
\eqref{can triv} and is compatible with the $\rW$-equivariant
structure. We call those trivializations
$\{c_\alpha\}_{\alpha\in\Phi}$ a $+$-structure on $(E_T,\gamma)$.
Note that, by Lemma \ref{scP1 and strongly equivariant}, we have the following short exact sequence 
of Picard stacks
\begin{equation}\label{short exact 2}
1\ra(\prod_{\alpha\in\Phi}\Res_{\tC_\alpha/B}(\mu_{\breve\alpha}\times
\tC_\alpha))^\rW\ra\Bun_T^\rW(\tC/B)^+\ra\sP^1\ra 1.
\end{equation}

\begin{lemma}\cite[Proposition 16.4]{DG}
We have 
$\scP\is \Bun_T^\rW(\tC/B)^+$.
\end{lemma}
\begin{proof}
Indeed, the exact sequence (\ref{J and J1}) implies that, for any 
$J$-torsor $E_J\in\sP$
the image $j(E_J)\in\Bun_T^\rW(\tC/B)$ 
carries a canonical $+$-structure. This defines a morphism  
$\sP\ra \Bun_T^\rW(\tC/B)^+$ and one can check that it is compatible with 
the short exact sequences (\ref{short exact}) and (\ref{short exact 2}). The lemma follows.
\end{proof}

Here is an application of the above discussion. Observe there is the norm map
\[\Nm: \Bun_T(\tC/B)\to \Bun_T^\rW(\tC/B), \quad E_T\mapsto (\bigotimes_{w\in W} w(E_T),\gamma_{\rm can}).\]
We claim that $\Nm$ admits a canonical lifting
\begin{equation}\label{Pnorm}
\Nm^\scP:\Bun_T(\tC/B)\to \scP.
\end{equation}
To show this, we need to exhibit a canonical trivialization
\[c_\alpha:\bigotimes_{w\in \rW}w(E_T)|_{\tC_\alpha}\times^{T,\alpha}\bbG_m\simeq \bbG_m\times\tC_\alpha.\]
compatible with the strongly $\rW$-equivariant structure. However,
for any $T$-torsor $E_T$, there is a canonical isomorphism
$(E_T|_{\tC_\alpha}\otimes
s_\alpha(E_T)|_{\tC_\alpha})\times^{T,\alpha}\bbG_m\simeq
\bbG_m\times\tC_\alpha$, and therefore, we obtain $c_\alpha$ by
writing
$$\bigotimes_{w\in
W}w(E_T)|_{\tC_\alpha}\times^{T,\alpha}\bbG_m\simeq \bigotimes_{w\in
s_\alpha\setminus W}(w(E_T)|_{\tC_\alpha}\otimes s_\alpha
w(E_T)|_{\tC_\alpha})\times^{T,\alpha}\bbG_m,
$$
where $s_\alpha\setminus W$ denotes the quotient of $W$ by the subgroup 
generated by $s_\alpha$.
The compatibility of the collection $\{c_\alpha\}$ with the
$\rW$-equivariant structure is clear.

\subsection{Galois description of $\sP$-torsors}\label{Galois des for torsors}
The above description of $\sP$ in terms of $\Bun_{T}^{\rW}(\tC/B)$ can be generalized as follows. 
Let $\sD$ be a $J$-gerbe on $C\times B$. Similar to \eqref{ET ind}, we define $$\sD_T:=(\pi^*\sD)^{j^1}$$
as the $T$-gerbe on $\tC$ induced from $\sD$ using maps $\pi:\tC\ra C\times B$ and  
$j^1:\pi^*J\ra T\times\tC$ (see \ref{mult torsor} and \ref{induction functor} for the notion of gerbes and functors between them). Since the map 
$j^1$ is $\rW$-equivariant the gerbe $\sD_T$ is strongly $\rW$-equivariant.
Equivalently, this means that $\sD_T$ descends to a $\mT$-gerbe on $[\tC/\rW]$.

Let $\sT_\sD$ be the stack of splittings of $\sD$ over $B$. By definition, for every $S\to B$, $\sT_\sD(S)$ is the groupoid of the splittings of the gerbe $\sD|_{C\times S}$. This is a (pseudo) $\sP$-torsor. 
On the other hand, let
$\sT_{\sD_T}^\rW$ denote the stack of strongly $\rW$-equivariant splittings of $\sD_T$, i.e., $\sT_{\sD_T}^{\rW}(S)$ is the groupoid of the splittings of $\sD_T|_{[\tC/\rW]\times_BS}$. Our goal is to give a description of $\sT_{\sD}$ in terms of $\sT_{\sD_T}^\rW$.

Let $\alpha\in\Phi$. Similar to $E_T^\alpha$ and $E_T^{\breve\alpha\circ\alpha}$ as defined in \eqref{can triv} and \eqref{calpha},
let $\sD_T^\alpha$, $\sD_T^{\breve\alpha\circ\alpha}$ denote the restrictions to $\tC_\alpha$ of the
$\bG_m$- and $T$-gerbes on $\tC$ induced from $\sD_T$ using the maps $\alpha:T\to \bG_m$ and $\breve\alpha\circ\alpha:T\to T$ respectively.
The strongly $\rW$-equivariant structure on $\sD_T$ implies that the $T$-gerbe $\sD_T^{\breve\alpha\circ\alpha}$ has a canonical splitting $F_\alpha^0$.  
Moreover, by a similar argument in \S\ref{+}, one can show that: (i) there is a 
canonical splitting $E_\alpha^0$ of the $\bG_m$-gerbe $\sD_T^\alpha$, which induces $F_\alpha^0$ via the canonical map
$\sD_T^\alpha\ra\sD_T^{\breve\alpha\circ\alpha}$ and (ii) for any strongly 
$\rW$-equivariant splitting $(E,\gamma)$ of 
$\sD_T$ there is a canonical isomorphism of splittings
\begin{equation}\label{can}
E^{\breve\alpha\circ\alpha}|_{\tC_\alpha}\is F_\alpha^0,
\end{equation}
where $E^{\breve\alpha\circ\alpha}$ is the splitting of 
$\sD_T^{\breve\alpha\circ\alpha}$
induces by $E$ via the canonical map $\sD_T^\alpha\ra\sD_T^{\breve\alpha\circ\alpha}$.
We define $\sT^{\rW,+}_{\sD_T}$ as the stack over $B$ whose $S$-points consist of 
\[\sT^{\rW,+}_{\sD_T}(S):=(E,\gamma,t_\alpha,\alpha\in\Phi),\] where 
$(E,\gamma)$ is a strongly $\rW$-equivariant  splittings of $\sD_T$
and \[t_\alpha:E^\alpha|_{\tC_{\alpha}}\is E_{\alpha}^0\] is an isomorphism of 
splittings of $\sD_T^\alpha$, which induces (\ref{can}) and is 
compatible with the $\rW$-equivariant structure. 
It is clear that $\sT^{\rW,+}_{\sD_T}$ is a $\sP=\Bun_T^\rW(\tC/B)^+$-torsor.
\begin{lemma}\label{T=T^+}
There is a canonial isomorphism of $\sP$-torsors $\sT_\sD\is\sT^{\rW,+}_{\sD_T}$.
\end{lemma}

\begin{proof}
Let $E\in\sT_\sD$ be a splitting of $\sD$. Then $E_T:=(\pi^*(E))^{j^1}$ defines 
a splitting of $\sD_T$. Since both maps $j^1$ and $\pi$ are 
$\rW$-equivariant the splitting $E_T$ has a canonial $\rW$-equivariant structure,
which we denote by $\gamma$. Moreover,
by the same reasoning as in \S\ref{+}, there is a canonical isomorphism of 
splittings
 $t_\alpha:E_T^\alpha|_{\tC_\alpha}\is E_\alpha^0$
such that the induced isomorphism 
$E_T^{\breve\alpha\circ\alpha}|_{\tC_\alpha}
\stackrel{}\is (E_\alpha^0)^\alpha\is F_\alpha^0$
is equal to the one coming form the $\rW$-equivariant structure $\gamma$.
The assignment $E\ra (E_T,\gamma,t_\alpha, \alpha\in\Phi)$ defines a 
morphism $\sT_\sD\ra\sT_{\sD_T}^{\rW,+}$, which
is compatible with their $\sP$-torsor structures and hence is an isomorphism.
\end{proof}

\subsection{The Abel-Jacobi map}\label{AJ map}
From now on till the end of this section, we restrict to the open
subset $B^0$ of the Hitchin base. To simplify the notations,
we use $B$ to denote $B^0$ unless specified. Recall from Proposition \ref{reg torsor} that the cameral curve $\tC$ is smooth over $B^0$.

Let
\[\AJ: \tC\times \xcoch(T)\to \Bun_T(\tC/B)\]
be the Abel-Jacobi map given by $(x,\breve\lambda)\mapsto
\calO(\breve\lambda x):=\calO(x)\times^{\bbG_m,\breve\lambda}T$.
By composition with  $\Nm^\sP$, we obtain a morphism
\[\AJ^\scP:\tC\times\xcoch(T)\to \scP.\]
It is $\rW$-equivariant, where $\rW$ acts on $\tC\times\xcoch(T)$ diagonally and on $\sP$ trivially, and is commutative and multiplicative with respect to the group structures on $\xcoch(T)$ and on $\sP$.
Observe that for any $x\in \tC_\alpha$, $\AJ^\scP(x,\breve\alpha)$
is the unit in $\scP$. This follows from
\[\bigotimes_{w\in W}w \calO(\breve\alpha
x)\simeq \bigotimes_{w\in W/s_\alpha} w\calO(\breve\alpha x+
s_\alpha(\breve\alpha) x)\] is canonically trivialized, and the
trivialization is compatible with the $\rW$-equivariant structure.
Here as before $W/s_\alpha$ is the quotient of $W$ by the subgroup 
generated by $s_\alpha$.

By pulling back the line bundles, we thus obtain
\[(\AJ^\scP)^\vee: \scP^\vee\to \Pic^m(\tC\times\xcoch(T)
)^\rW,\] 
where $\Pic^m(\tC\times\xcoch(T))^\rW$ denotes the Picard stack over $B$ of $\rW$-equivariant line bundles on $\tC\times\xcoch(T)$ which are multiplicative with respect to $\xcoch(T)$.
Observe that there is the canonical isomorphism
$\Bun_{\breve{T}}^\rW(\tC/B)\to \Pic^m(\tC\times\xcoch(T))^\rW$ given
by $(E_{\breve T},\gamma)\mapsto \mL$, where
$\mL|_{(x,\breve\lambda)}=E_{\breve T}^\lambda|_x$. Therefore, we can regard $(\AJ^\scP)^\vee$ as a morphism
\[(\AJ^\scP)^\vee:\scP^\vee\to \Bun_{\breve{T}}^\rW(\tC/B).\]
We claim that $(\AJ^\scP)^\vee$ canonically lifts to a morphism
\[\frakD_{cl}:\scP^\vee\to\breve\scP.\]

Let $\mL$ be a multiplicative line bundle on $\scP$. We thus need to
show that
\[ (\AJ^\scP)^*\mL|_{(\tC_\alpha,\breve\alpha)}\]
admits a canonical trivialization, which is compatible with the
$\rW$-equivariance structure. However, this follows from
$\AJ^\scP((x,\breve\alpha))$ is the unit of $\scP$ and a
multiplicative line bundle on $\scP$ is canonically trivialized over
the unit. To summarize, we have constructed the following
commutative diagram
\begin{equation}\xymatrix{
\scP^\vee\ar^{\frakD_{cl}}[rr]\ar_{(\AJ^{\scP})^\vee}[dr]&&\breve\scP\ar^{\breve{j}}[dl]\\
&\Bun_{\breve T}^\rW }\end{equation}

Now, the classical duality theorem reads as
\begin{thm}\label{classical limit}
$\frakD_{cl}$ is an isomorphism.
\end{thm}

The proof of this theorem occupies \S \ref{first reduction}- \S\ref{cal of finite} below.

\subsection{First reductions}\label{first reduction}
We first show that $\frakD_{cl}$ induces an isomorphism
\[\pi_0(\frakD_{cl}):\pi_0(\scP^\vee)\to \pi_0(\breve\scP).\]

For any $S$-point  $b\in B^0$, $\sP_b$ is a Beilinson 1-motive (Appendix \ref{B}). We have
\[\underline{\Aut}(e)\simeq{\rm H}^0(C,J_b),\quad\pi_0(\scP_b)=\scP_b/ W_1\scP_b. \]

Observe that
\[{\rm H}^0(C,J_b)\simeq \ker (T^\rW\to (\prod_{\alpha\in\Phi}\Res_{\tC_\alpha/b}(\mu_{\breve\alpha}\times\tC_\alpha))^\rW)=Z(G).\]
By Corollary
\ref{aut-pi0 dual}
\[\pi_0(\scP^\vee)\simeq (\Aut_\scP(e))^*.\]

Let us also recall the description of $\pi_0(\scP)$ as given in
\cite[\S 4.10, \S 5.5]{N2}. As we restrict $\scP$ to $B^0$, the
answer is very simple. Namely, the Abel-Jacobi map
\[\AJ^\scP: \tC\times\xcoch(T)\to \scP\]
induces a surjective map
\[\pi_0(\tC\times\xcoch(T))\simeq \xcoch(T)\twoheadrightarrow \pi_0(\scP).\]
which induces
\[\pi_0(\scP)^*\simeq Z(\breve{G})\subset \breve{T}^{\rW}.\]
Therefore, as abstract groups $\pi_0(\scP^\vee)\simeq
\pi_0(\breve\scP)$.

Since $\pi_0(\scP^\vee)\simeq \pi_0(\breve\scP)$ are finitely
generated abelian groups and are isomorphic abstractly, to show that
$\pi_0(\frakD_{cl})$ is an isomorphism, it is enough to show that
\begin{lem}
The induced map $\pi_0(\frakD_{cl})$ is surjective.
\end{lem}
\begin{proof}According to the above description, it is enough to
construct a morphism $\tC\times\xch(T)\to\scP^\vee$ making the
following diagram is commutative.
\[\xymatrix{
&\tC\times\xch(T)\ar[dl]\ar^{\AJ^{\breve\scP}}[dr] & \\
 \scP^\vee\ar^{\frakD_{cl}}[rr]&  & \breve\scP.}\]
To this goal, observe that there is the universal line bundle
$\sL_{\rm univ}$ on $(\tC\times\xch(T))\times\Bun_T$. Then the
pullback of this line bundle to $(\tC\times \xch(T))\times\scP$
gives rise to the desired map. The commutativity of this diagram is
an easy exercise.
\end{proof}

Next, we see that
\[W_0(\frakD_{cl}): W_0\scP^\vee\to W_0\breve\scP\]
is an isomorphism. Indeed, we can construct $\AJ^{\breve\scP}:
\tC\times\xch(T)\to \breve\scP$, and therefore $\breve\frakD_{cl}:
\breve\scP^\vee\to \scP$. By the same argument, it induces an
isomorphism $\pi_0(\breve\frakD_{cl}):\pi_0(\breve\scP^\vee)\to
\pi_0(\scP)$. It is easy to check that
$\breve\frakD_{cl}=\frakD_{cl}^\vee$, and therefore
$W_0(\frakD_{cl})$ is also an isomorphism.

Therefore, it is enough to show that $D_{cl}:P^\vee\to \breve P$ is
an isomorphism, where $P$ (resp. $\breve P$) is the neutral
connected component of the coarse moduli space of $\scP$ (resp. $\breve
\scP$), and $D_{cl}$ is the map induced by $\frakD_{cl}$. We can
prove this fiberwise, and therefore we fix $b\in B(k)$. However, to
simplify the notation, in the following discussion we write $\tC,
\scP$ instead of $\tC_b,\scP_b$, etc.

\subsection{The calculation of the coarse moduli}
We introduce a few more notations. Let $\scP^0$ be the Picard stack
of $J^0$-torsors on $C$, and let $P^0$ (resp. $P^1$) be the neutral
connected components of the coarse moduli space of $\scP^0$ (resp.
$\scP^1$).

We first understand $P^1$. Let $\on{Jac}$ denote the Jacobi variety
of $\tC$. Then $\on{Jac}\otimes \xcoch$ is the neutral connected
component of the coarse moduli space of $\Bun_T$.

\begin{lem}\label{P1}
The map $P^1\to \on{Jac}\otimes\xcoch$ is an embedding, and $P^1$
can be identified with $(\on{Jac}\otimes\xcoch)^{\rW,0}$, the
neutral connected component of the $\rW$-fixed point subscheme of
$\on{Jac}\otimes\xcoch$.
\end{lem}
\begin{proof}
We first show that $P^1\to \on{Jac}\otimes\xcoch$ is injective at the level of $k$-points.
Indeed, up to isomorphism, the strongly $\rW$-equivariant structures
on a trivializable $T$-torsor on $\tC$ are classified by
$\on{H}^1(\rW,T(k))$. By Lemma \ref{scP1 and strongly equivariant},
the kernel of $P^1\to \on{Jac}\otimes\xcoch$ can be identified with
the kernel of the natural map
\[\on{H}^1(\rW,T(k))\to\bigoplus_{\tC_\alpha}T_\alpha(\tC_\alpha).\]
Therefore, it is enough to show that this latter map is
injective. Over $B^0$, $\tC_\alpha$ is nonempty for every root
$\alpha$.\footnote{Indeed, the same argument in Remark \ref{transversal locus}, with the discriminant function
$\underline{Disc}$ replaced by the $\rW$-invariant function 
$\prod_{\beta\in \rW\alpha}d\beta:\ft\ra k$,
shows that the fixed point $\tC_\alpha$ is nonempty.
} Then the injectivity is a consequence of the following lemma applied to $M=T(k)$.
\begin{lem}\label{L:H1W}
Let $M$ be a $\rW$-module satisfying the following condition: for some (and therefore any) choice of a set of simple roots $\{\al_1,\ldots,\al_l\}$, the natural map
\[M\to \prod_{i=1}^l (1-s_{\al_i})M,\quad m\mapsto ((1-s_{\al_1})m,\cdots,(1-s_{\al_l})m)\]
is surjective.
Then the natural map
\[{\rm H}^1(W,M)\to \prod_{1\leq i\leq s} M/(1-s_{\beta_i})M, \quad [c]\mapsto \prod_{1\leq i\leq s} (c(s_{\beta_i}) \mod (1-s_{\beta_i})M)\]
is injective for any choice of a set $\{\beta_1,\ldots,\beta_s\}\subset \Phi$ of representatives of $\Phi/\rW$.
\end{lem}
\begin{proof}
Let $c: W\to M$ be a cocycle. It follows from the cocycle condition that if $c(s_{\beta_i})\in (1-s_{\beta_i})M$, then $c(s_{w(\beta_i)})\in (1-s_{w(\beta_i)})M$. Therefore, a class $[c]$ is in the kernel of the map in the lemma only if $c(s_{\al_i})\in (1-s_{\al_i})M$ for a set of simple roots $\{\al_1,\ldots,\al_l\}$. But by our assumption of $M$, there exists $m\in M$ such that $c(s_{\al_i})=(1-s_{\al_i})m$ for all $1\leq i\leq l$. Then using the cocycle condition, one can show by induction on the length of $w$ that $c(w)=(1-w)m$ for every $w\in \rW$. But this means that $c$ is a coboundary. 
\end{proof}

To complete the proof, observe that the restriction of the norm map
\[\Nm: \on{Jac}\otimes \xcoch\to P\to P^1\to (\on{Jac}\otimes \xcoch)^\rW\]
to $\Nm:(\on{Jac}\otimes \xcoch)^\rW\to (\on{Jac}\otimes
\xcoch)^\rW$ is the multiplication by $|\rW|$. Therefore, the image
of $P^1\to \on{Jac}\otimes\xcoch$ is
$(\on{Jac}\otimes\xcoch)^{\rW,0}$. In addition, $P^1\to (\on{Jac}\otimes\xcoch)^{\rW,0}$ is a prime-to-$p$ isogeny and therefore its kernel is \'etale. Then this kernel must be trivial since its underlying group of $k$-points is trivial.
\end{proof}

As a result, for any prime $\ell\neq p$,
\[T_\ell P^1\simeq ({\rm H}^1(\tC,\bbZ_\ell(1))\otimes\xcoch)^\rW.\]
In addition, observe that from the definition of $D_{cl}$, the map
$P^1\subset \on{Jac}\otimes\xcoch\stackrel{\Nm}{\to} P^1$ factors as
\[P^1\subset \on{Jac}\otimes\xcoch\simeq (\on{Jac}\otimes\xch)^\vee\to (\breve P^1)^\vee\to \breve P^\vee\stackrel{\breve D_{cl}}{\to} P\to P^1.\]
Therefore $D_{cl}$ is a prime-to-$p$ isogeny. In addition, the map
$$
T_\ell\Nm:T_\ell(\on{Jac}\otimes\xcoch)\twoheadrightarrow
T_\ell(\breve P^1)^\vee\hookrightarrow T_\ell P^1
$$
can be identified with
\[\Nm:{\rm H}^1(\tC,\bbZ_\ell(1))\otimes\xcoch\twoheadrightarrow ({\rm H}^1(\tC,\bbZ_\ell(1))\otimes\xcoch)_{\rW}/(\on{torsion})\hookrightarrow ({\rm H}^1(\tC,\bbZ_\ell(1))\otimes\xcoch)^{\rW}.\]

On the other hand, as $J^0$ is connected, the norm map $\Nm: \pi_*T\to J^1=(\pi_*
T)^\rW$ factors as $\pi_*T\to J^0\to
J^1$. Therefore, $\Nm: \on{Jac}\otimes\xcoch\to P^1$ also factors as
$$\Nm:\on{Jac}\otimes\xcoch\to P^0\to P^1.$$ 
It follows that $P^0\to P^1$ is also
a prime-to-$p$ isogeny, and for $\ell\neq p$ there is a factorization
\[\Nm:{\rm H}^1(\tC,\bbZ_\ell(1))\otimes\xcoch\to T_\ell P^0\hookrightarrow ({\rm H}^1(\tC,\bbZ_\ell(1))\otimes\xcoch)^{\rW}.\] 

We need the following key result.
\begin{prop}The two isogenies $(\breve P^1)^\vee\to P^1\leftarrow P^0$ induce an isomorphism $(\breve P^1)^\vee\simeq P^0$.
\end{prop}
\begin{proof}By the above considerations, the lemma is equivalent to
saying that the induced map of Tate modules
$T_\ell\Nm: T_\ell(\on{Jac}\otimes\xcoch)\to T_\ell P^0$ is
surjective for every $\ell\neq p$.

Note that we have the following commutative diagram
\[\xymatrix{
(\on{Jac}\otimes\xcoch)[\ell^n]={\rm H^1}(C,\pi_*(\xcoch\otimes\mu_{\ell^n}))\ar[r]\ar[d]&\ar[d] P^0[\ell^n]\ar[d]\\
{\rm H}^1(C, J^0[\ell^n])\ar@{->>}[r]& {\rm H}^1(C, J^0)[\ell^n],
}\]
where the left vertical arrow is induced by $\pi_*T[\ell^n]=\pi_*(\xcoch\otimes\mu_{\ell^n})\to J^0[\ell^n]$, and the bottom row is induced by the Kummer sequence for $J^0$ and therefore is surjective.
Since $\pi_0(\sP^0)= H^1(C,J^0)/P^0$ is finitely generated,  passing to the inverse limit gives 
\[\xymatrix{
\underleftarrow\lim_n{\rm H^1}(C,\pi_*(\xcoch\otimes\mu_{\ell^n}))\ar[r]\ar[d]&\ar[d] T_\ell P^0\ar^{\simeq}[d]\\
\underleftarrow\lim_n{\rm H}^1(C, J^0[\ell^n])\ar@{->>}[r]& \underleftarrow\lim_n{\rm H}^1(C, J^0)[\ell^n],
}\]
where the bottom arrow is surjective.
So it is enough to show that the left vertical arrow is also surjective.

Let $y\in C$, and choose a point $\tilde y\in\tC$ lying over $y$. Let $\rW_{\tilde y}\subset \rW$ denote the stabilizer of $\tilde y$ under the action of $\rW$ on $\tC$. Note that $\rW_{\tilde y}=\langle s_\alpha\rangle$ if $\tilde y\in \tC_\alpha$ and is trivial otherwise. 
Then the inclusion of $J^0[\ell^n]\subset J^1[\ell^n]$ at $y$ can be identified as
\begin{equation}\label{E:stalk J0}
J^0[\ell^n]_y\simeq  T^{\rW_{\tilde y},0}[\ell^n]= \xcoch^{\rW_{\tilde y}}\otimes\mu_{\ell^n}\subset J^1[\ell^n]_y\simeq T^{\rW_{\tilde y}}[\ell^n]=(\xcoch\otimes\mu_{\ell^n})^{\rW_{\tilde y}}. 
\end{equation}
Therefore, the cokernel of the inclusion $J^0[\ell^n]\subset J^1[\ell^n]=\pi_*(\xcoch\otimes\mu_{\ell^n})^{\rW}$ is a sheaf supported on the ramification loci of $\pi:\tC\to C$, and whose stalk at $y$ can be identified with ${\rm H}^1(\rW_{\tilde y},\xcoch)[\ell^n]\otimes \mu_{\ell^n}$. Since ${\rm H}^1(\rW_{\tilde y},\xcoch)$ is a finite group,
passing to the inverse limit gives 
\[ \underleftarrow{\lim}_n\on{H}^1(C,J^0[\ell^n])\simeq \underleftarrow{\lim}_n{\rm H}^1(C,\pi_*(\xcoch\otimes \mu_{\ell^n})^{\rW}).\]
Therefore, it is enough to show that the inverse limit of the system of maps 
$$\Nm: {\rm H}^1(C,\pi_*(\xcoch\otimes\mu_{\ell^n}))\to {\rm H}^1(C,\pi_*(\xcoch\otimes\mu_{\ell^n})^\rW)$$ 
is surjective.

Let $j:U\subset C$ be the complement of the ramification loci of
$\pi:\tC\to C$ and let $\tilde{j}:\tU\to \tC$ be its preimage in $\tC$. Let $i: C\setminus U\to C$ be the closed embedding of ramification loci.
Then $L_n:=\pi_{*}(\xcoch\otimes\mu_{\ell^n})|_U$ is a locally free $\bZ/\ell^n$-module
on $U$ with an action of $\rW$, and the norm map $\Nm: L_n\to L_n^\rW$ is surjective. Let $F_n$ denote its kernel.
Note that since $\tilde{j}_*(\xcoch\otimes \mu_{\ell^n})=\xcoch\otimes \mu_{\ell^n}$, we have
$$\pi_*(\xcoch\otimes\mu_{\ell^n})=j_*L_n,\quad \pi_*(\xcoch\otimes\mu_{\ell^n})^\rW=j_* L_n^{\rW}.$$ 

Now, let $N_n= i^*j_*L_n^\rW$ be the restriction of $j_*L_n^\rW$ over the ramification loci. Taking cohomology of $0\to j_!L_n^\rW\to j_*L_n^\rW\to N_n\to 0$ then induces the following commutative diagram with rows and columns exact
\[\xymatrix{
&\on{H}^1_c(U,L_n) \ar[r]\ar^{\Nm}[d]&{\rm
H}^1(C,\pi_*(\xcoch\otimes\mu_{\ell^n}))\ar^{\on{Nm}}[d]\ar[r]& 0\\
{\rm H}^0(C,N_n)\ar^{\partial_n}[r]\ar@{=}[d]&{\rm H}_c^1(U,L_n^\rW)\ar[r]\ar^{q_n}[d]& {\rm H}^1(C,\pi_*(\xcoch\otimes\mu_{\ell^n})^\rW)\ar[r]\ar[d]&0\\
{\rm H}^0(C,N_n)\ar^{\delta_n}[r]&\Delta_n\ar[d]\ar[r]&Q_n\ar[r]\ar[d]&0\\
&0&0}.\]
Here $\Delta_n$ and $Q_n$ denote the cokernels of the norm maps. Recall that we want to show $\underleftarrow\lim Q_n=0$. From this diagram, this is equivalent to the surjectivity of $\underleftarrow\lim_n \delta_n$. 

It is easier to first describe the Pontrjagin dual of $\partial_n$ and $q_n$.  Note that the distinguished triangle 
$i^*j_*L_n^\rW\to j_!L^\rW_n[1]\to j_*L^\rW_n[1]\to$ is the Verdier dual of the natural distinguished triangle
$$j_*((L^\rW_n)^*\otimes\mu_{\ell^n})[1]\to Rj_*((L^{\rW}_n)^*\otimes \mu_{\ell^n})[1]\to R^1j_*((L^{\rW}_n)^*\otimes \mu_{\ell^n})\to.$$ Therefore, the dual of $\partial_n$ is the natural restriction map
\begin{equation}\label{E: coboundary}
\res:{\rm H}^1(U,(L_n^{\rW})^*\otimes\mu_{\ell^n})\to \bigoplus_{y\in C-U}{\rm H}^1(\Spec \mO_{C,y}^h\setminus\{y\}, (L_n^{\rW})^*\otimes\mu_{\ell^n}),
\end{equation}
where $\mO_{C,y}^h$ denotes the henselization of $\mO_{C,y}$.

Let $\bar \eta$ denote a geometric generic point of $\tU$. Its image in $U$ under $\pi$ is still denoted by $\bar\eta$. Then we have
\[(L_n)_{\bar\eta}\simeq \bbZ[\rW]\otimes(\xcoch\otimes\mu_{\ell^n}),\]
and the monodromy representation $\rho:\pi_1(U,\bar\eta)\to \GL((L_n)_{\bar\eta})$ is
given by 
$$\rho(\gamma)(a\otimes b)=\rho(\gamma)a\otimes b.$$ There is another action of $\rW$ on
$(L_n)_{\bar\eta}$ given by 
$$w(a\otimes b)=aw^{-1}\otimes wb,$$ which gives rise
to the $\rW$-action on $L_n$. Then there is a canonical isomorphism
\begin{equation}\label{E: Winv}
\xcoch\otimes\mu_{\ell^n}\simeq (L^{\rW}_n)_{\bar\eta},\quad \lambda\mapsto \sum_{w\in \rW} w\otimes w^{-1}\lambda.
\end{equation} 
Now we have the following commutative diagram 
\[\xymatrixcolsep{1pc}\xymatrix{
((L_n)^*_{\bar\eta}\otimes \mu_{\ell^n})^{\rW}\ar[r]\ar@{=}[d]&((F_n)_{\bar\eta}^*\otimes\mu_{\ell^n})^{\rW}\ar[r]\ar@{=}[d]& {\rm H^1}(\rW, \xch/\ell^n)\ar[r]\ar@{^{(}->}^{\rho^*}[d]& 0\ar[d]\\
((L_n)^*_{\bar\eta}\otimes \mu_{\ell^n})^{\pi_1(U,\bar\eta)}\ar[r] & ((F_n)_{\bar\eta}^*\otimes\mu_{\ell^n})^{\pi_1(U,\bar\eta)}\ar[r]&{\rm H}^1(\pi_1(U,\bar\eta),\xch/\ell^n)\ar[r]& {\rm H}^1(\pi_1(U,\bar\eta),(L_n)^*_{\bar\eta}\otimes \mu_{\ell^n}),
}\]
where the second row is the long exact sequence of \'etale cohomology for locally free $\bZ/\ell^n$-modules $0\to (L^{\rW}_n)^*\otimes\mu_{\ell^n}\to L_n^*\otimes\mu_{\ell^n}\to F_n^*\otimes\mu_{\ell^n}\to 0$ on $U$, and the first row is the long exact sequence of the group cohomology for their stalks at $\bar\eta$, regarded as $\rW$-modules. Here, we use: (i) ${\rm H}^1(\rW, (L_n)^*_{\bar\eta})=0$ by Shapiro's lemma;  (ii) under the isomorphism \eqref{E: Winv}, the $\pi_1(U,\bar\eta)$-action on $(L^{\rW}_n)_{\bar\eta}$ corresponds to the natural action of $\rW=\rho(\pi_1(U,\bar\eta))$ on $\xcoch\otimes\mu_{\ell^n}$; (iii) $\rho^*$ is injective since it is induced by the surjective map $\pi_1(U,\bar\eta)\to\rW$.
Therefore, it follows from the Poincar\'e duality on $U$ that the Pontrjagin dual of $q_n$ is $\rho^*$. 

Putting together, the dual of $\delta_n$ is $\res\circ\rho^*$. Now we choose a geometric generic point $\bar\eta_{\tilde y}$ of $\Spec \mO_{C,y}^h\setminus\{y\}$ over $\bar\eta$. Then $\rho(\pi_1(\Spec \mO_{C,y}^h\setminus\{y\}),\bar\eta_{\tilde y})=\langle s_{\alpha_{\tilde y}}\rangle\subset \rW$ for some root $\alpha_{\tilde y}$ (depending on $\bar{\eta}_{\tilde y}$), and there is the following commutative diagram
\[\xymatrixcolsep{5pc}\xymatrix{
{\rm H}^1(\rW, \xch/\ell^n)\ar^{\res}[r]\ar@{^{(}->}_{\rho^*}[d]\ar^{\delta_n^*}[dr]& {\rm H}^1(\langle s_{\alpha_{\tilde y}}\rangle,\xch/\ell^n)\ar@{^{(}->}^{\rho^*}[d]\\
{\rm H}^1(\pi_1(U,\bar\eta),\xch/\ell^n)\ar^{\res}[r]& {\rm H}^1(\pi_1(\Spec \mO_{C,y}^h\setminus\{y\},\bar\eta_{\tilde y}),\xch/\ell^n),
}\]  
with vertical arrows injective.
Therefore, it remains to show that 
$$\bigoplus_{y\in C- U} \underleftarrow\lim{\rm H}^1(\langle s_{\alpha_{\tilde y}}\rangle, \xch/\ell^n)^*\to \underleftarrow\lim{\rm H}^1(\rW,\xch/\ell^n)^*$$ 
is surjective. Note
\begin{equation}\label{E:delta}
\underleftarrow\lim_n {\rm H}^1(\rW, \xch/\ell^n)^*=\Hom(\underrightarrow\lim_n \on{H}^1(\rW, \frac{1}{\ell^n}\xch/\xch),\bbQ_\ell/\bZ_\ell).
\end{equation}
Using $0\to {\rm H}^1(\rW, \xch)/\ell^n\to \on{H}^1(\rW, \xch/\ell^n)\to \on{H}^2(\rW, \xch)[\ell^n]\to 0$, and the fact that $\on{H}^1(W,\xch)$ is finite, we have
$$\underrightarrow\lim_n \on{H}^1(\rW, \frac{1}{\ell^n}\xch/\xch)=\on{H}^2(\rW,\xch)[\ell^\infty]=\on{H}^1(\rW,\xch\otimes \bbQ_\ell/\bZ_\ell).$$
So it reduces to show that
\[{\rm H}^1(\rW, \xch\otimes \bbQ_\ell/\bbZ_\ell)\to \bigoplus_{y\in C-U}{\rm H}^1(\langle s_{\alpha_{\tilde y}}\rangle , \xch\otimes \bbQ_\ell/\bbZ_\ell)\]
is injective. As mentioned in the proof of Lemma \ref{P1}, $\tC_\al$ is non-empty for every $\al\in\Phi$. So $\{\al_{\tilde y}\}$ contain a set of representatives of $\Phi/\rW$. Now we can apply Lemma \ref{L:H1W} to $M=\xch\otimes \bbQ_\ell/\bbZ_\ell$ to finish the proof of the proposition.
\end{proof}

Now, let $A'=\ker(P^0\to P)$, and $A=\ker(P\to P^1)$. Then by the above
proposition, $$\ker D_{cl}=A'/(\breve A)^*.$$ As both $A'$ and $\breve A$ are
finite \'{e}tale groups, it is enough to show that $|A'|=|\breve
A|$, where for a finite group $\Gamma$, $|\Gamma|$ denotes the
number of its elements. This is the subject of the next subsection.

\subsection{Calculation of finite groups}\label{cal of finite}
Let us understand $A$. In fact, it is better to pick up $\infty\in
C$ \emph{away from} the ramification loci. Let $\calO_\infty$
denote the completed local ring of $C$ at $\infty$. Let $J_\infty$
be the dilatation of $J$ along the unit of the fiber of $J$ at
$\infty$. By definition (see \cite[\S 2]{BLR} for details), $J_\infty$ is the unique smooth group scheme over $C$ equipped with a natural map $J_\infty\to J$, which is an isomorphism away from $\infty$ and induces an isomorphism from
$J_\infty(\calO_\infty)$ to the first congruence
subgroup of $J(\calO_\infty)$. Let $\scP_\infty$ be the Picard
stack of $J_\infty$-torsors on $C$. One can also interpret
$\scP_\infty$ as the Picard stack of $J$-torsors on $C$ together
with a trivialization at $\infty$. Observe that $\scP_\infty$ is in
fact a scheme. Let $P_\infty$ denote the neutral connected component
of $\scP_\infty$. Similarly, one can define $J^0_\infty, J^1_\infty,
P^0_\infty, P^1_\infty$ etc. Let $A_\infty=\ker (P_\infty\to
P^1_\infty)$ and $A'_\infty=\ker(P^0_\infty\to P_\infty)$.

\begin{lem}\label{L: four term exact}
There are the following two exact sequences
\[1\to A_\infty\to\Gamma(C,J^1/J)\to \pi_0(\scP)\to \pi_0(\scP^1)\to 1\]
and
\[1\to \Aut_\scP(e)\to \Aut_{\scP^1}(e)\to A_\infty\to A\to 1.\]
Similarly,
\[1\to A'_\infty\to \Gamma(C,J/J^0)\to \pi_0(\scP^0)\to \pi_0(\scP)\to 1\]
and
\[1\to \Aut_{\scP^0}(e)\to \Aut_{\scP}(e)\to A'_\infty\to A'\to 1.\]
\end{lem}
\begin{proof} Consider 
\[\begin{CD}
1@>>> J_\infty @>>> J^1_\infty @>>> J^1_\infty/J_\infty @>>>1\\
@.@VVV@VVV@VVV@.\\
1@>>> J@>>> J^1@>>> J^1/J@>>> 1.
\end{CD}\]
Taking $R\Gamma(C,-)$, we obtain
\begin{equation}\label{level vs non level}\begin{CD}
1@>>> \Gamma(C, J^1_\infty/J_\infty)@>>> \sP_\infty @>>> \sP_\infty^1 @>>>1\\
@.@VVV@VVV@VVV@.\\
1@>>> \Gamma(C, J^1/J)@>>> \sP @>>> \sP^1 @>>>1.
\end{CD}\end{equation}
Since $\sP_\infty$ and $\sP_\infty^1$ are schemes, the first row of \eqref{level vs non level} gives
\[1\to A_\infty\to \Gamma(C, J^1_\infty/J_\infty)\to \pi_0(\sP_\infty)\to \pi_0(\sP_\infty^1)\to 1.\]
Since $J^1_\infty/J_\infty=J^1/J$ and
$\pi_0(\scP_\infty)=\pi_0(\scP),
\pi_0(\scP^1_\infty)=\pi_0(\scP^1)$, we obtain the first exact sequence of the lemma.
In addition, combining with the second row of \eqref{level vs non level}, we obtain the short exact sequence of Beilinson's 1-motives
\[
1\to A_\infty \to W_1\sP \to W_1\sP^1 \to 1, 
\]
which in turn gives the second exact sequence of the lemma. The proof of the last two exact sequences of the lemma is similar (by considering $R\Gamma$ of the short exact sequence $1\to J_\infty^0\to J_\infty \to J_\infty/J_\infty^0\to 1$).
\end{proof}

As a corollary, we can write
\[|A|=\frac{|\Gamma(C,J^1/J)|}{|\on{coker}(\Aut_{\scP}(e)\to\Aut_{\scP^1}(e))||\ker(\pi_0(\scP)\to\pi_0(\scP^1)|}.\]
and
\[|A'|=\frac{|(\Gamma(C,J/J^0)|}{|\on{coker}(\Aut_{\scP^0}(e)\to\Aut_{\scP}(e))||\ker(\pi_0(\scP^0)\to\pi_0(\scP))|}.\]
Therefore to show that $|\breve A|= |A'|$, it is enough to show
that
\begin{enumerate}
\item $|\Gamma(C,\breve J^1/\breve J)|=|\Gamma(C,J/J^0)|$;
\item
$|\on{coker}(\Aut_{\breve\scP}(e)\to\Aut_{\breve\scP^1}(e))|=|\on{coker}(\pi_0(\scP)^*\to\pi_0(\scP^0)^*)|$;
\item
$|\ker(\pi_0(\breve\scP)\to\pi_0(\breve\scP^1))|=
|\ker(\Aut_{\scP}(e)^*\to\Aut_{\scP^0}(e)^*)|$.
\end{enumerate}
We first prove (1). By \eqref{J and J1}, 
\begin{equation}\label{E:J and J1}
\Gamma(C,\breve J^1/\breve
J)=(\bigoplus_{\alpha}\mu_\alpha(\tC_\alpha))^\rW.
\end{equation}
 Observe that
$\mu_\alpha\neq 0$ if and only if $\alpha$ is not a primitive root,
i.e. $\alpha/2\in\xch$. On the other hand, from
\[1\to J/J_0\to J_1/J_0\to J_1/J\to 1,\]
one can see that
the character group of $\Gamma(C,J/J_0)$ is
$(\bigoplus_{x\in\sqcup\tC_\alpha}\frac{\bbQ\alpha\cap\xch}{\bbZ\alpha})^\rW$.
Then (1) follows.

Next, we prove (2). In fact, it follows from \S \ref{first reduction} and \cite[\S 4.10]{N2}  that
both maps can be identified with the natural inclusion
$Z(\breve G)\to \breve T^\rW$. 

Finally, we show (3). Recall that $\Aut_{\scP}(e)=\{t\in T\mid \al(t)=1, \ \al\in\Phi\}$. On the other hand, from the above description of $\Gamma(C,J/J_0)^*$,
\[\Aut_{\scP^0}(e)=\{t\in T\mid \lambda(t)=1 \mbox{ if } \lambda\in \bbQ\al\cap \xch, \al\in\Phi)\}.\]
Therefore,
$$
\ker(\Aut_{\scP}(e)^*\to\Aut_{\scP^0}(e)^*)= \frac{ \sum_{\al\in\Phi} (\bbQ \al\cap
\xch)}{\bbZ\Phi}.
$$

To calculate $\ker(\pi_0(\breve\scP)\to\pi_0(\breve\scP^1))$, we choose $\tilde y\in \tC_{\al_{\tilde y}}$ above $y\in C-U$ for every point in the ramification loci. The restriction of $1\to \breve J\to \breve J^1\to \breve J^1/\breve J\to 0$ at $y$ then can be identified with
$$1\to \ker(\breve\al_{\tilde y})\to \breve T^{s_{\al_{\tilde y}}}\stackrel{\breve \al_{\tilde y}}{\to} \mu_{\al_{\tilde y}}\simeq \frac{\bbQ\al_{\tilde y}\cap \xch}{\bZ\al_{\tilde y}}\to 1.$$
It follows that the coboundary map $\Ga(C,\breve J^1/\breve J)\to {\rm H}^1(C,\breve J)$ can be identified with
$$\bigoplus_{y\in C-U} \frac{\bbQ\al_{\tilde y}\cap \xch}{\bZ\al_{\tilde y}}\to {\rm H}^1(C,\breve J), \quad \lambda_{\tilde y}\in \frac{\bbQ\al_{\tilde y}\cap \xch}{\bZ\al_{\tilde y}}\mapsto\on{AJ}^{\breve\sP}(\tilde y, \lambda_{\tilde y}),$$
where $\on{AJ}^{\breve\sP}$ is the Abel-Jacobi map introduced before. 
Of course, this map does not really depend on the choice of liftings of $y\in C-U$ since $\on{AJ}^{\breve\sP}(\tilde y, \lambda_{\tilde y})=\on{AJ}^{\breve\sP}(w\tilde y, w\lambda_{\tilde y})$.

Now, as in the proof of Lemma \ref{L: four term exact}, we have a right exact sequence
\[\Ga(C, \breve J^1/\breve J)\to \pi_0(\breve \sP)\to \pi_0(\breve \sP^1)\to 0.\] 
Since the Abel-Jacobi map induces $\frac{\xch}{\bZ\Phi}\simeq \pi_0(\breve\sP)$, we deduce that
$$\ker(\pi_0(\breve\scP)\to\pi_0(\breve\scP^1))=\on{Im}(\Ga(C, \breve J^1/\breve J)\to \pi_0(\breve \sP))=\frac{ \sum_{\al\in\Phi} (\bbQ \al\cap
\xch)}{\bbZ\Phi}.$$ 
Therefore, (3) follows and the proof of Theorem \ref{classical limit} is complete.

\begin{remark}
As a byproduct of the proof, we obtain
$$\pi_0(\breve\scP^1)=\frac{\xch}{\sum_{\al\in\Phi}(\bbQ\al\cap\xch)}.$$
It seems that this expression of $\pi_0(\breve\scP^1)$ did not
appear in literature before.
\end{remark}

\subsection{A property of $\mathfrak D_{cl}$}\label{l_J}
In this subsection, we assume that $G$ is semisimple. We show that the classical duality $\mathfrak D_{cl}$ intertwines certain homomorphisms of Picard stacks over the Hitchin base $B^0$. As before, we omit the subscript $^0$.

Let $Z(\breve G)\on{-tors}(C)$ denote the Picard stack of $Z(\breve G)$-torsors on $C$. We start with the construction of two homomorphisms
\begin{equation}\label{flJ}
\fl_J: Z(\breve G)\on{-tors}(C)\times B\to \sP^\vee,\quad \breve\fl_{J}: Z(\breve G)\on{-tors}(C)\times B\to \breve\sP.
\end{equation}

The definition of $\breve\fl_J$ is easy.  It is induced by the natural map of group schemes 
$$Z({\breve G})\times (C\times S)\ra\breve J_b,$$
for every $b:S\to B$.  
For $K\in Z(\breve G)\on{-tors}(C)$, let
$$K_{J}:=\breve\fl_J(\{K\}\times B)\in\breve\sP(B).$$

Next we define $\fl_J$. For the purpose, we need to generalize a construction of \cite[\S 4.1]{BD}.
Let $\pi:\mC\to B$ be a smooth proper relative curve over an affine base $B$ (later on $\mC=C\times B$). Let \[0\to \Pi(1)\to \tilde \mG\to \mG\to 0\]
be an extension of smooth affine group schemes on $\mC$ with $\Pi$ commutative finite \'{e}tale. Let $\Pi^\vee=\Hom(\Pi,\bG_m)$ be its Cartier dual, which is assumed to be \'{e}tale as well (in particular, the order of $\Pi$ is prime to $\cha\ k$), and let $\Pi^\vee\on{-tors}(\mC/B)$ denote the Picard stack (over $B$) of $\Pi^\vee$-torsors on $\mC$ relative to $B$. We construct a Picard functor
\[\fl_{\mG}: \Pi^\vee\on{-tors}(\mC/B)\to \Pic(\Bun_{\mG}(\mC/B))\]
of Picard stacks over $B$
as follows. First, let $\Pi\on{-gerbes}(\mC/B)$ denote the Picard 2-stack of $\Pi$-gerbes on $\mC$ relative to $B$, regarded as a Picard stack. Then there is the 
generalized (or categorical) Chern class map 
$$\tilde c_{\mG}:\Bun_{\mG}(\mC/B)\ra\Pi(1)\on{-gerbes}(\mC/B)$$
that assigns every $B$-scheme $S$ and a $\mG$-torsor $E$ on $\mC_S$, the Picard groupoid of the lifting of $E$ to a $\tilde \mG$-torsor.
We have
\begin{lemma}\label{pairing and duality}
The dual of the Picard stack
$\Pi\on{-gerbes}(\mC/B)$ (as defined in \S \ref{duality}) is canonically isomorphic to $\Pi^\vee\on{-tors}(\mC/B)$. 
\end{lemma}

We follow \cite[\S 4.1.5]{BD} for a ``scientific interpretation" of this lemma and refer to \cite[\S 4.1.2-4.1.4]{BD} for the precise construction. As explained in \S \ref{Picard}, the Picard stack $\Pi\on{-gerbes}(\mC/B)$ is incarnated by the complex $\tau_{\geq -1}R\pi_*\Pi[2](1)$, and $\Pi^\vee\on{-tors}(\mC/B)$ is incarnated by the complex $\tau_{\leq 0}R\pi_*\Pi^\vee[1]$. Let $\mu'_\infty$ denote the group of prime-to-$p$ roots of unit. Note that $\pi^!\mu'_\infty\simeq \mu'_\infty[2](1)$. Then by the Verdier duality, 
$$R\Hom(R\pi_*\Pi[2](1),\mu'_\infty)\simeq R\pi_*R\Hom(\Pi[2](1),\pi^!\mu'_\infty)\simeq R\pi_*\Pi^\vee.$$
By shifting by $[1]$ and truncating $\tau_{\leq 0}$, one obtains the lemma. As explained in \cite[\S 4.1.5]{BD}, working in the framework of derived categories is in not enough to turn the above heuristics into a proof. One can either give a concrete construction as in \cite[\S 4.1.2-4.1.4]{BD} or understand the above argument in the framework of stable $\infty$-categories.
 
Therefore, each $K\in\Pi^\vee\on{-tors}(\mC/B)$ defines a morphism
\[\fl_{\mG,K}:\Bun_{\mG}(\mC/B)\stackrel{\tilde c_{\mG}}\ra\Pi(1)\on{-gerbes}(\mC/B)\stackrel{<,K>}
\ra  B\bG_m\]
or equivalently a line bundle $\mL_{\mG,K}$ on $\Bun_{\mG}(\mC/B)$ and 
the assignment $K\ra\mL_{\mG,K}$
defines a homomorphism of Picard stacks
\[\fl_\mG:\Pi^\vee\on{-tors}(\mC/B)\ra\on{Pic}(\Bun_{\mG}(\mC/B)),\] 
which factors through the $n$-torsion of $\on{Pic}(\Bun_{\mG})(\mC/B)$ where 
$n$ is the order of $\Pi^\vee$.

Note that in the above discussion we do not assume that $\mG$ is commutative. But if $\mG$ is commutative, $\Bun_{\mG}(\mC/B)$ has a natural structure of Picard stack over $B$ and 
one can check that $\fl_{\mG}$ factors through a homomorphism 
$\fl_{\mG}:\Pi^\vee\on{-tors}(\mC/B)\ra(\Bun_{\mG}(\mC/B))^\vee$.

Now let $\mC=C\times B$, where $B$ is the Hitchin base as before. Let $\mG=J_b$ and $\tilde{\mG}=(J_{sc})_b$, where $b:B\to B$ is the identity map, and $J_{sc}$ is the universal regular centralizer for $G_{sc}$, the simply-connected cover of $G$. Then $\Pi(1)=\Pi_G(1)$
is the fundamental group and $\Pi_G^\vee$ is canonical isomorphic
to the center $Z({\breve G})$ of $\breve G$. Therefore, the above construction gives $\fl_J$ as promised in \eqref{flJ}.

Note that similarly we can set $\mG=G\times\mC$ and $\mG=T\times\mC$ in the above construction so we obtain
\[\fl_G:Z({\breve G})\on{-tors}(C)\ra\on{Pic}(\Bun_G),\quad \fl_T:Z({\breve G})\on{-tors}(C)\ra 
(\Bun_T)^\vee.\]
For $K\in Z({\breve G})\on{-tors}(C)$, let
$\mL_{G,K}:=\fl_G(K)\in\on{Pic}(\Bun_G)$,
$\mL_{J,K}:=\fl_{J}(\{K\}\times B)\in(\sP)^\vee(B)$. 
The following lemma will be used in \S \ref{tensoring action}.
\begin{lemma}\label{Aut line bundle 2}
Let $\kappa$ be a square root of $\omega$.
Then the pullback of $\mL_{G,K}$ along the map
$\sP\stackrel{\epsilon_\kappa}\ra\on{Higgs}\stackrel{\pr}\ra\Bun_G$ is 
isomorphic to $\mL_{J,K}$, i.e., 
we have $\mL_{J,K}\is \epsilon_\kappa^*\circ \pr^*\mL_{G,K}$.
\end{lemma}
\begin{proof}
It is enough to
show that the composition 
\[\sP\stackrel{\epsilon_\kappa}\ra\on{Higgs}\stackrel{\pr}\ra\Bun_G\stackrel{\tilde c_G}\ra\Pi_G(1)\on{-gerbes}(C)\]
is isomorphic to 
\[\sP\stackrel{\tilde c_J}\ra\Pi_G\on{-gerbes}(C)\times B\ra\Pi_G(1)\on{-gerbes}(C).\]
Let $P\in\sP$ and $(E,\phi):=\epsilon_\kappa(P)$. We need to construct a functorial isomorphism between
$\tilde c_J(P)$ and $\tilde c_G(E)$
where $\tilde c_J(P)$ (resp. $\tilde c_G(E)$) is the $\Pi_G(1)$-gerbe
of liftings of $P$ to $J_{sc}$-torsors (resp. $G_{sc}$-torsors).

Note that the $G$-torsor $E_\kappa$ given by the Kostant section has a natural lifting 
$\tilde E_\kappa\in\Bun_{G_{sc}
}$, since the cocharacter $2\rho:\bG_m\ra G$ has a natural lifting to $G_{sc}$. Thus any lifting $\tilde P\in\tilde c_J(P)$ defines a lifting 
$\tilde E:=\tilde P\times^{J_{sc}}\tilde E_\kappa\in\Bun_{G_{sc}}$ 
of $E=P\times^{J}E_\kappa$ and the assignment $\tilde P\ra\tilde E$ defines a functorial
isomorphism between $\tilde c_J(P)$ and $\tilde c_G(E)$. The lemma follows.
\end{proof}

\medskip
 
We write $l_G,l_T, l_J$ for the induced maps between the corresponding coarse moduli spaces. The following lemma is a specialization of our construction of the duality given in Lemma \ref{pairing and duality}.
\begin{lemma}\label{l_T}
Let 
$n$ be a positive integer such that $p\nmid n$.
Let 
$$\fl:\breve T[n]\on{-tors}(C)\ra(\Bun_T)^\vee[n]$$ be the
tensor functor given by the extension 
$0\ra T[n]\ra T\stackrel{n}\ra T\ra 0$.\footnote{Recall that we have a canonical isomorphism 
$\breve T[n]\is (T[n])^\vee$.} Then the induced map
$l:H^1(C,\breve T[n])\ra H^1(C, T[n])^\vee$  
between the coarse moduli spaces is the same the as map given by the Poincare duality.
\end{lemma}

Now we are ready to state the result in this subsection.

\begin{prop}\label{Aut line bundle 1}
There is a natural isomorphism of functors 
$\mathfrak D_{cl}\circ\mathfrak l_J\is\breve\fl_J$. In particular, we have 
$\mathfrak D_{cl}(\mL_{J,K})\is K_J$.
\end{prop}
\begin{proof}
Let $\breve G_{ad}$ denote the adjoint group of $\breve G$. Note that it is the Langlands dual group of $G_{sc}$. Let $\breve J_{ad}$ be the universal centralizer for $\breve G_{ad}$, and $\breve\s_{ad}$ denote the Picard stack of $\breve J_{ad}$-torsors.
We first claim that the composition 
\[Z({\breve G})\on{-tors}(C)\times B\stackrel{\fl_J}\ra \sP^\vee\stackrel{\mathfrak D_{cl}}\is\breve\sP\ra\breve\sP_{ad}\]
is trivial. From the construction of $\frakD_{cl}$, we have the following commutative diagram 
\[\xymatrix{\sP^\vee\ar[r]^{\mathfrak D_{cl}}\ar[d]&\breve\sP\ar[d]\\
(\sP_{sc})^\vee\ar[r]^{\mathfrak D_{cl}}&\breve\sP_{ad}.}\]
Thus above composition can be identified with
\[Z({\breve G})\on{-tors}(C)\times B\stackrel{\fl_J}\ra (\sP)^\vee\ra(\sP_{sc})^\vee\stackrel{\mathfrak D_{cl}}\is\breve\sP_{ad}.\]
This is trivial since the composition of the first two maps is the dual of 
\[\sP_{sc}\ra\sP\stackrel{\tilde c_J}\ra\Pi_G(1)\on{-gerbes}(C)\times B,\]
which is trivial by the construction of $\tilde c_J$.

On the other hand, the short exact sequence $0\ra Z({\breve G})\times B\ra\breve J\ra\breve J_{ad}\ra 0$
induces a left exact sequence of Picard stacks 
\[0\ra Z({\breve G})\on{-tors}(C)\times B\stackrel{\breve\fl_J}\ra\breve\sP\ra\breve\sP_{ad}.\]
I.e., $Z({\breve G})\on{-tors}(C)\times B$ is identified as the kernel of $\breve\sP\ra\breve\sP_{ad}$.
Therefore, there is a morphism
$$\mathfrak i:Z({\breve G})\on{-tors}(C)\times B\ra Z({\breve G})\on{-tors}(C)\times B$$
such that $\mathfrak D_{cl}\circ\mathfrak l_J\is\breve{\mathfrak l}_J\circ\mathfrak i$. We now show that 
$\mathfrak i$ is isomorphic to the identity morphism.  
As argued in \S \ref{first reduction}, we reduce to show that $\mathfrak i$ induced 
the identity map on the coarse moduli space $H^1(C,Z({\breve G}))\times B$. 

Let $i:H^1(C,Z({\breve G}))\times B\ra H^1(C,Z({\breve G}))\times B$, 
$l_{J}:H^1(C,Z({\breve G}))\times B\ra P^\vee$ and 
$\breve l_J:H^1(C,Z({\breve G}))\times B\ra \breve P$ be the induced maps 
on the corresponding coarse moduli spaces. 
Our goal is to show that $i=\id$. Since $\Gamma(C\times B,\breve J_{ad})=0$, $\breve l_J$ is injective. Therefore, 
 it suffices to show that 
$$\breve l_J\circ (i-\id):H^1(C,Z({\breve G}))\times B\ra\breve P$$ is zero.
As in \S \ref{first reduction}, we can
prove this fiberwise, and therefore we fix $b\in B^0(k)$. Again, to
simplify notations, in the following discussion we write $\tC, J,P, \breve P$ instead of 
$\tC_b, J_b,P_b, \breve P_b$, etc. 

Let $\breve j^1:\breve P\ra H^1(\tC,\breve T)$ be the map induced by the morphism 
$\breve j^1:\pi^*\breve J\ra \breve T$.  
Then the composition $\breve j^1\circ\breve l_J:H^1(C,Z({\breve G}))\ra 
H^1(\tC,\breve T)$ is also injective (note that $\breve j^1\circ\breve l_J$ is induced by
the natural map $Z({\breve G})\ra\breve T$). Thus it is enough 
to show that $\breve j^1\circ\breve l_J\circ(i-\id)=0$. Since
$D_{cl}\circ l_J=\breve l_J\circ i$, it is equivalent to show that
\begin{equation}\label{equ}
\breve j^1\circ D_{cl}\circ l_J-\breve j^1\circ\breve l_J=0.
\end{equation}
Let us consider the following diagram 
\[
\xymatrix{H^1(C,Z({\breve G}))\ar[r]^{ l_J}\ar[d]^{\id}&H^1(C,J)^\vee\ar^{D_{cl}}[d]\ar[r]^{\Nm^\vee}&H^1(\tC, T)^\vee\ar[d]\\
H^1(C,Z({\breve G}))\ar[r]^{\breve l_J}&H^1(C,\breve J)\ar[r]^{\breve j^1}&H^1(\tC,\breve T),
}\]
where $\Nm^\vee$ is the dual of \eqref{Pnorm}, and the right vertical map is $(\on{Jac}\otimes\xcoch)^\vee\simeq \on{Jac}\otimes\xch$.
The right rectangle in the above diagram is commutative by the construction of $D_{cl}$ in \S \ref{classical duality}. Therefore it is enough to show that the outer diagram is also commutative.

Let $n$ be the order of $Z({\breve G})$. Then $\breve j^1\circ\breve l_J$
and $\on{Nm}^\vee\circ l_J$ will factor through 
$H^1(\tC,\breve T)[n]\is H^1(\tC,\breve T[n])$ and $H^1(\tC,T)^\vee[n]\is
H^1(\tC,T[n])^\vee$\footnote{Note that $p\nmid n$.}. Thus the outer diagram factors as
\[
\xymatrix{H^1(C,Z({\breve G}))\ar[r]^{\on{Nm}^\vee\circ\fl_J}\ar[d]^{\id}&H^1(\tC,T[n])^\vee\ar[d]\\
H^1(C,Z({\breve G}))\ar[r]^{\breve j^1\circ\breve l_J}&H^1(\tC,\breve T[n]),
}\]
where the right vertical arrow is now given by the Poincare duality.
Unraveling the definition of $\mathfrak l_J$, one sees that 
$\Nm^\vee\circ l_{J}$ can be identified with 
\[H^1(C,Z({\breve G}))\stackrel{}\ra H^1(\tC,\breve T[n])\stackrel{}\ra H^1(\tC,T[n])^\vee\]
where the first map is induced by the natural morphism $Z({\breve G})\ra\breve T[n]$ and the second map
is the map $l$ in Lemma \ref{l_T}. Then the commutativity of above diagram 
follows from Lemma \ref{l_T}.
\end{proof}


\section{Multiplicative 1-forms}\label{mult one forms}
In this section, we establish a technical result. Namely, we show that the pullback of the canonical 1-form $\theta_{can}$ on $T^*\Bun_G$ along $\sP\to T^*\Bun_G$ induced by a Kostant section $\kappa$ is multiplicative in the sense of \S \ref{appen:var}.

\subsection{Lie algebra valued 1-forms}\label{Lie algebra valued one forms}
In this subsection, we restrict everything to $B^0$ and therefore omit the subscript $^0$ from the notation.  
Recall that there is a group scheme $\calT=\tC\times^{\rW}T$
 over $[\tC/\rW]$ and Proposition \ref{J} says that there is a homomorphism 
 $[j^1]:[\pi]^*J\ra\calT$
 where $[\pi]: [\tC/\rW]\to C\times B$ is the projection.
It induces the following commutative diagram
\[\begin{CD}
{[\pi]}^*(\Omega_{C\times B}\otimes \Lie J)@>>> {[\pi]}^*(\Omega_{C\times B/B}\otimes \Lie J)\\
@VVV@VVV\\
\Omega_{[\tC/W]}\otimes \Lie \calT @>>> \Omega_{[\tC/W]/B}\otimes \Lie\calT\\
\end{CD}\]
Note that, due to the product structure on $C\times B$, the arrow in the upper row admits a canonical splitting. Therefore, the 
tautological section in (\ref{taut sect II})
\[(\tau:B\to B_{J})\in \Gamma(C\times B, \Omega_{C\times B/B}\otimes \Lie J)\] can be regarded as a section of ${[\pi]}^*(\Omega_{C\times B}\otimes \Lie J)$, which in turn gives 
\begin{equation}\label{thetatC}
\theta_{\tC}\in
\Gamma([\tC/\rW],\Omega_{[\tC/\rW]}\otimes\Lie\calT)
=\Gamma(\tC,\Omega_{\tC}\otimes\ft)^\rW.
\end{equation}
We denote by $\breve\theta_{\tC}\in\Gamma(\tC,\Omega_{\tC}\otimes\breve\ft)^\rW$ the corresponding section for the dual group.
 
 We shall give an alternative description of $\theta_{\tC}$.
We denote by \[\delta_\omega:\ft_\omega\ra\ft_\omega\times_{C}\ft_\omega\]
the $\bG_m$-twist by $\omega$ of the map $\delta$ as in Lemma \ref{iota}.
We regard $\ft_\omega$ and $\ft_\omega\times_C\ft_\omega$ (via the first projection) as schemes over $\fc_\omega$ and define
\[\delta_{\tC}:\tC=e^*\ft_\omega\ra e^*(\ft_\omega\times_C\ft_\omega)=
\tC\times_C(T^*C\otimes\ft)\]
to be the base change of $\delta_\omega$ via the evaluation map 
$e:C\times B\ra \fc_\omega$. By Lemma \ref{iota}, $\delta_{\tC}$ is just the pull back of the 
diagonal map $\ft_\omega\to \ft_\omega\times_C\ft_\omega$ along 
$e:C\times B\ra\fc_\omega$.

By  construction, the section $\theta_{\tC}\in\Gamma(\tC,\Omega_{\tC}\otimes\ft)^\rW$ is equal to
the following composition
\begin{equation}\label{nu2}
\tC\stackrel{\delta_{\tC}}\ra\tC\times_C(T^*C\otimes\ft)\ra T^*\tC\otimes\ft
\end{equation}
where the last map is the cotangent map for the projection  $\tC\ra C$.

The description of $\theta_{\tC}$ in 
(\ref{nu2}) implies the following relation between $\theta_{\tC}$ and $\breve\theta_{\tC}$:

\begin{lemma}\label{theta and dual theta}
Let $\sigma:\Gamma(\tC,\Omega_{\tC}\otimes\ft)^\rW\is\Gamma(\tC,\Omega_{\tC}\otimes\breve\ft)^\rW$ be the canonical isomorphism induced by 
the non-degenerate invariant form $(,)$ on $\ft$.
We have $\sigma(\theta_{\tC})=\breve\theta_{\tC}$.
\end{lemma}

\remark{}\label{relation with omega}
The 1-form $\theta_{\tC}$ is related to the canonical 1-form $\omega_C$ of $C$ in 
the following way. Let $\tilde e:\tC\ra T^*C\otimes\ft\ (=\ft_\omega)$
be the natural $\rW$-equivariant map (see \S\ref{cameral}). The natural  $\rW$-equivariant pairing 
$\xch\times\ft\ra k$ induces a $\rW$-equivariant map 
\begin{equation}\label{nu}
\nu:\tC\times\xch\stackrel{\tilde e\times id}\ra (T^*C\otimes\ft)\times\xch\ra T^*C,
\end{equation}
where $\rW$ acts diagonally on  $\tC\times\xch$.
Now the pull back of 
the canonical 1-form $\omega_C$ on $T^*C$ along $\nu$
defines a section $\nu^*\omega_C\in\Gamma(\tC,\Omega_{\tC}\otimes\ft)^\rW
$, and using the description of $\theta_{\tC}$ in (\ref{nu2}) one can check that
$\theta_{\tC}=\nu^*\omega_C$. 

\subsection{Canonical 1-form}\label{canonical one forms}
Let us denote by $T^*\Bun_G^0$ the maximal smooth open substack of $T^*\Bun_G$. Then there is a tautological section
\[\theta_{can}: T^*\Bun_G^0\to T^*(T^*\Bun_G^0).\]
Note that $T^*\Bun_G\times_BB^0\subset T^*\Bun_G^0$. From now on, we restriction everything to the open part $B^0$ and therefore will omit ${^0}$ from the subscript.
Note that for a choice of the Kostant section $\kappa$, we have an isomorphism $\epsilon_\kappa:\sP\simeq T^*\Bun_G$, and therefore we may regard $\theta_{can}$ as a section $\sP\to T^*\sP$, denoted by $\theta_\kappa$.

Let $\AJ^\sP:\tC\times\xcoch\ra\sP$ be the Abel-Jacobi map. Write 
pull back \[(\AJ^\sP)^*\theta_\kappa=\{\theta_{\kappa,\lambda}\}_{\lambda\in\xcoch}
\in\Gamma(\tC\times\xcoch,\Omega_{\tC})^\rW,\]
where $\theta_{\kappa,\lambda}\in\Gamma(\tC,\Omega_{\tC})$ 
is the restriction of $(\AJ^\sP)^*\theta_\kappa$
to $\tC\times\{\lambda\}$. A 
section $\{\alpha_\lambda\}_{\lambda\in\xcoch}\in\Gamma(\tC\times\xcoch,\Omega_{\tC})$ (resp. $\Gamma(\tC\times\xcoch,\Omega_{\tC/B})$) is called 
$\xcoch$-multiplicative if it satisfies $\alpha_{\lambda+\mu}=\alpha_\lambda+\alpha_\mu$,
for any
$\lambda,\mu\in\xcoch$.
Clearly, any $\xcoch$-multiplicative section 
$\{\alpha_\lambda\}_{\lambda\in\xcoch}$ 
corresponds to a $\breve\ft$-valued section  
$\alpha\in\Gamma(\tC,\Omega_{\tC}\otimes\breve\ft)$ (resp. $\Gamma(\tC,\Omega_{\tC/B}\otimes\breve\ft)$). 
The rest of the section is mainly devoted to the proof of the 
following result.
\begin{prop}\label{kappa=tC}
The 1-form $(\AJ^\sP)^*\theta_\kappa$ is $\xcoch$-multiplicative. 
Moreover, if we regard
$(\AJ^\sP)^*\theta_\kappa$ as a section of $\Gamma(\tC,\Omega_{\tC}\otimes\breve\ft)^{\rW}$ 
we have \[(\AJ^\sP)^*\theta_\kappa=\breve\theta_{\tC}.\]
Where $\breve\theta_{\tC}$ is the section defined in \S\ref{Lie algebra valued one forms}.
\end{prop}

We have the following corollary.
Recall the notion of multiplicative sections $\sP\to T^*\sP$ as defined In \S \ref{appen:var}. 
\begin{corollary}\label{multiplicative}
The section $\theta_\kappa$ is multiplicative in the sense of \S \ref{appen:var}. In addition, it is independent of the choice of 
Kostant section $\kappa$.
\end{corollary}

\begin{proof}
We first show that $\theta_\kappa$ is multiplicative.
Consider the section 
$$m^*\theta_\kappa:  \sP\times_B\sP\to T^*\sP\times_\sP(\sP\times_B\sP)\to T^*(\sP\times_B\sP),$$
where the first map is the base change of $\theta_\kappa$ along the multiplication $m:\sP\times_B\sP\to \sP$, and the second map is the differential $m_d$ of $m$. On the other hand, consider
\[(\theta_\kappa,\theta_\kappa): \sP\times_B\sP\to (T^*\sP\times T^*\sP)|_{\sP\times_B\sP}\to T^*(\sP\times_B\sP).\]
We need to show that $(\theta_\kappa,\theta_\kappa)=m^*\theta_\kappa$.

We first have the following lemma, whose proof is independent of Proposition \ref{kappa=tC}.
\begin{lem}\label{proj mult}
The projection of $\theta_\kappa$ along $T^*\sP\to T^*(\sP/B)$ is multiplicative. More precisely, the images of $m^*\theta_\kappa$ and $(\theta_\kappa,\theta_\kappa)$ in $T^*(\sP\times_B\sP/B)$ are the same.
\end{lem}
\begin{proof}
Consider the following short exact sequence of vector bundles on $\sP\times_B\sP$ 
\[0\to T^*B\times_B(\sP\times_B\sP)\to T^*(\sP\times_B\sP)\to T^*(\sP\times_B\sP/B)\to 0.\]
As the projection of $\theta_\kappa$ along $T^*\sP\to T^*(\sP/B)$ is identified with $\tau^*\times\id$ (cf. Lemma \ref{iden with tau}), the restriction of $\theta_\kappa$ to 
each fiber $\sP_b$ is given by the "constant" 1-form $\tau^*|_b\in
\Gamma(C,(\Lie J_b)^*\otimes\omega)$. Therefore, $(\theta_\kappa,\theta_\kappa)=m^*\theta_\kappa$ in $T^*(\sP\times_B\sP/B)$. 
\end{proof}

Therefore, their difference can be regarded as a section $$m^*\theta_\kappa-(\theta_\kappa,\theta_\kappa)\in\Gamma(\sP\times_B\sP,\pr^*\Omega_B)=(\pi_0(\sP)\times\pi_0(\sP))\times\Gamma(B,\Omega_B).$$

The Abel-Jacobi map $\AJ^\sP:\tC\times\xcoch\to \sP$ induces a map $$\AJ^{\sP,2}:\tC\times\xcoch\times\xcoch\ra\sP\times_B\sP.$$
It is enough to show that the pullback of $m^*\theta_\kappa-(\theta_\kappa,\theta_\kappa)$ in
$$\Gamma((\tC\times\xcoch\times\xcoch),\pr^*\Omega_B)=(\xcoch\times\xcoch)\times\Gamma(B,\Omega_B)$$
vanishes.  
By Proposition \ref{kappa=tC}, the one form $(\AJ^\sP)^*\theta_{\kappa}=\{\theta_{\kappa,\lambda}\}_{\lambda\in\xcoch}$ is  
$\xcoch$-multiplicative, 
thus for any $\lambda,\mu\in\xcoch$
we have 
\[(\AJ^{\sP,2})^*(m^*\theta_\kappa-(\theta_\kappa,\theta_\kappa))|_{\tC\times\{\lambda\}\times\{\mu\}}=\theta_{\kappa,\lambda+\mu}-
(\theta_{\kappa,\lambda}+\theta_{\kappa,\mu})=0.\] This finishes the proof of multiplicative property of $\theta_\kappa$.
The independence of $\kappa$ follows from $(\AJ^\sP)^*\theta_\kappa=\breve\theta_{\tC}$.

\end{proof}

\medskip

\noindent\bf Notation. \rm In what follows, we denote the multiplicative 1-form $\theta_\kappa$ on $\sP$ by $\theta_m$.

\begin{remark}
Let $a:=m\circ (\AJ^\sP\times\on{id}):(\tC\times\xcoch)\times_B\sP\ra\sP\times_B\sP\ra\sP$
be the action map. Then Remark \ref{relation with omega} together with Corollary \ref{multiplicative} implies that
\[a^*\theta_m=\breve\nu^*(\omega_C)\boxtimes\theta_m.\]
Here $\breve\nu$ is the map in (\ref{nu}) for the dual group.
In the case of $G=\GL_n$, (a variant of) this identity was proved in \cite[Theorem 4.12]{BB}. 
\end{remark}
\subsection{Proof of Proposition \ref{kappa=tC}: first reductions}
Let $\underline\theta_\kappa$ and $\underline{\breve\theta}_{\tC}$ be the projections of 
$(\AJ^\sP)^*\theta_\kappa$ and $\breve\theta_{\tC}$ along 
$$\Gamma(\tC\times\xcoch,\Omega_{\tC})\to \Gamma(\tC\times\xcoch,\Omega_{\tC/B}).$$
Lemma \ref{proj mult} implies that $\underline\theta_\kappa$ is $\xcoch$-multiplicative and can
be regarded as an element in $\Gamma(\tC,\Omega_{\tC/B}\otimes\breve\ft)^\rW$. Let us first show that $\underline\theta_\kappa=\underline{\breve\theta}_{\tC}$.

Recall in (\ref{transpose III}) we have introduced a morphism 
$\iota:\Lie J\ra(\Lie J)^*$. It follows
from the definition of $\iota$ in \emph{loc. cit.} that the following diagram 
is commutative
\[\xymatrix{[\pi]^*\Lie J\ar[r]^{[\pi]^*\iota\ \ }\ar[d]^{}&[\pi]^*(\Lie J)^*
\\\Lie\calT\ar[r]&\Lie\breve\calT\ar[u]_{},}\]
where the arrow in the bottom row is the morphism $\Lie\calT\ra\Lie\breve\calT$ 
induced by the invariant from $(,)$ on $\ft$. (Recall (cf. \S\ref{Lie algebra valued one forms}) 
that $\calT:=\tC\times^W T$ is a group scheme over 
$[\tC/W]$ and $[\pi]:[\tC/W]\ra C\times B$.)
It induces the following commutative diagram

\[\xymatrix{\Gamma(C\times B,\Omega_{C\times B/B}\otimes\Lie J)\ar[r]^{\iota_*\ }\ar[d]&
\Gamma(C\times B,\Omega_{C\times B/B}\otimes\Lie J^*)
\\\Gamma(\tC,\Omega_{\tC/B}\otimes\ft)^\rW\ar[r]^{\sigma}&
\Gamma(\tC,\Omega_{\tC/B}\otimes\breve\ft)^\rW\ar[u]_\upsilon
}\]
Recall the sections $\tau\in\Gamma(C\times B,\Omega_{C\times B/B}\otimes\Lie J)$ and 
$\tau^*\in\Gamma(C\times B,\Omega_{C\times B/B}\otimes\Lie J^*)$ in \S\ref{tau and c}. Note 
the map $\upsilon$ in the diagram above is an isomorphism\footnote{It is the relative cotangent map of the isogeny 
$\sP\ra\Bun_T^\rW(\tC/B)$.} and it identifies $\underline{\theta}_\kappa$ 
with the section $\tau^*$.
On the other hand, 
Lemma \ref{theta and dual theta}
implies the section 
$\underline{\breve\theta}_{\tC}\in\Gamma(\tC,\Omega_{\tC/B}\otimes\breve\ft)^\rW$ is 
equal to the image of $\tau$ under the composition of the morphisms in the lower left corner of the above diagram. Thus, 
$\upsilon(\underline{\breve\theta}_{\tC})=\iota_*(\tau)=\tau^*$. 
Therefore both $\underline{\breve\theta}_{\tC}$ and $\underline{\theta}_\kappa$
map to $\tau^*$ under the isomorphism $\upsilon$, which implies 
$\underline{\breve\theta}_{\tC}=\underline\theta_\kappa$.

As a consequence, difference $\breve\theta_{\tC}-(\AJ^\sP)^*\theta_\kappa$ can
be regarded as a section
\begin{equation}\label{vertical}
\breve\theta_{\tC}-(\AJ^\sP)^*\theta_\kappa\in
\Gamma(\tC\times\xcoch,\pr^*\Omega_B).
\end{equation}
We need to show that it
is zero. Let $\tU\subset\tC$ be the largest open subset such that 
$\tU\ra C\times B$ is \'etale. 
It is enough to show that 
$\breve\theta_{\tC}-(\AJ^\sP)^*\theta_\kappa|_{\tU\times\xcoch}=0.$
Note that for $\tilde x\in\tU$ we have a canonical decomposition
$T_{\tilde x}\tC=T_xC\oplus T_bB$ and by \eqref{vertical} it 
suffices to show that 
$(\breve\theta_{\tC}-(\AJ^\sP)^*\theta_\kappa)|_{T_bB}=0$.
As the section $\breve\theta_{\tC}$ is induced by the canonical 
splitting $\Omega_{C\times B/B}\otimes\Lie\breve 
J\ra\Omega_{C\times B}\otimes\Lie\breve J$,
the restriction of  $\breve\theta_{\tC}$ to $T_bB$ is zero, so we reduce to show that
$(\AJ^\sP)^*\theta_\kappa|_{T_bB}=0$, i.e. for any $\lambda\in\xcoch$ and
$v\in T_bB$ we have 
\begin{equation}\label{main equ}
\langle\theta_{\kappa,\lambda},v\rangle=\langle(\AJ^\sP)^*\theta_\kappa|_{\tC\times\{\lambda\}},v\rangle=0.
\end{equation}
For the later purpose, we introduce some notations.
Let $(E_\kappa,\phi_\kappa)$ be the Higgs field on $C\times B$ obtained by the pullback along the Kostant section $\kappa$. For every $\lambda\in\xcoch$ 
let $\AJ^{\sP,\lambda}:\tC\ra\sP$ denote the corresponding component of the Abel-Jacobi map and 
let
$$(E_{\tilde{x}},\phi_{\tilde{x}}):=\AJ^{\sP,\lambda}(\tilde x)\times^{J_b}(E_\kappa,\phi_\kappa)|_{C\times\{b\}},$$
be the image of $\tilde x$ under the map $\tC\stackrel{\AJ^{\sP,\lambda}}\ra\sP\stackrel{\epsilon_\kappa}\is T^*\Bun_G$.
We also define 
\[a_\lambda:\tC\stackrel{\AJ^{\sP,\lambda}}\ra\sP\stackrel{\epsilon_\kappa}\is T^*\Bun_G\ra\Bun_G.\]
Since $\theta_\kappa=\epsilon_\kappa^*\theta_{can}$, we have 
\[\langle\theta_{\kappa,\lambda},v\rangle=\langle(\AJ^{\sP,\lambda})^*\epsilon_\kappa^*\theta_{can},v\rangle=
\langle\theta_{can},(\epsilon_\kappa)_*(\AJ^{\sP,\lambda})_*v\rangle=
\langle\phi_{\tilde x},{a_\lambda}_*v\rangle,\]
where 
${a_\lambda}_*:T_{\tilde x}\tC\ra T_{E_{\tilde x}}\Bun_G\is H^1(C,\ad E_{\tilde x})$ is the
differential of $a_\lambda$ and 
the last pairing is induced by the Serre duality 
$H^0(C,\ad E_{\tilde x}\otimes\Omega_C)\is H^1(C,\ad E_{\tilde x})^*$.

Therefore we reduce
to show the following:
\begin{proposition}\label{phi=0}
For any $v\in T_bB\subset T_{\tilde x}\tC=T_xC\oplus T_bB$,
the pairing $\langle\phi_{\tilde x},{a_\lambda}_*v\rangle$ is zero. 
\end{proposition}
\subsection{Proof of Proposition \ref{phi=0}: calculations of differentials}
We
shall need several preliminary 
steps. 
Recall that there is the $E_\kappa$-twist global Grassmannian $\Gr(E_\kappa)$ which classify the 
triples $(x,E,\beta)$ where $x\in C$, $E$ is a $G$-torsor and $\beta: E_\kappa|_{C-\{x\}}\simeq E|_{C-\{x\}}$ is an isomorphism. Given a dominant coweight $\mu$ (with respect to the set of simple roots we choose), it makes sense to talk about the closed substack $\on{Gr}_{\leq \mu}(E_\kappa)$, consisting of those $\beta: E_\kappa|_{C-\{x\}}\simeq E|_{C-\{x\}}$ having relative position $\leq\mu$ at $x$ (cf. \cite[\S 5.2.2]{BD}). Let $\on{Gr}_{\mu}
(E_\kappa)=\on{Gr}_{\leq\mu}(E_\kappa)-\bigcup_{\lambda<\mu}\on{Gr}_{\leq\lambda}(E_\kappa)$.  
We have natural projection maps 
\[\Bun_G\stackrel{pr_1}{\leftarrow} \Gr_{\leq\mu}(E_\kappa)\stackrel{pr_2}{\to} C.\]
For any $x\in C$, let
$$\Gr_{x}(E_\kappa):=\Gr(E_\kappa)\times_C\{x\},$$ 
and similarly we have
$\Gr_{x,\leq\mu}(E_\kappa), \ \Gr_{x,\mu}(E_\kappa)$.

Note that for any $\tilde x\in\tC$ the $J$-torsor $\AJ^{\sP,\lambda}(\tilde x)\in\sP$ has a canonical 
trivialization $s$ over $C-x$ (here $x$ is the image of $\tilde x$ in $C$), thus it induces a 
canonial isomorphism $\beta:E_\kappa|_{C-x}\is E_{\tilde x}|_{C-x}$ 
(recall that $E_{\tilde x}:=\AJ^{\sP,\lambda}(\tilde x)\times^JE_\kappa$).
The assignment $\tilde x\ra (x,E_{\tilde x},\beta)$ defines 
a morphism $\tilde a_\lambda:\tC\ra\Gr(E_\kappa)$.
We have the following key lemma:
\begin{lemma}\label{Hk}
Let $\mu$ be a dominant coweight and $\lambda\in\rW\cdot\mu$. The morphism $\tilde a_\lambda$ factors through $\Gr_{\leq\mu}(E_\kappa)$
and the following diagram 
\[\xymatrix{\tC\ar[r]^{\tilde a_\lambda\ \ \ \ }\ar[dr]_{a_\lambda}&
\Gr_{\leq\mu}(E_\kappa)
\ar[d]^{pr_1}
\\&\Bun_G}\]
is commutative.
Moreover, for any $k$-point $\tilde x\in\tU(k)$,
$\tilde a_{\lambda}(\tilde x)\in\Gr_{\mu}(E_\kappa)(k)$.
\end{lemma}

The proof is given at the end of this subsection.
We also need the following lemma about the differential of $\tilde a_\lambda$.
\begin{lemma}
Let $\tilde x\in\tU(k)$, and let $\tilde a_\lambda(\tilde x)=(x,E_{\tilde x},\beta)\in\Gr_\mu(E_\kappa)(k)$ (by Lemma \ref{Hk}).
For every $v\in T_bB\subset T_{\tilde x}\tC=T_xC\oplus T_bB$, we have 
$$u:=(\tilde a_\lambda)_*v\in T_{(E_{\tilde x},\beta)}\Gr_{x,\mu}(E_\kappa)\subset 
T_{(x,E_{\tilde x},\beta)}\Gr_\mu(E_\kappa).$$
\end{lemma}
\begin{proof}
The subspace $T_{(E_{\tilde x},\beta)}\Gr_{x,\mu}(E_\kappa)\subset 
T_{(x,E_{\tilde x},\beta)}\Gr_\mu(E_\kappa)$ 
is equal to 
$$\on{Ker}((pr_2)_*:T_{(x,E_{\tilde x},\beta)}\Gr_{\mu}(E_\kappa)
\ra T_xC).$$ Therefore
it is enough to show $(pr_2)_*(\tilde a_\lambda)_*v=0$. 
Recall that we have 
the 
following 
commutative diagram (not cartesian)
\[\xymatrix{\tC\ar[r]^{\tilde a_\lambda\ \ \ }\ar[d]^\pi&\Gr_{\leq\mu}(E_\kappa)\ar[d]^{pr_2}
\\C\times B\ar[r]^{pr_C}&C}.\]
Thus we have $(pr_2)_*(\tilde a_\lambda)_*v=
(pr_C)_*(\pi_*v)=(pr_C)_*v=0$. This finishes the proof.
\end{proof}

Combining the above two lemmas  we obtain that
\begin{equation}\label{main equ 2}
\langle\theta_{\kappa,\lambda},v\rangle=\langle\phi_{\tilde x},{a_\lambda}_*v\rangle=\langle\phi_{\tilde x},(pr_1)_*u\rangle
\end{equation} 
where $u:=
(\tilde a_\lambda)_*v\in T_{(E_{\tilde x},\beta)}\Gr_{x,\mu}(E_\kappa)$.
So
we need show that the last pairing is zero. To calculate it, 
we need a few more notations. For any $x\in C$
we denote by $\mO_x$ the 
completion of the local ring of $C$ at $x$ and $F_x$ its fractional field. Let $\omega_{\mO_x}$ (resp. $\omega_{F_x}$) denote the completed regular (resp. rational) differentials on $\Spec \mO_x$. 
We denote by 
$$\Res(,):\fg(\omega_{F_x})\times
\fg(F_x)\ra k$$ the residue pairing induced by the $G$-invariant form
$(,)$ on $\fg$.  

Let us fix a trivialization  $\gamma_\kappa:E_\kappa\simeq E^0$ on $\Spec \mO_x$. Then, for every trivialization
$\gamma$ of $E$ on $\Spec \mO_x$, we obtain 
\[g= \gamma_\kappa^{-1} \beta \gamma \in G(F_x).\]
In this way, $\gamma_\kappa$ induces an isomorphism 
$$\Gr_{x,\mu}(E_\kappa)\is\on{Orb}_\mu,\quad (E,\beta)\mapsto \gamma_\kappa^{-1}\beta \gamma G(\mO_x),$$ where $\on{Orb}_\mu$
is the $G(\mO_x)$-orbit of 
$\mu\cdot G(\mO_x)\in G(F_x)/G(\mO_x)$. 
Under the isomorphism, we have the identification of  
the tangent 
spaces 
\[T_{(E,\beta)}\Gr_{x,\mu}(E_\kappa)\simeq\fg(\mO_x)/(\Ad_{g}\fg(\mO_x)\cap\fg(\mO_x)).\] 
For any $u\in T_{(E,\beta)}\Gr_{x,\mu}(E_\kappa)$ and $\phi\in T_E\Bun_G$
the pairing $\langle\phi,(pr_1)_*u\rangle$ can be calculated as follows. Let $\tilde u\in\fg(\mO_x)$ be a lifting of 
$u$ under the above isomorphism. Let $\phi(\gamma)$ denote the $\phi:\Spec \mO_x\to \ad E\otimes \omega_C\stackrel{\gamma}{\to} \fg(\omega_{F_x})$.
Now we have
\[\langle\phi,(pr_1)_*u\rangle=\Res (\phi(\gamma),\Ad_{g}^{-1}\tilde u),\]
In our case $\phi=\phi_{\tilde x}=\AJ^{\sP,\lambda}(\tilde x)\times^J\phi_\kappa$,
the following lemma  will imply the vanishing of 
$\langle\phi_{\tilde x},(pr_1)_*u\rangle$, and  
therefore will finish the proof of Proposition \ref{phi=0}.
\begin{lemma}\label{phi}
We have $\Ad_g\phi_{\tilde x}(\gamma)\in \fg(\omega_{\mO_x})$.
\end{lemma}
\begin{proof}
Indeed, unraveling the definitions, we have
$\Ad_{g}\phi(\gamma)=\phi_\kappa(\gamma_\kappa)$, which is regular.
\end{proof}

It remains to prove Lemma \ref{Hk}.
Let $\tilde a_\lambda:\tC\ra\Gr(E_\kappa)$ be the morphism 
constructed as in the lemma. 
Since $\tC$ is smooth and $\tU\subset\tC$ 
is open dense, it is enough to show that
 $\tilde a_\lambda(\tU(k))\subset\Gr_{\mu}(E_\kappa)(k)$.  Let $\tilde x\in\tU(k)$ and 
 $\tilde a_\lambda(\tilde x)=(x,E_{\tilde x},\beta)\in\Gr(E_\kappa)(k)$
 be its image, where $E_{\tilde x}:=\AJ^{\sP,\lambda}(\tilde x)\times^JE_\kappa$ and 
$\beta:E_\kappa|_{C-x}\is E_{\tilde x}|_{C-x}$ is the isomorphism 
induced by the canonical section $s\in\AJ^{\sP,\lambda}(\tilde x)(C-x)$.
Let $$rel:\Gr_x(E_\kappa)\ra\xcoch/\rW$$ be the map sending an element $(E,\beta)$ to the relative position of $\beta$  
(cf. \cite[\S 5.2.2]{BD}).
We have $(E_{\tilde x},\beta)\in\Gr_x(E_\kappa)$ and we need to show that $rel(E_{\tilde x},\beta)=\mu$.  
For simplicity, we will denote $\mP:=\AJ^{\sP,\lambda}(\tilde x)$.

Let $\Gr_{J}$ (resp. $\Gr_T$) be the global Grassmannian for the
group scheme $J$ (resp. $T$). By \cite[Lemma 3.2.5]{Yun}, the 
morphism $j^1:\pi^*J\ra T\times\tC$ induces a $\rW$-equivariant isomorphism 
$$j_{\Gr}:\Gr_J\times_{(C\times B)}\tU\is\Gr_T\times_C\tU$$
of group ind-schemes
over $\tU$.  We denote by $j_{\tilde x,\Gr}:\Gr_{x,J_b}\is\Gr_{x,T}$ the restriction of
$j_\Gr$ to $\tilde x$. We have 
$(\mP,s)\in\Gr_{x,J_b}(k)$ (here $s\in\mP(C-x)$ is the canonical section) and
one can check that 
$j_{\tilde x,\Gr}(\mP,s)=\lambda\in\Gr_{x,T}(k)\is\xcoch$.
The action of $\Gr_{x,J_b}$ on $(E_\kappa,\phi_\kappa)$
defines a map $a_\kappa:\Gr_{x,J_b}\ra\Gr_x(E_\kappa)$. 
We claim that the following 
diagram is commutative
\begin{equation}\label{claim}
\xymatrix{\Gr_{x,J_b}(k)\ar[r]^{a_\kappa}\ar[d]^{j_{\tilde x,\Gr}}&\Gr_{x}(E_\kappa)(k)
\ar[d]^{rel}
\\\Gr_{x,T}(k)\is\xcoch\ar[r]^{}&\xcoch/\rW}.
\end{equation}
Assuming the claim we see that $rel(E,\beta)=rel(a_\kappa(\mP,s))$ is equal to the image of 
$j_{\tilde x,\Gr}(\mP,s)=\lambda\in\xcoch$
in $\xcoch/\rW$. But by assumption $\lambda\in\rW\cdot\mu$.
This finishes the proof of Lemma \ref{Hk}.

To prove the claim, recall that a trivialization $\gamma_\kappa$ of $E_\kappa$ on $\Spec \mO_x$
defines an isomorphism $\Gr_{x}(E_\kappa)\is G(F_x)/G(\mO_x)$. Moreover,
under the canonical isomorphism $\Gr_{x,J_b}(k)\is J_b(F_x)/J_b(\mO_x)$,
$\Gr_{x,T}(k)\is T(F_x)/T(\mO_x)$ and 
$G(\mO_x)\backslash G(F_x)/G(\mO_x)=\xcoch/\rW$, the diagram (\ref{claim})
can be identified with
\[
\xymatrix{J_b(F_x)/J_b(\mO_x)\ar[r]^{}\ar[d]^{}&
G(F_x)/G(\mO_x)
\ar[d]^{pr}
\\T(F_x)/T(\mO_x)\ar[r]^{}&G(\mO_x)\backslash G(F_x)/G(\mO_x),}\]
where the the upper arrow is induced by the homomorphism
\begin{equation}\label{upper}
J_b\stackrel{a_{E_\kappa,\phi_\kappa}}\is\on{Aut}(E_\kappa,\phi_\kappa)\ra\on{Aut}
(E_\kappa)\stackrel{\gamma_\kappa}\is G
\end{equation}
and the arrow in the left column 
is induced by the homomorphism $j^1:\pi^*J\ra T\times\tC$. 
Let $b_x\in\fc^{rs}(\mO_x)$ be the restriction of $b$ to $\Spec\mO_x$.
Using the definition of $a_{E_\kappa,\phi_\kappa}$ in (\ref{actEphi}),
it is not hard to see that the restriction of (\ref{upper}) to $\Spec\mO_x$ is
equal to 
\begin{equation}\label{equ 1}
J_{b_x}\is I_{kos(b_x)}\hookrightarrow G\times\Spec\mO_x,
\end{equation}
up to conjugation by an element in $G(\mO_x)$. Here $kos(b_x):\Spec\mO_x\stackrel{b_x}\ra\fc\stackrel{kos}\ra\fg\in\fg^{reg}(\mO_x)$
and the first isomorphism is induced by the canonical isomorphism $\chi^*J|_{\fg^{reg}}\is I|_{\fg^{reg}}$ in Proposition \ref{J}.
Therefore, to prove the claim, it is enough to show that the restriction of $j^1$ to $\Spec\mO_x$\footnote{Here we 
identify $\Spec\mO_{\tilde x}\is\Spec\mO_x$ and 
regard $j^1$ as a map of group schemes over $\Spec\mO_x$.} is 
is equal to the map (\ref{equ 1}) up to left and right multiplication 
by elements in $G(\mO_x)$. 
To see this, we 
first observe that the point $\tilde x$ defines a lifting 
$\tilde b_x\in\ft^{rs}(\mO_{x})$ of $b_x\in\fc^{rs}(\mO_x)$.
Since the map $G\times\fg^{rs}\ra\fg^{rs}\times_{\fc}\fg^{rs},\ \ (y,v)\ra(\Ad y(v),v)$ 
is smooth, there exists $g\in G(\mO_x)$ such that $\Ad g(kos(b_x))=\tilde b_x$.
The map $\Ad g$ induced an isomorphism $\iota_g:I_{kos(b_x)}\is I_{\tilde b_x}=T\times\Spec\mO_x$, which is independent of the choice of $g$, and 
according to \cite[Proposition 2.4.2]{N2}  
the restriction of $j^1$ to $\Spec\mO_x$ is given by
\begin{equation}\label{equ 2}
J_{b_x}\is I_{kos(b_x)}\stackrel{\iota_g}\is T\times\Spec\mO_x,
\end{equation}
where the first map is the canonical isomorphism mentioned before.
The above description implies the map (\ref{equ 1}) is equal to the map (\ref{equ 2}) up to left and right multiplication 
by elements in $G(\mO_x)$. This finishes the proof of the claim.



\section{Main result}\label{Main}
We assume that $G$ is semi-simple over $k$ whose characteristic $p$ is positive and does not divide the order of the Weyl group of $G$. Let $C$ be a smooth projective curve over $k$, of genus at least two.
In this case, $\Bun_G$ is a ``good" stack 
in the sense of \cite[\S 1.1.1]{BD} (see also \ref{azumaya property}). Let $D_{\Bun_G}$ be the sheaf of algebras on $\on{Higgs}'_G$ in
Proposition \ref{Azumaya}. Denote by $D_{\Bun_G}^0:=D_{\Bun_G}|_{\on{Higgs}_G'\times_{B'}B^{'0}}$ the restriction of $D_{\Bun_G}$ to the smooth part of the 
Hitchin fibration.
We define $\mD\on{-mod}(\Bun_G)^0$ as the category of 
$D_{\Bun_G}^0$-modules. As explained in \ref{azumaya property}, the category  
$\mD\on{-mod}(\Bun_G)^0$ 
is a localization of the category of $\mD$-modules on $\Bun_G$ and
is canonical equivalent to the 
category of twisted sheaves $\on{QCoh}(\sD_{\Bun_G}^0)_1$,
where $\sD_{\Bun_G}^0=\sD_{\Bun_G}\times_{B'}{B^{'0}}$ and $\sD_{\Bun_G}$ is the gerbe 
of crystalline differential operators on $\on{Higgs}_G'$.
On the dual side, let $\Loc_{\breve G}$ be the stack of de Rham $\breve G$-local systems on $C$.
Recall that in \cite{CZ}, we constructed a fibration \[h_p:\Loc_{\breve G}\ra B'\]
from $\Loc_{\breve G}$ to the Hitchin base $B'$, which can be regraded as a 
deformation of the usual Hitchin fibration.
We define \[\Loc_{\breve G}^0:=\Loc_{\breve G}\times_{B'}B^{'0}.\]
Our goal is to prove the following Theorem:

\begin{thm}\label{main}\label{D_kappa}
Assume $G$ is semi-simple and the genus of $C$ is at least two.
For a choice of a square root $\kappa$ of $\omega_C$, we have a 
canonical equivalence of bounded derived categories
$$\frakD_\kappa:D^b(\mD\on{-mod}(\Bun_G)^0)\is D^b(\on{QCoh}(\Loc_{\breve G}^{0}))$$
\end{thm}

The proof of above theorem is divided into two steps. The first step, 
which involves the Langlands duality, 
is to establish a twisted version of 
the classical duality (see \S\ref{tw}). The second step, 
which does not involve the Langlands duality,
is to establish two abelianization theorems (see \S\ref{ab thm}) for which
we need a choice of square root $\kappa$ of $\omega_C$.
Combining above two steps, our main theorem follows from a general version of the Fourier-Mukai
transform (see \S\ref{final}).

\subsection{The $\breve\sP'$-torsor $\breve\sH$.}
We first recall that in \cite{CZ}, we constructed a $\breve\sP'$-torsor $\breve\sH$. It is defined via the following Cartesian diagram
\begin{equation}\label{Har}
\begin{CD}
\breve\sH @>>>\Loc_{\breve J^p}\\
@VVV@VVV\\
B'@>\breve \tau'>>B'_{\breve J'}
\end{CD}
\end{equation}
Here $\breve J^p$ is the pullback of the universal centralizer $\breve J'$ over $C'\times B'$ along 
the relative frobenius map $F_{C'\times B'/B'}:C\times B'\to C'\times B'$. This is a group scheme with a canonical connection along $C$, and therefore it makes sense to talk about the stack $\Loc_{\breve J^p}$ of $\breve J^p$-torsors with flat connections. In addition, it admits a map $\Loc_{\breve J^p}\to B'_{\breve J'}$. We refer to \cite[Appendix]{CZ} for generalities.

Recall that there is a description of $\sP$ in terms of  $\Bun_T^\rW(\tC/B)$.  We give a similar description of the $\breve\sP'$-torsor $\breve\sH$ in terms of a $\Bun_{\breve T}^\rW(\tC/B)'$-torsor. Recall that $\breve\tau'$ is regarded as a section of $\Omega_{C'\times B'/B'}\otimes\Lie\breve J'$, which defines a $\breve J'$-gerbe $\sD(\breve\tau')$ on $C'\times B'$ (see \ref{one forms}) and according to \cite[Proposition A.9]{CZ},
$\breve\sH$ is isomorphic to $\sT_{\sD(\breve\tau')}$, the stack of splittings of $\sD(\breve\tau')$ over 
$B'$. Therefore by Lemma \ref{T=T^+} we have 
\begin{equation}\label{H=T^+}
\breve\sH|_{B^{'0}}\is\sT_{\sD(\breve\tau')_{\breve T}}^{\rW,+}|_{B^{'0}},
\end{equation}
where $\sD(\breve\tau')_{\breve T}:=(\pi^*\sD(\breve\tau'))^{\breve j^1}$ is the 
$\breve T$-gerbe on $\tC'$ induced from $\sD(\breve\tau')$ using maps
$\pi:\tC'\ra C'\times B'$ and $\breve j^{1}:\pi^*\breve J'\ra\breve T'\times\tC'$
(see \ref{induction functor} for the induction functor of gerbes).

On the other hand, using the definition of 
$\theta_{\tC'}\in\Gamma(\tC',\Omega_{\tC'}\otimes\breve\ft')^\rW$ in \S\ref{Lie algebra valued one forms} one can check that 
$\breve j^1_*\pi^*(\breve\tau')=\theta_{\tC'}$, where $\breve j^1_*\pi^*(\breve\tau')$ is the $\breve\ft'$-valued 1-form 
induced from $\breve\tau'$ using maps $\pi$ and $\breve j^1$.
Therefore, by
Lemma \ref{theta} we see that over $B^{'0}$ we have
\begin{equation}\label{ind=theta}
\sD({\breve\tau'})_{\breve T}:=(\pi^*\sD(\breve\tau'))^{\breve j^1}\is\sD(\breve j^1_*\pi^*(\breve\tau'))\is\sD(\theta_{\tC'}).\end{equation}
Hence combining (\ref{H=T^+}) and (\ref{ind=theta}) we get the following Galois description of $\breve\sH$.
\begin{corollary}\label{Galois description of Loc_J}
There is a canonical isomorphism of $\breve\sP'$-torsors 
$\breve\sH|_{B^{'0}}\is\sT^{\rW,+}_{\sD(\theta_{\tC'})}|_{B^{'0}}$.
\end{corollary}

\subsection{Twisted duality}\label{tw}
Let us construct the twisted duality.
Let $\theta'_m: \sP'\to T^*\sP'$ denote the canonical multiplicative one form constructed in \S\ref{canonical one forms}. Let
$\sD(\theta'_m)$ denote the corresponding $\bG_m$-gerbe on $\sP'$ obtained by pullback of $\sD_{\sP}$ on $T^*\sP'$ by $\theta_m'$ (see \ref{one forms}). According to \ref{appen:var}, the gerbe $\sD(\theta_m')$ is canonically 
multiplicative. Moreover,
according to \ref{DFT}, the stack of multiplicative splittings of $\sD(\theta'_m)$ over $B'$ is a $(\sP')^\vee$-torsor $\sT_{\sD(\theta'_m)}$. 
Our goal is to prove the following theorem.
\begin{thm}\label{tw duality}
There is a canonical isomorphism of 
$\sP'^\vee\is\breve\sP'$-torsors
\[\frakD:\sT_{\sD(\theta'_m)}|_{B^{'0}}\is\breve\sH|_{B^{'0}}.\]
\end{thm}
For the rest of this subsection we will restrict everything to $B^{'0}$.
Recall the Abel-Jacobi map $\AJ^{\sP'}:\tC'\times\xcoch\ra\sP'$.
By Proposition \ref{kappa=tC} we have $(\AJ^{\sP'})^*\theta_m'=\theta_{\tC'}$. Therefore,
Lemma \ref{theta} implies that
\[(\AJ^{\sP'})^*\sD(\theta'_m)=\sD(\theta_{\tC'}).\] 
Since the 
Abel-Jacobi map $\AJ^{\sP'}$ is $\rW$-equivariant, pullback via $\AJ^{\sP'}$
defines a functor \[\tilde\frakD:\sT_{\sD(\theta'_m)}\ra\sT^\rW_{\sD(\theta_{\tC'})}.\]
We claim that $\tilde\frakD$ canonically lifts to a morphism 
$$\frakD:
\sT_{\sD(\theta_m')}\ra
\sT^{\rW,+}_{\sD(\theta_{\tC'})}\is\breve\sH,
$$
where the second isomorphism is Corollary \ref{Galois description of Loc_J}.
Let $E\in\sT_{\sD(\theta_m')}$ be a tensor splitting of $\sD(\theta_m')$.
We need to show that the splitting
\[(\tilde\frakD(E))^\alpha|_{\tC'_{\alpha}}=(\AJ^{\sP'})^*E|_{(\tC'_{\alpha},\breve\alpha)}\]
admits a canonical isomorphism compatible with the canonical splitting $E^0_\alpha$ of 
$\sD(\theta_{\tC'})^\alpha|_{\tC'_\alpha}=(\AJ^{\sP'})^*\sD(\theta'_m)|_{(\tC'_{\alpha},\breve\alpha)}.$
However, this follows from the fact that $\AJ^{\sP'}((x,\breve\alpha))$ is the unit of $\sP'$ for $x\in\tC'_\alpha$ and a tensor splitting $E$ of a multiplicative $\bG_m$-gerbe 
$\sD(\theta_m')$ is canonically isomorphic to the canonical splitting $E^0_\alpha$
of $\sD(\theta_m')$ over the unit. To summarize, we have 
constructed the following 
commutative diagram
\begin{equation}\xymatrix{
\sT_{\sD(\theta_m')}\ar^{\frakD}[rr]\ar_{\tilde\frakD}[dr]&&
\breve\sH\ar^{For}[dl]\\
& \sT^{\rW}_{\sD(\theta_{\tC'})}}\end{equation}
By construction, the morphism $\frakD$ is compatible with the 
$\sP'^\vee\is\breve\sP'$-action, hence 
is an equivalence. This finishes the proof of Theorem \ref{tw duality}.

\subsection{Abelianization Theorems}\label{ab thm}
We need to fix a square root $\kappa$ of $\omega_C$. Then the Kostant 
section for $\on{Higgs}_G'\ra B'$ induces a map 
$$\epsilon_{\kappa'}:\sP'\is{\on{Higgs}_G'}^{reg}\subset\on{Higgs}_G',$$
where ${\on{Higgs}_G'}^{reg}$ is the smooth sub-stack consisting of regular Higgs fields. 
The first abelianization theorem is the following.
\begin{thm}\label{Ab1}
We have a canonical isomorphism 
$\epsilon_{\kappa'}^*\sD_{\Bun_G}\is\sD(\theta_m')$, where $\sD_{\Bun_G}$ is the 
$\bG_m$-gerbe on $\on{Higgs}_G'$ of crystalline differential operators.
Moreover, the pullback along the map $\epsilon_{\kappa'}$ defines an equivalence of categories
of twisted sheaves 
\[\frakA_\kappa:D^b(\mD\on{-mod}(\Bun_G^0))\is D^b(\on{QCoh}(\sD_{\Bun_G}^0))_1
\stackrel{\epsilon_{\kappa'}^*}\is 
D^b(\on{QCoh}(\sD(\theta_m')|_{B^{'0}}))_1.\]
\end{thm}
\begin{proof}
By Proposition \ref{BB lemma}, the restriction of $\sD_{\Bun_G}$ to ${\on{Higgs}_G'}^{reg}$ is isomorphic to the gerbe
$\sD(\theta_{can}')$ defined by the canonical 1-form $\theta_{can}'$ on ${\on{Higgs}_G'}^{reg}$.
On the other hand, it follows from the construction of $\theta_m$  in \S\ref{canonical one forms} that we have $\epsilon_{\kappa'}^*\theta_{can}'=\theta_m'$. Hence
\[\epsilon_{\kappa'}^*\sD_{\Bun_G}\is\epsilon_{\kappa'}^*\sD(\theta_{can}')\is\sD(\epsilon_{\kappa'}^*\theta_{can})\is\sD(\theta_m').\]
The last statement follows from the fact that the base change of $\epsilon_{\kappa'}:\sP'\ra\on{Higgs}_G'$ to 
$B^{'0}$
is an isomorphism (see Proposition \ref{reg torsor}).
\end{proof}
To state the second abelianization theorem, recall that in \cite{CZ} we constructed a canonical isomorphism
\[\frakC:\breve\sH\times^{\breve\sP'}\on{Higgs}_{\breve G}'\is\Loc_{\breve G}.\]
Moreover, by \cite[Remark 3.14]{CZ} the choice of the theta characteristic
$\kappa$ defines an isomorphism $\epsilon_{\kappa'}:\sP'\is{\on{Higgs}_G'}^{reg}$,
and hence induces 
an isomorphism 
\[\frak C_\kappa:\breve\sH\is\Loc_{\breve G}^{reg}\subset\Loc_{\breve G}.\]
Here $\Loc_{\breve G}^{reg}$ is the open substack consisting of 
$\breve G$-local systems with regular $p$-curvature, and we have
$\Loc_{\breve G}^{reg}|_{B^{'0}}=\Loc_{\breve G}^0$.
It implies:

\begin{thm}\label{Ab2}
For each choice of a square root $\kappa$  of $\omega_C$, we have 
a canonical isomorphism of $\breve\sP'$-torsors
 $\frak C_\kappa|_{B^{'0}}:\breve\sH|_{B^{'0}}\is\Loc_{\breve G}^{0}$ and it induces
an equivalence of categories
\[\frakC_{\kappa}^*:D^b(\on{QCoh}(\Loc_{\breve G}^0))\is
D^b(\on{QCoh}(\breve\sH|_{B^{'0}})).\]
\end{thm}
\subsection{Proof of Theorem \ref{main}}\label{final}
We deduce our main theorem from the twisted duality and above two abelianization theorems.
By the twisted duality we have an isomorphism of $\sP'^\vee\is\breve\sP'$-torsors
$\sT_{\sD(\theta_m')}|_{B^{'0}}\is\breve\sH|_{B^{'0}}$. Therefore  
the twisted Fourier-Mukai transform (Theorem \ref{twist FM}) 
implies an equivalence of categories
\[\frakD:D^b(\on{QCoh}(\sD(\theta_m')|_{B^{'0}}))_1\is D^b(\on{QCoh}(\breve\sH|_{B^{'0}})).\]
Now combining Theorem \ref{Ab1} and Theorem \ref{Ab2} we get the desired equivalence 
\[\frakD_\kappa=(\frakC_\kappa^*)^{-1}\circ\frakD\circ\frakA_\kappa:
D^b(\mD\on{-mod}(\Bun_G^0))\is D^b(\on{QCoh}(\Loc_{\breve G}^0)).\]

\subsection{A $\mu_2$-gerbe of equivalences}\label{a_chi and b_chi}
In this subsection we study how those equivalences $\frak D_\kappa:D^b(\mD\on{-mod}(\Bun_G)^0)\is D^b(\on{QCoh}(\Loc_{\breve G}^{0}))$ in Theorem \ref{main} depend on the 
choices of the theta characteristics $\kappa$. Our discussion is very similar to \cite{FW} and can be regarded as a verification of the predictions of \cite{FW} in our setting.

Let $\omega^{1/2}(C)$ be the groupoid of square roots of $\omega_C$. The groupoid
$\omega^{1/2}(C)$ is a torsor over the Picard category $\mu_2\on{-tors}(C)$ of $\mu_2$-torsors on $C$.
Let $\mathbf{GLC}$ be the groupoid of equivalences between 
$D^b(\mD\on{-mod}(\Bun_G)^0)$ and $D^b(\on{QCoh}(\Loc_{\breve G}^{0}))$, i.e.
objects in $\mathbf{GLC}$ are equivalences 
$\mathbf E:D^b(\mD\on{-mod}(\Bun_G)^0)\is D^b(\on{QCoh}(\Loc_{\breve G}^{0}))$ and 
morphisms are isomorphisms between equivalences.
We first construct an action of $\mu_2\on{-tors}(C)$ on $\mathbf{GLC}$.

\medskip

Let $Z=Z(G)$ be the center of $G$. 
We have a map $\alpha:\mu_2\ra Z(G)$ 
by restricting the cocharacter $2\rho:\bG_m\ra G$ to $\mu_2$
(see \cite[\S 3.4.2]{BD}). Thus for each 
$\chi\in\mu_2\on{-tors}(C)$ and 
$(E,\nabla)\in\Loc_{ G}$ we can
twist $(E,\nabla)$ by $\chi$ using the map 
\[\mu_2\ra\ Z\ra\Aut(E,\nabla)\] 
to
get a new $G$-local system $(E\otimes\chi,\nabla_{E\otimes\chi})
\in\Loc_{G}$. 
The assignment $(\chi,E,\nabla)\ra (E\otimes\chi,\nabla_{E\otimes\chi})$
defines a geometric action 
\[act_G:\mu_2\on{-tors}(C)\times\Loc_G\ra\Loc_G.\]
Likewise, there is $act_G:\mu_2\on{-tors}(C)\times\Bun_G\to\Bun_G$. For $\chi\in\mu_2\on{-tors}(C)$,
let $a_{\chi,G}:\Bun_G\is\Bun_G$ (resp, $b_{\chi,G}:\Loc_{G}\is\Loc_G$ )
be the automorphisms 
of $\Bun_G$ (resp. $\Loc_G$) given by $a_{\chi,G}(E):=E\otimes\chi$, (resp.
$b_{\chi,G}(E,\nabla)=act_G(\chi,E,\nabla)$).
They induce auto-equivalences $a_{\chi,G}^*$ and $b_{\chi,G}^*$ 
 of $D^b(\mD\on{-mod}(\Bun_G))$ 
and 
$D^b(\on{QCoh}(\Loc_G))$ respectively. Note that for the definition of $a_{\chi,G}^*$ and $b_{\chi,G}^*$, there is no restriction of the characteristic of $k$. However, if $\cha k=p\nmid |\rW|$, we have

\begin{lemma}
1)
The equivalence $a_{\chi,G}^*$ 
preserves the full subcategory $D^b(\mD\on{-mod}(\Bun_G)^0)$.
\
2) The equivalence $b_{\chi,G}^*$ preserves
the full subcategory $D^b(\on{QCoh}(\Loc_G^0))$.
\end{lemma}
\begin{proof}
This lemma will be clear after we give alternative descriptions of $a_{\chi,G}^*$ and $b_{\chi,G}^*$. 

First, recall that in \S \ref{gerbe} we introduce a $\bG_m$-gerbe $\sD_{\Bun_G}$
over $T^*\Bun_G'$ and the category $\on{QCoh}(\sD_{\Bun_G})_1$ of twisted sheaves on
$\sD_{\Bun_G}$ such that there is an equivalence of categories between
$\mD\on{-mod}(\Bun_G)$ and $\on{QCoh}(\sD_{\Bun_G})_1$.
Let 
$f:=da_{\chi,G}':T^*\Bun_G'\is T^*\Bun_G'$
be the differential of $a_{\chi,G}'$.
The map $f$ preserves the canonical one form $\theta_{can}'$, and thus 
by Lemma \ref{theta}, there is 
a canonical 1-morphism $M:f^*\sD_{\Bun_G}\sim\sD_{\Bun_G}$ 
of gerbes on $T^*\Bun_G'$. The 1-morphism $M$ induces an equivalence 
$M:\on{QCoh}(f^*\sD_{\Bun_G})_1\is\on{QCoh}(\sD_{\Bun_G})_1$ and it follows from definitions
that the functor $a_{\chi,G}^*$ is isomorphic to the composition 
\begin{equation}\label{morita}
D^b(\on{QCoh}(\sD_{\Bun_G})_1)\stackrel{\ f^*}\is
D^b(\on{QCoh}(f^*\sD_{\Bun_G})_1)\stackrel{M}\is D^b(\on{QCoh}(\sD_{\Bun_G})_1).
\end{equation}
Recall that the category $\mD\on{-mod}(\Bun_G)^0$ is by definition the category of 
twisted sheaves on $\sD_{\Bun_G}^0=\sD_{\Bun_G}|_{B^{'0}}$. Therefore, Part 1) follows.

To prove Part 2), note that the map $act_G:\mu_2\on{-tors}(C)\times\Loc_G\ra\Loc_G$ can be also described as follows.
There is a map of group schemes 
$(\mu_2)_{C'\times B'}\ra Z(G)_{C'\times B'}\ra J'$
over $C'\times B'$, which induces a morphism of Picard stacks 
\begin{equation}\label{l_mu}
\fl_{\mu_2}:\mu_2\on{-tors}(C)\times B'\ra\sP',
\end{equation}
and the action map $act_G$ can be identified with
\begin{equation}\label{act loc}
act_G:\mu_2\on{-tors}(C)\times\Loc_G\stackrel{\fl_{\mu_2}\times id}\ra\sP'\times_{B'}\Loc_G\ra\Loc_G
\end{equation}
where the last map is the action of $\sP'$ on $\Loc_G$ defined in 
\cite[Proposition 3.5]{CZ}. In particular, if we endow $B'$ with the trivial $:\mu_2\on{-tors}(C)$ action, the $p$-Hitchin map $\Loc_G\to B'$ is $:\mu_2\on{-tors}(C)$-equivariant. Therefore $\Loc_G^0$ is invariant under the action of $b_{\chi,G}$, and Part 2) follows.
\end{proof}

From now on we regard $a_{\chi,G}^*$ and $b_{\chi,G}^*$ as automorphisms of 
the category $D^b(\mD\on{-mod}(\Bun_G)^0)$ and $D^b(\on{Qcoh}(\Loc_G^0))$.

For each $\chi\in\mu_2\on{tors}(C)$ and 
$\mathbf E\in\mathbf{GLC}$ we define
$$\chi\cdot\mathbf E:=b_{\chi,\breve G}^*\circ\mathbf E\circ a_{\chi,G}^*\in\mathbf{GLC}.$$ The following 
lemma follows from the construction of 
$b_{\chi,\breve G}^*$ and 
$a_{\chi,G}^*$ .
\begin{lemma} 
The functor $\mu_2\on{-tors}(C)\times\mathbf{GLC}\ra\mathbf{GLC}$
given by $(\chi,\mathbf E)\ra\chi\cdot\mathbf E$
defines an action of the Picard category $\mu_2\on{-tors}(C)$ on 
 $\mathbf{GLC}$.
\end{lemma}

\medskip

Now
let $\mathscr C_1$ and $\mathscr C_2$ be two categories acted by a 
Picard category $\sG$.
A $\sG$-module functor from $\mathscr C_1$ to $\mathscr C_2$ is a 
functor 
$N:\mathscr C_1\ra\mathscr C_2$ 
equipped with functorial isomorphisms $N(a\cdot c)\is a\cdot N(c)$
satisfying the natural compatibility condition.
Here is the main result of this subsection
\begin{prop}\label{mu_2 gerbe}
The assignment $\kappa\ra\frak D_\kappa$
defines a $\mu_2\on{-tors}(C)$-module functor \[\Phi:\omega^{1/2}(C)\ra\mathbf{GLC}.\]
\end{prop}
\begin{proof}
Given $\chi\in\mu_2\on{-tors}(C)$ and
 $\kappa\in\omega^{1/2}(C)$ we need to specify a functorial isomorphism 
 $\frak D_{\chi\cdot\kappa}\is\chi\cdot \frak D_\kappa$ 
 satisfying the natural compatibility condition. First, observe that the maps $\epsilon_{\kappa'},\epsilon_{\kappa'_1}:\sP'\ra\on{Higgs}'_G$
induced by $\kappa,\kappa_1:=\chi\cdot\kappa\in\mu_2\on{-tors}(C)$ 
differ by a translation of the section $\fl_{\mu_2}(\{\chi\}\times B')\in\sP'(B')$, where $\fl_{\mu_2}$ is the map in (\ref{l_mu}). Then it follows from the construction of $\frakA_\kappa$ and $\frakC_\kappa$ in \S\ref{ab thm}
 that there are canonical functorial isomorphisms
$\frakA_{\chi\cdot\kappa}\is\frakA_\kappa\circ a^*_{\chi,G}$ and
$\frakC_{\kappa}^*\circ b^*_{\chi,\breve G}\is\frakC_{\chi\cdot\kappa}^*$.
Therefore we get a functorial isomorphism 
\[\frak D_{\chi\cdot\kappa}=(\frakC_{\chi\cdot\kappa}^*)^{-1}\circ\frak D\circ\frakA_{\chi\cdot\kappa}\is b^*_{\chi,\breve G}\circ(\frakC_{\kappa}^*)^{-1}\circ\frak D\circ\frakA_{\kappa}\circ a^*_{\chi,G}=\chi\cdot\frakD_\kappa,\]
and one can check that it satisfies the natural compatibility condition. 
\end{proof}
\begin{remark}
The above construction suggests  that the geometric Langlands correspondence should be a $\mu_2\on{-gerbe}$ of equivalences between 
$D^b(\mD\on{-mod}(\Bun_G))$ and  $D^b(\on{QCoh}(\Loc_{\breve G}))$.
This gerbe is trivial, but is not canonically trivialized. One obtains a particular trivialization of this gerbe, and hence a particular equivalence $\frak D_\kappa$, for each choice of a square root of the canonical line bundle of $C$.  A similar $\mu_2$-gerbe also appears in the work of Frenkel and Witten \cite[\S 10]{FW},
where the geometric Langlands correspondence is interpreted as gauge theory duality between the twisted $A$-model of $\on{Higgs}_G$ and
the twisted $B$-model of $\on{Higgs}_{\breve G}$.
\end{remark}

\subsection{The actions $a_{\chi,G}^*$ and 
$b_{\chi,G}^*$ as tensoring of line bundles}\label{tensoring action}
In this subsection we show that, under the equivalence 
$\mathfrak D_\kappa$, 
the geometric actions $a_{\chi,G}^*$ and 
$b_{\chi,G}^*$ constructed in previous subsection become 
functors of tensoring with certain line bundles. 

Recall that in \S \ref{l_J} we associated to every $Z(\breve G)$-torsor $K$ on $C$ a line bundle $\mL_{G,K}$ on $\Bun_G$. 
For any $\chi\in\mu_2\on{-tors}(C)$ 
let $K_{G,\chi}:=\chi\times^{\mu_2}Z_{G}\in Z(G)\on{-tors}(C)$
be the induced $Z(G)$-torsor  via the canonical map $2\rho:\mu_2\ra Z(G)$.
We denote by 
$\mL_{G,\chi}$ and $\mL_{\breve G,\chi}$ 
be the line bundles on $\Bun_G$ and $\Bun_{\breve G}$
corresponding to $K_{\breve G,\chi}$ and $K_{G,\chi}$. Since the line bundle $\mL_{G,\chi}$ carries a 
canonical connection with zero $p$-curvature, 
 tensoring with $\mL_{G,\chi}$
defines an auto-equivalence $\mL_{G,\chi}\otimes ?$ of $D^b(\mD\on{-mod}(\Bun_G)^0)$.

For any $\kappa\in\omega^{1/2}(C)$ let $\mathfrak D_\kappa:
D^b(\mD\on{-mod}(\Bun_G)^0)\is D^b(\on{QCoh}(\Loc_{\breve G}^{0}))$
be the  equivalence in Theorem \ref{D_kappa}.

\begin{thm}\label{FW}
1) The equivalence $\mathfrak D_\kappa$ intertwines the auto-equivalence $\mL_{G,\chi}\otimes ?$ of $D^b(\mD\on{-mod}(\Bun_G)^0)$
and the auto-equivalence $b_{\chi,\breve G}^*$ on $D^b(\on{QCoh}(\Loc_{\breve G}^{0}))$ constructed in \S\ref{a_chi and b_chi}.
\\
2)
The equivalence $\mathfrak D_\kappa$ intertwines the auto-equivalence $a_{\chi,G}^*$ of $D^b(\mD\on{-mod}(\Bun_G)^0)$ as in \S\ref{a_chi and b_chi}
and the auto-equivalence $\mL_{\breve G,\chi}\otimes ?$ on $D^b(\on{QCoh}(\Loc_{\breve G}^{0}))$ (here we regard $\mL_{\breve G,\chi}$ as a line bundle on $\Loc_{\breve G}^0$ via the projection $\Loc_{\breve G}\to \Bun_{\breve G}$).
\end{thm}

\begin{remark}
Similar actions by tensoring line bundles on $\Loc_G$ 
and on $\on{Higgs}_G$ also 
appear in the work of Frenkel and Witten \cite[\S 10.4]{FW}. Moreover, the authors also 
predict that the geometric Langlands correspondence should interchange these actions.
\end{remark}

Combining Theorem \ref{mu_2 gerbe} and Theorem \ref{FW} we have the following: 
\begin{corollary}
Let $\kappa_1,\kappa_2\in\omega^{1/2}(C)$. Then there is a natural isomorphism 
of equivalences
 \[\frak D_{\kappa_1}\is (\mL_{\breve G,\chi}\otimes\ ?)\circ\frak D_{\kappa_2}
 \circ (\mL_{G,\chi}\otimes\ ?).\]
Here $\chi\in\mu_2\on{-tors}(C)$ such that $\kappa_1=\chi\cdot\kappa_2$
and $\mL_{G,\chi}\otimes\ ?$ (resp. $\mL_{\breve G,\chi}\otimes\ ?$)
is the functor of tensoring with the line bundle $\mL_{G,\chi}$ (resp. $\mL_{\breve G,\chi}$).

\end{corollary}

The rest of this subsection is devoted to the proof of this theorem.

We first introduce a morphism of Picard stack 
\[\tilde\fl:Z(G)\on{-tors}(C)\times B'\ra\on{Pic}(\breve\sH)\]
and prove a 
twisted version of Proposition \ref{Aut line bundle 1}. 
We begin with the construction of $\tilde\fl$.
Let $\Bun_{J^p}$ be the Picard stack of $J^p$-torsors over $C$. 
We have the generalized Chern class map
 $\tilde c_{\breve J^p}:\Bun_{\breve J^p}\ra\Pi_{\breve G}(1)\on{-gerbes}(X)\times B'$ 
and a Picard functor $\fl_{\breve J^p}:Z(G)\on{-tors}(C)\times B'\ra\on{Pic}(\Bun_{\breve J^p})$.
We define 
\[\tilde\fl:Z(G)\on{-tors}(C)\times B'\stackrel{\fl_{\breve J^p}}\ra\on{Pic}(\Bun_{\breve J^p})
\ra\on{Pic}(\breve\sH)\]
where the last map is induced by the restriction map
$\breve\sH=\Loc_{\breve J^p}(\tau')\ra\Bun_{\breve J^p}$.

Recall the morphism $\breve\fl_{\breve J}:Z_{G}\on{-tors}(C)\times B'\ra\sP'$ constructed in \S \ref{l_J}.
For any $Z(G)$-torsor $K$ over $C$, 
we define $$\mL_{\breve J^p,K}:=\tilde\fl(\{K\}\times B')\in\on{Pic}(\breve\sH).$$
Let $K'$ denote the Frobenius descendent of $K$ (as $C_{et}\simeq C'_{et}$), and let
$$K'_{J'}=\breve\fl_{\breve J}(\{K'\}\times B')\in\sP'(B').$$ 
We will relate $\mL_{\breve J^p,K}$ with $K'_{J'}$ via the twisted duality.
From the definition of $\theta'_m$ in \S\ref{canonical one forms}, one can easily check that 
the restriction of $\theta'_m$ to $K'_{J'}$ is zero. Thus the restriction of 
the $\bG_m$-gerbe $\sD(\theta'_m)$ to $K'_{J'}$ is canonical trivial and we can regard the structure sheaf  $\delta_{K'_{J'}}\in\on{QCoh}(\sP')$
as an object in $\on{QCoh}(\sD(\theta'_m))_1$.
Let $\tilde\mL_{K}=\mathfrak D(\delta_{K'_{J'}})\in\on{Pic}(\breve\sH)$ be 
the image of $\delta_{K'_{J'}}$ under the twisted duality 
$\mathfrak D:D^b(\on{QCoh}(\sD(\theta')))_1\is D^b(\on{QCoh}(\breve\sH))$.

\begin{lemma}\label{Spectral line-bundle 1}
We have $\tilde\mL_{K}\is\mL_{\breve J^p,K}$.
\end{lemma}
\begin{proof}
Let $\breve\mG:=\sD(\theta_m')^\vee$. We have a short exact sequence of Beilinson's 1-motives 
$0\ra\breve\sP'\ra\breve\mG\stackrel{p}\ra\bZ\ra 0$ and $\breve\sH=p^{-1}(1)$.
The construction of duality for torsors in \ref{DFT} implies that
there is a multiplicative line bundle $\tilde\mL_{\breve\mG,K}$ on $\breve\mG$ such that 
$\tilde\mL_{\breve\mG,K}|_{\breve\sH}\is\tilde\mL_K$. Moreover, this line bundle is characterized by the property that  
$\tilde\mL_{\breve\mG,K}|_{\breve\sP'}\is\breve{\mathfrak D}_{cl}^{-1}(K_{J'}')$.
Observe that we have a natural map $\breve\mG\ra\Bun_{\breve J^p}$ of 
Picard stacks\footnote{We have $\breve\mG=\{(n,t)|n\in\bZ,\ t\in\breve\sH^{\otimes n}\}$ 
and $\breve\sH^{\otimes n}$ is isomorphic to $\Loc_{\breve J^p}(n\cdot\tau')$, the base change of 
$\Loc_{\breve J^p}\ra B_{\breve J'}$ along the section $n\cdot\tau':B'\ra B_{\breve J'}$.
Thus there is a natural map $\breve\sH^{\otimes n}\ra\Bun_{\breve J^p}$ and 
the map $\breve\mG\ra\Bun_{\breve J^p}$ is given by
$\breve\mG\ra\breve\sH^{\otimes n}\ra\Bun_{\breve J^p}$.
}
such that 
the composition $\breve\sH\ra\breve\mG\ra\Bun_{\breve J^p}$ is the natural inclusion.
Thus the morphism $\tilde\fl:Z_{G}\on{-tors}(C)\times B'\ra\on{Pic}(\breve\sH)$
factors through a morphism $\tilde\fl_{\breve\mG}:
Z_{G}\on{-tors}(C)\times B'\ra\breve\mG^\vee$,
and the corresponding multiplicative line bundle $
\mL_{\breve\mG,K}:=\tilde\fl_{\breve\mG}(\{K\}\times B')\in\breve\mP^\vee(B')$ 
satisfies $\mL_{\breve\mG,K}|_{\breve\sH}\is\mL_{\breve J^p,K}$.
It is enough to show that  
$\tilde\mL_{\breve\mG,K}\is\mL_{\breve\mG,K}$.
From the characterization of $\tilde\mL_{\breve\mG,K}$,  it is enough to show 
that $\mL_{\breve\mG,K}|_{\breve\sP'}\is\breve{\mathfrak D}_{cl}^{-1}(K_{J'}')$.
But this follows from Proposition \ref{Aut line bundle 1} and the fact that $\mL_{\breve\mG,K}|_{\breve\sP'}$
is isomorphic to $\mL_{\breve J',K'}$.
\end{proof}

Recall that a choice of $\kappa\in\omega^{1/2}(C)$ defines 
an isomorphism $\frak C_\kappa:\breve\sH\is\Loc_{\breve G}^{reg}$.
More precisely, we have $\frak C_\kappa(P,\nabla)=(P\otimes F_C^*E_{\kappa'},\nabla_{P\otimes F_C^*E_{\kappa'}})$ where $P\otimes F_C^*E_{\kappa'}:=P\times^{J^p}F_C^*E_{\kappa'}$
and $\nabla_{P\otimes F_C^*E_{\kappa'}}$ is the product connection.

\begin{lemma}\label{Spectral line-bundle 2}
The pull back of the line bundle $\mL_{\breve G,K}$ along the map
$\breve\sH\stackrel{\frak C_\kappa}\ra\Loc_{\breve G}\stackrel{\pr}\ra\Bun_{\breve G}$
is isomorphic to $\tilde\mL_{K}$. I.e. we have 
$\tilde\mL_{K}\is \frak C_\kappa^*\circ\pr^*\mL_{\breve G,K}$.
\end{lemma}
\begin{proof}
The proof is similar to the proof of Lemma \ref{Aut line bundle 2}.
Recall that the line bundles $\mL_{\breve G,K}$ and 
$\tilde\mL_{K}\is\mL_{\breve J^p,K}$ are induced by 
the generalized Chern class map $\tilde c_{\breve G}$, $\tilde c_{\breve J^p}$. 
Therefore it is enough to show that for any $(P,\nabla)\in\breve\sH$
there is a canonical isomorphism 
$\tilde c_{\breve J^p}(P)\is\tilde c_{\breve G}(\frak C_\kappa(P))$
of $\Pi_{\breve G}\on{-gerbes}$, where $\frak C_\kappa(P)=P\times^{J^p}F_C^*E_{\kappa'}$.
Let $\tilde P\in\tilde c_{\breve J^p}(P)$
and $\tilde E_{\kappa'}$ be the canonical lifting of 
the Kostant section appearing in Lemma \ref{Aut line bundle 2}.
The $G_{sc}$-torsor $\tilde P\times^{(J^p_{sc})}F_C^*\tilde E_{\kappa'}$ is a lifting of 
$\frak C_\kappa(P)$ and the assignment 
$\tilde P\ra \tilde P\times^{(J^p_{sc})}F_C^*\tilde E_{\kappa'}$ defines an isomorphism 
between $\tilde c_{J^p} (P )$ and $\tilde c_{\breve G}(\frak C_\kappa(P))$. This finishes the proof.
\end{proof}

\medskip

Now we prove the theorem.
Recall that we have 
$\mathfrak D_\kappa=(\mathfrak C_\kappa^*)^{-1}\circ\mathfrak D\circ\mathfrak A_\kappa $
where $\mathfrak A_\kappa$ and $\mathfrak C_\kappa^*$ are 
equivalences constructed in \S\ref{ab thm}.
It follows from the definition 
that under the equivalence $\mathfrak C_\kappa^*$ 
the functor $b_{\breve G,\chi}^*$ becomes the functor induced by
the geometric action of $K'_{\breve G,\chi}\in Z({\breve G})\on{-tors}(C')$ on $\breve\sH$\footnote{Recall that $K_{\breve G,\chi}$ carries a canonical connection with 
zero $p$-curvature and $K_{\breve G,\chi}'$ is its Frobenius descent.}. 
Now Theorem \ref{twist FM} implies, under the equivalence 
\[\frakD^{}:D^b(\on{QCoh}(\sD(\theta_m')|_{B^{'0}}))_1\is D^b(\on{QCoh}(\breve\sH|_{B^{'0}})),\]
above
geometric action 
becomes the functor of tensoring with the line bundle 
$\mL_{J,\chi}':=\mathfrak D_{cl}^{-1}(K_{\breve G,\chi}')\in (\Bun_{J'})^\vee$\footnote{
Here we use the fact that $\omega_{\sP^\vee/B}\cong pr_B^*(e^*\omega_{\sP^\vee/B})$ is trivial. Indeed, since $B$ is isomorphic to an affine space
we have $e^*\omega_{\sP^\vee/B}\in\Pic(B)=0$.}.
By Lemma \ref{Aut line bundle 2} and 
Lemma \ref{Aut line bundle 1}, the line bundle 
$\mL_{J,\chi}'$
is equal to the pull back of $\mL_{G,\chi}'$ under the 
map 
$\sP'\stackrel{\epsilon_\kappa'}\ra\on{Higgs}'_G\ra\Bun_G'$.
On the other hand, since the equivalence $\mathfrak A_\kappa:
D^b(\mD\on{-mod}(\Bun_G^0))\is 
D^b(\on{QCoh}(\sD(\theta_m')|_{B^{'0}}))_1$ is induced 
by pullback along the morphism $\epsilon_\kappa:\sP\ra\on{Higgs}_G$,
an easy exercise shows that under the equivalence $\mathfrak A_\kappa$
the functor of tensoring with $\mL_{J,\chi}'$
becomes the functor of tensoring with $\mL_{G,\chi}$.
This implies Part 1).

The proof of part 2) is similar to part 1).
Unraveling the definition of $a_{G,\chi}^*$ and 
the construction of $\mathfrak A_\kappa$, one sees that $\mathfrak A_\kappa$ interchanges the 
functor $a_{G,\chi}^*$ with the functor of convolution product
with $\delta_{K_{G,\chi}'}\in\on{QCoh}(\sP')$.
Now Theorem \ref{twist FM} implies that, under the equivalence $\mathfrak D$, the above
convolution action 
becomes the functor of tensoring with the line bundle 
$\tilde\mL_{K_{G,\chi}}:=\mathfrak D(K_{G,\chi}')\in\on{Pic}\breve\sH$.
By Lemma \ref{Spectral line-bundle 1} and Lemma \ref{Spectral line-bundle 2}, the line bundle 
$\tilde\mL_{K_{G,\chi}}$
is isomorphic to the pull back of $\mL_{\breve G,\chi}$ under the 
map 
$\breve\sH\stackrel{\frak C_\kappa}\ra\Loc_{\breve G}\stackrel{\pr}\ra\Bun_{\breve G}$.
It implies that
$\frak C_\kappa^*\circ (\pr^*\mL_{\breve G,\chi}\otimes ?)\is(\tilde\mL_{K_{G,\chi}}\otimes ?)\circ\frak C^*_\kappa$.



\appendix

\pagestyle{empty}

\section{Beilinson's 1-motive}\label{A}
In this section, we review the duality theory of Beilinson's
1-motives. The main references are \cite{A,DP,DP2,Lau}.
\subsection{Picard Stack}\label{Picard}
Let us first review the theory of Picard stacks. The standard
reference is \cite[\S 1.4]{Del}. Let $\mT$ be a given site. Recall
that a Picard Stack is a stack $\sP$ over $\mT$ together with a
bi-functor
\[\otimes:\sP\times\sP\to\sP,\]
and the associativity and commutative constraints
$$
a:\otimes\circ(\otimes\times 1)\simeq
\otimes\circ(1\times\otimes),\quad\quad
c:\otimes\simeq\otimes\circ\rm flip,
$$
such that for every $U\in\mT$, $\sP(U)$ form a Picard groupoid (i.e.
symmetrical monoidal groupoid such that every object has a monoidal
inverse). The Picard stack is called strictly commutative if
$c_{x,x}=\id_x$ for every $x\in \sP$. In the paper, Picard stacks
will always mean strictly commutative ones.

Let us denote by $\mP\mS/\mT$ the 2-category of Picard stacks
over $\mT$. This means that if $\sP_1,\sP_2$ are two Picard stacks
over $\calT$, $\Hom_{\mP\mS/\mT}(\sP_1,\sP_2)$ form a category.
Indeed, $\mP\mS/\mT$ is canonically enriched over itself. For
$\sP_1,\sP_2\in\mP\mS/\mT$, we use $\underline{\Hom}(\sP_1,\sP_2)$
to denote the Picard stack of 1-homomorphisms from $\sP_1$ to $\sP_2$
over $\calT$ (cf. \cite{Del} \S 1.4.7). On the other hand, let
$C^{[-1,0]}$ be the 2-category of 2-term complexes of sheaves of
abelian groups $d:\calK^{-1}\to\calK^0$ with $\calK^{-1}$ injective
and 1-morphisms are morphisms of chain complexes (and 2-morphisms
are homotopy of chain complexes)\footnote{The 2-category $C^{[-1,0]}$ is an enhancement of the subcategory $D^{[-1,0]}\subset D$ of the derived category consisting of complexes concentrated in 
cohomological degrees $[-1,0]$. That is, the homotopy category of $C^{[-1,0]}$ is equivalent to $D^{[-1,0]}$.}.  Let $\calK\in C^{[-1,0]}$. We
associate to it a Picard prestack $\rm pch(\calK)$ whose $U$ point
is the following Picard category
\begin{enumerate}
\item Objects of $\rm pch(\calK)(U)$ are equal to $\calK^0(U)$.\\
\item If $x,y\in\calK^0(U)$, a morphism from $x$ to $y$ is an element
$f\in\calK^{-1}(U)$ such that $df=y-x$.
\end{enumerate}
Let $\rm ch(\calK)$ be the stackification of $\rm pch(\calK)$. Then
a theorem of Deligne (cf. \cite[Corollaire 1.4.17]{Del}) says that the functor
\[\rm ch:C^{[-1,0]}\to\mP\mS/\mT\]
is an equivalence of 2-categories.

Let us fix an inverse functor $()^\flat$ of the above equivalence.
So for $\sP$ a Picard stack, we have a 2-term complex of sheaves of
abelian groups $\sP^\flat:=\calK^{-1}\to\calK^0$. For example, if
$A$ is an abelian group in $\mT$, then its classifying stack $BA$ is
a natural Picard stack and $(BA)^\flat$ can be represented by a
2-term complex quasi-isomorphic to $A[1]$. The following result of
Deligne (cf. \cite[Construction 1.4.18]{Del}) is convenient for computations.
\begin{equation}\label{compare Hom}
(\underline{\Hom}(\sP_1,\sP_2))^\flat\simeq\tau_{\leq
0}\on{R}\underline{\Hom}(\sP_1^\flat,\sP_2^\flat).
\end{equation}

\subsection{Short exact sequences of Picard stacks}\label{short exact seq}
Let $a: \sP_1\to \sP_2$ be a homomorphism of Picard stacks. We define $\ker(a)$ as the fiber $\sP_1\times_{\sP_2}\{e\}$, where $e\in\sP_2$ is the unit. Then $\ker(a)$ acquires a natural Picard stack structure. It follows from the construction of $\rm ch$ that
\begin{lem}
There is a natural isomorphism $\ker(a)^\flat\simeq \tau_{\leq 0}C(a^\flat)[-1]$, where $C(a^\flat)$ is the cone of the morphism of complexes
\[a^\flat:\sP_1^\flat\to \sP_2^\flat.\]
\end{lem}

A left exact sequence of Picard stacks, usually denoted by 
\[1\ra\sP_1\stackrel{a}\ra\sP_2\stackrel{b}\ra\sP_3,\]
is a sequence of homomorphisms of Picard stacks that exhibits $\sP_1$ as $\ker(b)$.
If, in addition locally on $\mT$, $b$ is essentially surjective, we call such a sequence exact and denote it by
\[1\ra\sP_1\stackrel{a}\ra\sP_2\stackrel{b}\ra\sP_3\to 1.\]
Sometimes, we also call $\sP_2$ an \emph{extension} of 
$\sP_3$ by $\sP_1$. The following lemma is used in several places in the paper.
\begin{lem}The sequence of homomorphisms $\sP_1\to \sP_2\to \sP_3$ is exact if and only if 
\[\sP_1^\flat\to \sP_2^\flat\to \sP_3^\flat\to\]
is a distinguished triangle.
\end{lem}

\subsection{Duality of Picard stacks}\label{duality}
Let $S$ be a noetherian scheme. We consider the category ${\rm
Sch}/S$ of schemes over $S$. We will endow ${\rm Sch}/S$ with
\emph{fpqc} topology in the following discussion.

\begin{definition}\label{appen:dual}
For a Picard stack $\sP$, we define the dual Picard stack as
$$\sP^\vee:=\underline{\rm Hom}(\sP,B\bG_m).$$
\end{definition}

\begin{example}\label{abscheme}
Let $A\to S$ be an abelian scheme over $S$. Then by definition
$A^\vee:=\underline{\rm Hom}(A,B\bG_m)=\underline{\rm
Ext}^1(A,\bbG_m)$ classifies the multiplicative line bundles on $A$,
is represented by an abelian scheme over $S$, called the dual
abelian scheme of $A$.
\end{example}

\begin{example}\label{fg group}
Let $\Gamma$ be a finitely generated abelian group over $S$. By
definition, this means locally on $S$, $\Gamma$ is isomorphic to the
constant sheaf $M_S$, where $M$ is a finitely generate abelian group
(in the naive sense). Recall that the Cartier dual of $\Gamma$,
denoted by $\on{D}(\Gamma)$ is the sheaf which assigns every scheme
$U$ over $S$ the group $\Hom(\Gamma\times_SU,\bbG_{m})$, which is
represented by an affine group scheme over $S$. We claim that
$\Gamma^\vee\simeq B\on{D}(\Gamma)$. By \eqref{compare Hom}, it is
enough to show that $\on{R}^i\underline{\rm Hom}(\Gamma,\bbG_m)=0$
if  $i>0$. This is clear since locally on $S$, $\Gamma$ is
represented by a 2-term complex $\bbZ_S^m\to\bbZ_S^n$.
\end{example}

\begin{example}\label{mult}
Let $G$ be a group of multiplicative type over $S$, i.e.
$G=D(\Gamma)$ for some finitely generated abelian group $\Gamma$
over $S$. Let $\sP=BG$, the classifying stack of $G$. We have
\[\sP^\vee\simeq \tau_{\leq 0}\on{R}\underline{\rm Hom}(BG,B\bG_m)\simeq \underline{\rm Hom}(G,\bbG_m)\simeq\Gamma.\]
\end{example}

\begin{definition}
Let $\sP$ be a Picard stack. We say that $\sP$ is dualizable if the
canonical 1-morphism $\sP\to\sP^{\vee\vee}$ is an isomorphism.
\end{definition}
By the above examples, abelian schemes, finitely generated abelian
groups, and the classifying stacks of groups of multiplicative type are
dualizable.

Let $\sP$ be a dualizable Picard stack. There is the Poincare line
bundle $\mL_{\sP}$ over $\sP\times_S\sP^\vee$. Let
$D^b(\on{QCoh}(\sP))$ denote the bounded derived category of
quasi coherent sheaves on $\sP$. We define the Fourier-Mukai functor
\[\Phi_{\sP}:D^b(\on{QCoh}(\sP))\ra D^b(\on{QCoh}(\sP^\vee)),\quad \Phi_{\sP}(F)=(\on{R}p_2)_*(\on{L}p_1^*F\otimes^{}\mL_{\sP}).\]
Here $p_1:\sP\times_S\sP^\vee\ra\sP$ and
$p_2:\sP\times_S\sP^\vee\ra\sP^\vee$ denote the natural projections.
It is easy to see in the case when $\sP$ is of the form given in the
above examples, $\Phi_\sP$ is an equivalence of categories. Indeed,
the case when $\sP=A$ follows from the results of Mukai; the case
when $\sP=\Gamma$ or $BG$ is clear.

It is not clear to us whether $\Phi_\sP$ is an equivalence for all dualizable Picard stacks. In the following subsection, we select out a particular class of
Picard stacks, called the Beilinson's 1-motive (following \cite{DP} and Arinkin's appendix to \cite{DP2}), for which the
Fourier-Mukai transforms are equivalences.

\subsection{Beilinson's 1-motives}
Let $\sP_1,\sP_2$ be two Picard stacks. We say that
$\sP_1\subset\sP_2$ if there is a 1-morphism $\phi:\sP_1\to\sP_2$,
which is a faithful embedding.

\begin{definition}
We called a Picard stack $\sP$ a Beilinson's 1-motive if it admits a
two step filtration $W_\bullet\sP$:
\[W_{-1}=0\subset W_{0}\subset W_1\subset W_2=\sP\]
such that (i) $\Gr^W_0\simeq BG$ is the classifying stack of a group
$G$ of multiplicative type; (ii) $\Gr^W_1\simeq A$ is an abelian
scheme; and (iii) $\Gr^W_2\simeq\Gamma$ is a finitely generated
abelian group.
\end{definition}

\begin{lemma}The dual of a Beilinson's 1-motive is a Beilinson's
1-motive and Beilinson's 1-motives are dualizable.
\end{lemma}
\begin{proof}This is proved via the induction on the length of the
filtration. We use the following fact. Let
\[0\to\sP'\to\sP\to\sP''\to 0\]
be a short exact sequence of Picard stacks. Then
\[0\to(\sP'')^\vee\to\sP^\vee\to(\sP')^\vee\]
with the right arrow surjective if $\on{R}^2\underline{\rm
Hom}((\sP'')^\flat,\bbG_m)=0$.

If $\sP=W_0\sP$, this is given by Example \ref{mult}. If
$\sP=W_1\sP$, we have the following exact sequence
\[0\to BG\to \sP\to A\to 0.\]
Using the fact that $\underline{\rm Ext}^2(A,\bbG_m)=0$ (See
\cite[Remark 6]{LB1}), we know that $\sP$ is also a Beilinson's
1-motive. In general, we have
\[0\to W_1\sP\to \sP\to \Gamma\to 0,\]
and the lemma follows from the fact $\underline{\rm
Ext}^2(\Gamma,\bbG_m)=0$ (see Example \ref{fg group}).
\end{proof}

\begin{corollary}\label{aut-pi0 dual}
Let $\sP$ be a Beilinson 1-motive, and $\sP^\vee$ be its dual. Then
$\on{D}(\Aut_{\sP}(e))=\pi_0(\sP^\vee)$, where $e$ denotes the unit
of $\sP$ and $\pi_0$ denotes the group of connected components of
$\sP^\vee$.
\end{corollary}

\begin{lemma}\label{local splitting}
Let $\sP$ be a Beilinson's 1-motive. Then locally on $S$,
\[\sP\simeq A\times BG\times \Gamma.\]
\end{lemma}
\begin{proof}It is enough to prove that 
$$\underline{\rm Ext}^1(\Gamma,
BG)=\underline{\rm Ext}^1(\Gamma,A)=\underline{\rm Ext}^1(A,BG)=0.$$
Clearly, $\underline{\rm Ext}^1(\Gamma, BG)=\underline{\rm
Ext}^2(\Gamma, G)=0$. To see that $\underline{\rm Ext}^1(\Gamma,
A)=0$, we can assume that $\Gamma=\bbZ/n\bbZ$. Then it follows that
$A\stackrel{n}{\to}A$ is surjective in the flat topology that
$\underline{\rm Ext}^1(\Gamma, A)=0$.

To see that $\underline{\rm Ext}^1(A, BG)=0$, let $\sP$ to the
Beilinson's 1-motive corresponding to a class in $\underline{\rm
Ext}^1(A, BG)$. Taking the dual, we have $0\to A^\vee\to\sP^\vee\to
\on{D}(G)\to 0$. Therefore, locally on $S$, $\sP^\vee\simeq
A^\vee\times \on{D}(G)$, and therefore locally on $S$,
$\sP^{\vee\vee}\simeq A\times BG$.
\end{proof}

\begin{definition}[cf. \cite{A}]\label{good}
We say that a Picard stack $\sP$ is \emph{good} if it satisfies the following two conditions
\begin{enumerate}
\item $\sP$ is dualizable, i.e., the map $r:\sP\ra\sP^{\vee\vee}$ is an isomorphism of Picard stacks.
\item The functor $\Phi_\sP:D^b(\on{QCoh}(\sP))\ra D^b(\on{QCoh}(\sP^\vee))$ is an equivalence of categories.
\end{enumerate}
\end{definition}

As explained in \S\ref{duality} (see also \cite{BB}), examples of good Picard stacks include 
$BG$, $\Gamma$ and abelian schemes over $S$, as well as fiber products over $S$ of such. More generally
\begin{thm}[\cite{A}, Proposition A.6]\label{gen FM} Let $\sP$ be a Beilinson's 1-motive. Then 
$\sP$ is ``good" in the sense of Definition \ref{good}. In particular, 
the functor $\Phi_{\sP}$ is an equivalence of categories.
\end{thm}
\begin{proof}Indeed,  the property of being 
good is fpqc-local on $S$. This can be seen by lifting $\Phi_{\sP}$ to a functor between stable $\infty$-categories of quasi-coherent sheaves and then applying a descent argument. Therefore, the theorem follows from Lemma \ref{local splitting} and the above examples.

Alternatively, similar to the usual Fourier-Mukai transform, one can show directly that in our generality, there is still an isomorphism of functors $\Phi_{\sP^\vee}\circ \Phi_{\sP}\simeq \omega_{\sP/S}^{-1}\otimes (-1)^*[-g]$, where $\omega_{\sP/S}$ is the canonical sheaf and $g$ is the relative dimension of 
$\sP/S$\footnote{This argument was suggested to us by the referees.}. By the argument as in \cite{Mu2}
(see also \cite{Lau}), one reduces to show that 
the kernel complex
\[\on{R}p_{12*}(\on{L}p_{13}^*\mL_\sP\otimes^{} \on{L}p_{23}^*\mL_\sP)\is m^*\on{R}p_{1*}\mL_\sP\]
for 
the functor $\Phi_{\sP^\vee}\circ \Phi_{\sP}$ is isomorphic to 
the kernel complex 
\[\sigma_*(\omega_{\sP/S}^{-1})[-g]\simeq m^*e_*(e^*\omega_{\sP/S}^{-1})[-g]\]
for $\omega_{\sP/S}^{-1}\otimes (-1)^*[-g]$.
Here $p_{ij}$ are the projections of $\sP\times_S\sP\times_S\sP^\vee$ on the $(i,j)$-factors,
$\sigma:\sP\ra\sP\times_S\sP, x\ra (x,x^{-1})$ and $e:S\ra\sP$ is the unit morphism. 

To prove this, we first observe that there is a natural map 
\[\on{R}p_{1*}\mL_\sP\simeq\on{R}^gp_{1*}\mL_\sP[-g]\ra e_*\on{R}^gpr_{S*}\mO_{\sP^\vee}[-g]\ra
e_*(e^*\omega_{\sP/S}^{-1})[-g].\]
We claim that the map above is an isomorphism, and hence induces $m^*\on{R}p_{1*}\mL_\sP\simeq m^*e_*(e^*\omega_{\sP/S}^{-1})[-g]$.
To prove the claim, we observe that, 
 using Lemma \ref{local splitting} and fpqc base change, we can assume $\sP\simeq A\times BG\times\Gamma$ 
and the claim follows from the results in \cite[p.519]{Mu2} or
\cite[Lemma 1.2.5]{Lau}.
\end{proof}

Entirely similar arguments as in \cite[p.160]{Mu1} and \cite[Corollary 1.3.3]{Lau} give us 
\begin{thm}\label{property of FM}
Let $\sP$ be a Beilinson 1-motive.
Let \[*:D^b(\on{QCoh}(\sP))\times D^b(\on{QCoh}(\sP))\ra D^b(\on{QCoh}(\sP))\] 
be the functor 
defined by 
$\mF_1*\mF_2:=\on{R}m_*(\mF_1\boxtimes\mF_2)$. We called $*$ the convolution product.
Then there are canonical isomorphisms
$$\Phi_\sP(\mF_1*\mF_2)\simeq\Phi_\sP(\mF_1)\otimes^{}\Phi_\sP(\mF_2)$$
and
$$ \Phi_\sP(\mF_1\otimes^{}\mF_2)\simeq(\Phi_\sP(\mF_1)*\Phi_\sP(\mF_2))
\otimes^{}\omega_{\sP^\vee/S}[g].$$
\end{thm}

\subsection{Multiplicative torsors and extensions of Beilinson 1-motives}\label{mult torsor}
Let us return to the general set-up. Let $\mT$ be a fixed site and
let $\sP$ be a Picard stack over $\mT$. 
A torsor of $\sP$ is a stack
$\sQ$ over $\mT$, together with a bi-functor
\[\on{Action}: \sP\times\sQ\to \sQ,\]
satisfying the following properties:
\\(i) the bi-functor $\on{Action}$  defines a monoidal action of
$\sP$ on $\sQ$;
\\(ii) For every $V\in \mT$, there exists a covering $U\to V$, such that $\sQ(U)$ is non-empty.
\\(iii) For every $U\in \mT$ such that $\sQ(U)$ is non-empty and let $D\in\sQ(U)$, the functor
\[\sP(U)\to \sQ(U),\quad C\mapsto \on{Action}(C,D)\]
is an equivalence.

In the case when $\sP$ is the Picard stack of $G$-torsors for some
sheaf of abelian groups $G$, people usually call a $\sP$-torsor
$\sQ$ a $G$-gerbe.

All $\sP$-torsors form a 2-category, denoted by $B\sP$, is
canonically enriched over itself (\cite[\S 2.3]{OZ}). I.e., given
two $\sP$-torsors $\sQ_1,\sQ_2$, $\underline\Hom_\sP(\sQ_1,\sQ_2)$
is a natural $\sP$-torsor. An object in
$\underline\Hom_\sP(\sQ_1,\sQ_2)$ induces an equivalence between
$\sQ_1$ and $\sQ_2$. In addition, there is a monoidal structure on
$B\sP$ making $B\sP$ a Picard 2-stack.

\begin{remark}
Let $1\ra \sP_1\ra\sP\ra\bZ_S\ra 1$ be an exact sequence  of Picard stacks.
Then $\sT:=\sP\times_{\bZ_S}\{1\}$ is naturally a $\sP_1$-torsor. As explained in \cite{A} and \cite[\S 3.1]{Trav}, the 
correspondence $\sP\ra\sT$ induces an equivalence of 2-categories between 
extensions of $\bZ_S$ by $\sP_1$ and $\sP_1$-torsors.
\end{remark}

Now, let $\sP$  and 
$\sP_1$ be two Picard stacks and
let $\sG$ be a $\sP_1$-torsor
over $\sP$. Let $m: \sP\times \sP\ra \sP$, $e:\mT\ra \sP$ be
the multiplication morphism and the unit morphism respectively, and let
$\sigma:\sP\times \sP\ra\sP\times \sP$ be the flip map
$\sigma(x,y)=(y,x)$.
\begin{definition}\label{app:comm}
 A commutative group structure on $\sG$ consists of the following data:
\begin{enumerate}
\item An equivalence $M:\sG\boxtimes\sG\simeq m^*\sG$ of $\sP_1$-torsors over $\sP
\times \sP$;
\item A $2$-morphism $\gamma$ between the resulting two $1$-morphisms between
$\sG\boxtimes\sG\boxtimes\sG$ and $m^*\sG$ over
$\sP\times\sP\times\sP$, which satisfies the cocycle condition.
\item A $2$-morphism
$i:\sigma^*M\is M$ such that $i^2=id$. (Note that $\sigma^*(M)$ is another $1$-morphism between $m^*\sG$ and $\sG\boxtimes\sG$.)
\end{enumerate}
\end{definition}
Clearly, all $\sP_1$-torsors over $\sP$ with a  commutative group structure also form a $2$-category.

We have the following Lemma.
\begin{lem}\label{ext}
A commutative group structure on $\sG$ makes $\sG$ into a Picard
stack which fits into the following short exact sequence:
$$0\ra \sP_1\ra\sG\ra \sP\ra 0.$$
In particular, if $\sP$ is a Beilinson's 1-motive, and $\sP_1=B\bG_m$,
then $\sG$ is a Beilinson's 1-motive. In this case, we also call
$\sG$ a multiplicative $\bG_m$-gerbe over $\sP$.
\end{lem}

\begin{definition}\label{mult splitting}
A multiplicative splitting of a $\sP_1$-torsor $\sG$ over $\sP$ with a commutative group structure 
is a $1$-morphism (in the category of all $\sP_1$-torsors over $\sP$ with a commutative group structure): $\sP\ra\sG$.
\end{definition}

\subsection{Induction functor }\label{induction functor}
Let $\phi:\sP\ra\sP_1$ be a morphism of Picard stacks. Then 
to each $\sP$-torsor $\sQ$ we may associate a $\sP_1$-torsor
$\sQ^\phi:=\underline\Hom_{\sP}(\sQ^{-1},\sP_1)$ whose sections 
are $\sP$-equivariant functors from $\sQ^{-1}:=
\underline\Hom_{\sP}(\sQ,\sP)$ to $\sP_1$ (here 
$\sP$ acts on $\sP_1$ via $\phi$) and whose morphisms are natural transformations of such functors.

We have a canonical functor
$\sQ\ra\sQ^\phi$, compatible with their $\sP$ and $\sP_1$-structure 
via $\phi$. 
For any section $E$ of $\sQ$ we denote by $E^\phi$
the section of $\sQ^\phi$ induced by the canonical map $\sQ\ra\sQ^\phi$.

\subsection{Duality for torsors}\label{DFT}
Let $\sY$ be an algebraic stack. Let $\widetilde\sY$ be a
$\bG_m$-gerbe over $\sY$, i.e. $\widetilde\sY$ is a $B\bG_m$-torsor 
over $\sY$. We say
$\widetilde\sY$ is split if it is isomorphic to $\sY\times B\bG_m$. Let
$D^b(\on{QCoh}(\widetilde\sY))$ be the bounded derived category of
quasi coherent sheaves on $\widetilde\sY$.  If $\widetilde\sY$ is split,
there is a decomposition
\begin{equation}\label{decomposition}
D^b(\on{QCoh}(\widetilde\sY))=\oplus_{n\in\bZ}
D^b(\on{QCoh}(\widetilde\sY))_n
\end{equation}
 according to the character of
$\bG_m$\footnote{The direct sum in (\ref{decomposition}) means that every object in $D^b(\on{QCoh}(\widetilde\sY))$
decomposes as a direct sum of objects in the subcategories $D^b(\on{QCoh}(\widetilde\sY))_n$.
}.  In general we still have such a
decomposition given as follows: 
$\mM\in D^b(\on{QCoh}(\widetilde\sY))_n$ if only if
$a^*(\mM)\in D^b(\on{QCoh}(\widetilde\sY))_n$, where 
$a:B\bG_m\times\widetilde\sY\ra\widetilde\sY$ is the action map.
\begin{definition}
The direct summand $D^b(\on{QCoh}(\widetilde\sY))_1$ is called the
category of twisted sheaves on $\widetilde\sY$.
\end{definition}

Now we further assume $\sY=\sP$ is a Beilinson's 1-motive over $S$ and
$\widetilde\sY=\sD$ is a multiplicative $\bG_m$-gerbe over $\sP$. 
Let $\sP$ and $\sD$ as above.
Then by Lemma \ref{ext} we have the following short
exact sequence

\begin{equation}\label{appen:ext}
0\to B\bG_m\stackrel{i}{\to} \sD\stackrel{p}{\to} \sP\to 0
\end{equation}
as Picard stacks.
Note that  in this case $\sD$ is also a Beilinson's 1-motive.
Let $\sD^\vee$ be the dual Beilinson's 1-motive. It fits into
the short exact sequence
$$0\ra \sP^\vee\ra\sD^\vee\stackrel{\pi}{\ra}\bZ_S\ra 0.$$
Let 
\begin{equation}\label{mult split}
\sT_{\sD}=\pi^{-1}(1)
\end{equation}
be the $\sP^\vee$-torsor associated to
above extension. We call $\sT_{\sD}$ the stack of multiplicative splitting of $\sD$. To justify the name, let us
give an alternative description of $\sT_{\sD}$. By definition the  dual of
$\sD$ is
$$\sD^\vee=\underline\Hom(\sD,B\bG_m).$$
An element  $s\in \sD^\vee$ belongs to
$\sT_{\sD}$ if and only if the composition
$$B\bG_m\stackrel{i}{\to}\sD\stackrel{s}{\to}B\bG_m$$ is equal to the identity. Equivalently,
$s\in\sT_{\sD}$ gives a splitting of the exact
sequence \eqref{appen:ext} and according to \ref{mult splitting} it is a multiplicative 
splitting of $\sD$.


\

The following theorem follows immediately from Theorem \ref{property of FM}:
\begin{thm}[\cite{A}, \cite{Trav} \S 3.2]\label{twist FM}
1)
The Fourier-Mukai functor $\Phi_{\sD}$ restricts to an equivalence
$$\Phi_{\sD}:D^b(\on{QCoh}(\sD))_1\is D^b(\on{QCoh}(\sT_{\sD})).$$

2) 
There is an action of $D^b(\on{Qcoh}(\sP))$ on $D^b(\on{QCoh}(\sD))_1$ by tensoring
and an action of $D^b(\on{QCoh}(\sP^\vee))$ on
$D^b(\on{QCoh}(\sT_{\sD}))$ by convolution.
Those two actions are compatible with the above equivalence in the following sense:
There is a canonical isomorphism
\[\Phi_{\sD}(\mF_1\otimes^{}\mF_2)\simeq(\Phi_{\sD}(\mF_1)*\Phi_{\sD}(\mF_2))\otimes^{}
\omega_{\sP^\vee/S}[g]\]
for $\mF_1\in D^b(\on{Qcoh}(\sP))$ and $\mF_2\in D^b(\on{QCoh}(\sD))_1$.
Here $\omega_{\sP^\vee/S}$ is the canonical sheaf and 
$g$ is the relative dimension of $\sP/S$.

3) The convolution product $*$ on $D^b(\on{QCoh}(\sD))$ induces a
convolution product  
on $D^b(\on{QCoh}(\sD))_1$ (by abuse of notation we still denote it by 
$*$). On the other hand, the category 
$D^b(\on{QCoh}(\sT_{\sD}))$ has the usual monoidal structure by tensoring. 
The equivalence $\Phi_{\sD}$ is compatible with those monoidal structures: There is a canonical isomorphism
\[\Phi_{\sD}(\mF_1*\mF_2)\simeq\Phi_{\sD}(\mF_1)\otimes\Phi_{\sD}(\mF_2)\]
for $\mF_1,\mF_2\in D^b(\on{QCoh}(\sD))_1$.
\end{thm}

\section{$\mD$-modules on stacks and Azumaya property}\label{B}
In this section we review some basic facts about $\mD$-modules on
algebraic stacks and the Azumaya property of the sheaf of differential
operators. Standard references are \cite{BD} and
\cite{BB}.

\subsection{Azumaya algebras and twisted sheaves}
Let us begin with a review of the basic definition of Azumaya
algebras and the category of twisted sheaves. Let $S$ be a
Noetherian scheme. Let $\sX$ be an algebraic stack over $S$. Recall
that an Azumaya algebra $\calA$ over $\sX$ is a quasi-coherent sheaf
of $\calO_\sX$-algebras, which is locally in smooth topology
isomorphic to $\calE nd(\calV)$ for some vector bundle $\calV$ on
$\sX$. Such an isomorphism between $\calA$ and the matrix algebra is
called a splitting of $\calA$. Given an Azumaya algebra $\mA$ on
$\sX$, one can associate to it the $\bG_m$-gerbe $\sD_{\calA}$ of
splittings over $\sX$, i.e., for any $U\ra S$ we have
\begin{equation}\label{ass azumaya}\sD_{\calA}(U)=\{(x,\calV,i)|x\in\sX(U),i:\calE nd(\calV)\is x^*(\calA)\}.\end{equation}

We will use the following proposition in the sequel:
\begin{proposition}[\cite{DP2},\ \S 2.1.2]\label{S_F}
Let $\calA$ be a sheaf of Azumaya algebras on $\sX$. There is the
following equivalence of categories
$$\on{QCoh}(\sD_{\calA})_1\is \mA\on{-mod}(\on{QCoh}(\sX))$$
where $\mA\on{-mod}(\on{Qcoh}(\sX))$ is the category of
$\mA$-modules which is quasi-coherent as $\mO_{\sX}$-modules.
\end{proposition}

\subsection{$\mD$-module on scheme}
Let $X$ be a scheme smooth over $S$. Let $\mD_{X/S}$ be the sheaf of
crystalline differential operators on $X$, i.e., $\mD_{X/S}$ is the
universal enveloping $\mD$-algebra associated to the relative
tangent Lie algebroid $T_{X/S}$. By definition, the category of
$\mD$-modules on $X$ is the category of modules over $\mD_{X/S}$ that
are quasi-coherent as $\mO_X$-modules. We denote by
$\mD\on{-mod}(X)$ the category of $\mD$-modules on $X$. In the case
$p\mO_S=0$, we have the following  fundamental observation:
\begin{thm}[ \cite{BMR},\ \S1.3.2,\ \S2.2.3]\label{BMR}
The center of $(F_{X/S})_*\mD_{X/S}$ is isomorphic to
$\mO_{T^*(X'/S)}$ and there is an Azumaya algebra $D_{X/S}$ on
$T^*(X'/S)$ such that
\[(F_{X/S})_*\mD_{X/S}\is(\tau_{X'})_*D_{X/S}.\]
where $\tau_{X'}:T^*(X'/S)\ra X'$ is the natural projection.
\end{thm}

In particular, we have the following:
\begin{corollary}\label{Azumaya-gerbe}
There is a canonical equivalence of categories
\[\mD\on{-mod}(X)\is\on{QCoh}(\sD_{D_{X/S}})_1\]
where $\sD_{D_{X/S}}$ is the gerbe of splittings of $D_{X/S}$.
\end{corollary}

In what follows, the gerbe $\sD_{D_{X/S}}$ will be denoted by
$\sD_{X/S}$ for simplicity.

\subsection{$\mD$-module on stack}\label{gerbe}
Let $S$ be a Noetherian scheme and $p\mO_S=0$. Let $\sX$ be a smooth
algebraic stack over $S$. A $\mD$-module $M$ on $\sX$ is an
assignment for each $U\ra \sX$ in $\sX_{sm}$, a $\mD_{U/S}$-module
$M_U$ and for each morphism $f:V\ra U$ in $\sX_{sm}$ an isomorphism
$\phi_f:f^*M_U\is M_V$ which satisfies the cocycle conditions. We denote
the category of $\mD$-modules  on $\sX$ by $\mD\on{-mod}(\sX)$.

Unlike the case of schemes, in general there does not exist a sheaf
of algebras $\mD_{\sX/S}$ on $\sX$ such that the category of
$\mD$-modules on $\sX$ is equivalent to the category of modules over
$\mD_{\sX/S}$, and therefore the naive stacky generalization of
Theorem \ref{BMR} is wrong. On the other hand, it is shown in \cite{Trav} that
the obvious stacky version of Corollary \ref{Azumaya-gerbe}  is
correct:

\begin{prop}\label{stacky-Azumaya-gerbe}
There exists a $\bG_m$-gerbe $\sD_{\sX/S}$ on $T^*(\sX'/S)$ such
that the category of twisted sheaves on $\sD_{\sX/S}$ is equivalent
to the category of $\mD$-modules on $\sX$, i.e., we have
\[\mD\on{-mod}(\sX)\is\on{QCoh}(\sD_{\sX/S})_1.\]
\end{prop}


\begin{remark}\label{derived}
It is a theorem of Gabber that on a quasi-projective scheme $X$,
every torsion element in $\on{H}^2_{et}(X,\bbG_m)$ can be
constructed from an Azumaya algebra via \eqref{ass azumaya}.
However, this fails for non-separated schemes. A theorem of T\"{o}en
\cite{Toen} shows that in a very general situation, every
$\bG_m$-gerbe arises from a derived Azumaya algebra. Although T\"{o}en's theorem does not directly apply to $T^*(\sX'/S)$, it suggests that the
derived category of $\mD$-modules on $\sX$ (which is not the derived
category of $\mD\on{-mod}(\sX)$ in general) probably should be equivalent to the
category of modules over some derived Azumaya algebra $D^{dr}_{\sX/S}$
on $T^*(\sX'/S)$. 
\end{remark}


Let us sketch the construction of the $\bG_m$-gerbe $\sD_{\sX/S}$ on
$T^*(\sX'/S)$. As gerbes satisfy smooth descent, it is enough to
supply a $\bG_m$-gerbe $(\sD_{\sX/S})_U$ on
$T^*(\sX/S)\times_{\sX'}U'$ for every $U\ra\sX$ in $\sX_{sm}$ and
compatible isomorphisms for any $\beta:U\ra V$ in $\sX_{sm}$. But
for any $f:U\ra\sX$ in $\sX_{sm}$ we have
\[ (f_U')_d:T^*(\sX/S)\times_{\sX'}U'\ra T^*(U'/S).\]
We have a $\bG_m$-gerbe $\sD_{U/S}$ on $T^*(U'/S)$ corresponding to
the sheaf of relative differential operators $\mD_{U/S}$. We define
a $\bG_m$-gerbe $(\sD_{\sX/S})_U$ on $T^*(\sX/S)\times_{\sX'}U'$ as the pull back of $\sD_{U/S}$ along $(f_U')_d$.  One can check
that these gerbes $(\sD_{\sX/S})_U$ are compatible under pullbacks,
and therefore define a $\bG_m$-gerbe $\sD_{\sX/S}$ on $\sX$.

Let $f:\sX\ra\sY$ be a schematic morphism between two smooth
algebraic stacks. From the above construction, the following lemma clearly follows from its scheme theoretic version. 
\begin{lemma}[\cite{Trav}]\label{BB lemma}
(1) There is a canonical $1$-morphism of $\bG_m$-gerbe on
$T^*(\sY'/S)\times_{\sY'}\sX'$
$$M_f:(f_p')^*\sD_{\sY/S}\is (f'_d)^*\sD_{\sX/S}.$$

(2) For a pair of morphisms $\sX\stackrel{g}\ra\sZ\stackrel{h}\ra\sY$
and their composition $f=h\circ g:\sX\ra\sY$
, there is a canonical $1$-morphisms
of $\bG_m$-gerbe on $T^*(\sY'/S)\times_{\sY'}\sX'$
\[M_{g,h}:(f_p')^*\sD_{\sY/S}\is (f'_d)^*\sD_{\sX/S},\]
together with a canonical $2$-morphism between $M_{h\circ g}$ and $M_{g,h}$.

(3) We have a canonical $1$-morphism of $\bG_m$-gerbe on $T^*(\sX'/S)^{sm}$:
\[\sD_{\sX/S}|_{T^*(\sX'/S)^{sm}}\is
\sD_{T^*(\sX'/S)^{sm}/S}(\theta_{can}):=\theta_{can}^*(\sD_{T^*(\sX'/S)^{sm}/S}),\]
where $T^*(\sX'/S)^{sm}$ is the maximal smooth open substack of $T^*(\sX'/S)$
and $\theta_{can}:T^*(\sX'/S)^{sm}\ra T^*(T^*(\sX'/S)^{sm})$ is the 
canonical one form. 
\end{lemma}

Let us discuss a stacky version of \cite[\S 4.3]{OV}. Let $\sX/S$ be
a smooth algebraic stack as above and let $\sP
ic^\natural(\sX/S)$ be the Picard stack of invertible sheaves on
$\sX$ equipped with a connection (i.e. objects in $\sP
ic^\natural(\sX/S)$ are $\mD$-modules on $\sX$ whose underlying
quasi-coherent sheaves are invertible). 
Let $B'_S=\Sect_S(\sX',T^*(\sX'/S))$. Note that the following proposition does not need the representability of $\sP ic^\natural(\sX/S)$.

\begin{prop}\label{OV stack}
(1) There is a natural morphism $\psi: \sP ic^\natural(\sX/S)\to
B'_S$.

(2) The pullback of the gerbe $\sD_{\sX/S}$ along $$\sX'\times_S
\sP ic^\natural(\sX/S)\stackrel{\id\times\psi}{\to}
\sX'\times_SB'_S\to T^*(\sX'/S)$$ is canonically trivialized.
\end{prop}
\begin{proof}
For (1), recall
that if $\sX$ is a scheme, the morphism $\psi$ is given by the
$p$-curvature map (see \cite[\S 4.3]{OV}). We explain how to generalize
this map to stacks. Let $U\to \sX$ be a smooth morphism. Then via
pullback, we obtain a morphism $\sP ic^\natural(\sX/S)\to \sP
ic^\natural(U/S)\to \Sect_S(U',T^*(U'/S))$. By considering further
pullbacks to $V=U\times_{\sX}U$, we find that the above maps fit
into the following commutative diagram
\[\begin{CD}
\sP ic^\natural(\sX/S)@>>> \sP ic^\natural(U/S)\\
@V\psi_U VV@VVV\\
\Sect_S(U',T^*(\sX'/S)\times_{\sX'}U')@>>> \Sect_S(U',T^*(U'/S)).
\end{CD}\]
These $\psi_U$'s are compatible under pullbacks and define the
$\pi:\sP ic^\natural(\sX/S)\to B'_S$.

For (2), again let $U\to \sX$ be a smooth morphism. Note that the
pullback of the gerbe $\sD_{U/S}$ along $U'\times_S \sP
ic^\natural(U/S)\to T^*(U'/S)$ is canonically trivialized by the
object $F_*(\calL,\nabla)$, where $(\calL,\nabla)$ is the universal
object on $U\times_S\sP ic^\natural(U/S)$. Combining this with Lemma
\ref{BB lemma} and the proof of part (i), this shows that the
pullback of $\sD_{\sX/S}$ along $U'\times_S\sP ic^\natural(\sX/S)\to
\sX'\times_S \sP ic^\natural(\sX/S)$ is canonically trivialized.
These trivializations glue together and give a canonical
trivialization of $\sD_{\sX/S}$ on $\sX'\times_S \sP
ic^\natural(\sX/S)$.
\end{proof}

\subsection{1-forms}\label{one forms}
In this subsection we make a digression into the construction of gerbes using 1-forms. We refer to \cite[Appendix A.8]{CZ} for more 
details.
Recall that for any smooth algebraic stack $\sX/S$ we can associate to it a 
$\bG_m$-gerbe $\sD_{\sX/S}$ on $T^*(\sX'/S)$. Thus giving a 1-form 
$\theta:\sX'\ra T^*(\sX'/S)$ we can construct a $\bG_m$-gerbe $\sD(\theta):=
\theta^*\sD_{\sX/S}$ on $\sX'$ by pulling back $\sD_{\sX/S}$ along $\theta$.

When $\sX=X$ is a smooth Noetherian scheme,
above construction can be generalized as follows. Let $\mG$ be a smooth affine commutative 
group scheme over $X$. For any section $\theta$ of $\Lie\mG'\otimes\Omega_{X'/S}$ we 
can associate to it a $\mG$-gerbe $\sD(\theta)$ on $X'$
using the four term exact sequence constructed in \emph{loc. cit.}. In the case $\mG=\bG_m$, the 
$\bG_m$-gerbe $\sD(\theta)$ is isomorphic to $\theta^*\sD_{X,S}$
the pull back of $\sD_{X/S}$ along $\theta:\sX'\ra T^*(X'/S)$.
We have the following functorial properties:

\begin{lemma}\label{theta}
\
\begin{enumerate}
\item
Let $\sY$ be another smooth algebraic stack over $S$ and let $f:\sY\ra\sX$ be 
a morphism. Let $\theta$ be a 1-form on $\sX$.
There is a canonical equivalence 
of $\bG_m$-gerbes on $\sY'$
\[f'^*\sD(\theta)\is\sD(f'^*\theta).\]
\item 
Let $X$ be a smooth Noetherian scheme and 
let $\phi:\mG\ra\mH$ be a morphism of smooth commutative affine group schemes over $X$. 
For any section $\theta$ of $\Lie\mG'\otimes\Omega_{X'}$ let 
$\phi'_*\theta$ denote its image $\Lie\mH'\otimes\Omega_{X'/S}$ under the map induced by $\phi$.
There is a canonical equivalence 
of $\mH'$-gerbes on $X'$
\[\sD(\theta)^{\phi'}\is\sD(\phi'_*\theta),\]
where $\sD(\theta)^{\phi'}$ is the $\mH'$-gerbe induced form $\sD(\theta)$ 
using the map $\phi'$ (see \ref{induction functor}).
\end{enumerate}
\end{lemma}

\subsection{Azumaya property of differential operators on good stacks}\label{azumaya property}

Recall that a smooth algebraic stack $\sX$ over $S$ of relative
dimension $d$ is called relatively good if it satisfies the following
equivalent properties:
\begin{enumerate}
\item
$\on{dim}(T^*(\sX/S))=2d$.
\item
$\on{codim}\{x\in\sX|\on{dim}Aut(x)=n\}\geqslant n$ for all $n>0$.
\item
For any $U\ra\sX$ in $\sX_{sm}$, the complex
$$\on{Sym}(T_{U/\sX}\ra T_{U/S})$$
has cohomology concentrated in degree $0$ and
$$H^0(\on{Sym}(T_{U/\sX}\ra T_{U/S}))
\is \on{Sym}(T_{U/S})/T_{U/\sX}\on{Sym}(T_{U/S}).$$
\end{enumerate}

The following proposition is proved in \cite{BB} (see also
\cite{Trav}).
\begin{prop}\label{Azumaya}
Let $\sX$ be a relatively good stack. Let $\pi_{\sX}:T^*(\sX/S)\ra\sX$
be the natural projection and $\pi_{\sX'}$ be its Frobenius twist.
Let $T^*(\sX'/S)^0$ be the maximal smooth open substack of
$T^*(\sX'/S)$. Then
\begin{enumerate}
\item There is a natural coherent sheaf of algebras $D_{\sX/S}$
on $T^*(\sX'/S)$ such that the restriction of $D_{\sX/S}$ to
$T^*(\sX'/S)^0$ is an Azumaya algebra on $T^*(\sX'/S)^0$ of rank
$p^{2\on{dim}(\sX/S)}$.
\item The $\bG_m$-gerbe $\sD_{\sX/S}^0:=\sD_{\sX/S}|_{T^*(\sX'/S)^0}$ is isomorphic to
$\sD_{D_{\sX/S}^0}$, the gerbe of splittings of $D_{\sX/S}^0$. In
particular, we have
\[D_{\sX/S}^0\on{-mod}\is\on{QCoh}(\sD_{\sX/S}^0)_1.\]
\end{enumerate}
\end{prop}
\begin{remark}
By Proposition \ref{stacky-Azumaya-gerbe}, the category
$D_{\sX/S}^0\on{-mod}$ can be thought as a localization of the
category of $\mD$-modules on $\sX$.
\end{remark}


\section{Abelian Duality}\label{C}
\subsection{Abelian duality for Beilinson's 1-motives}\label{ab duality for Beilinson 1-motive}

Assume that $S$ is a scheme and $p\mO_S=0$. Let $\sA$ be a Picard stack over $S$. In this subsection, we denote the base change of $\sA$ along $Fr_S:S\to S$ by $\sA'$ instead of $\sA^{(S)}$.
Let
$\bbT^*_e\sA'$ be the vector bundle on $S$, which is the restriction of the relative (to $S$) cotangent
bundle of $\sA'$ along $e:S\to \sA'$. Then there is a
canonical isomorphism
$$\sA'\times_S\bbT^*_e\sA'\is T^*(\sA'/S).$$ Therefore, via the map $\pi_S: T^*(\sA'/S)\simeq \sA'\times_S\bbT^*_e\sA'\ra \bbT^*_e\sA'$, $T^*(\sA'/S)$ becomes a Picard stack over
$\bbT^*_e\sA'$ and we denote by $m_S$ the multiplication map:
$$m_S:T^*(\sA'/S)\times_{\bbT^*_e\sA'}T^*(\sA'/S)\ra T^*(\sA'/S).$$
Recall that it makes sense to talk about a gerbe on a Picard stack
with a commutative group structure (cf.  \ref{app:comm}).

\begin{lem}\label{comm grp}
The gerbe $\sD_{\sA/S}$ on $T^*(\sA'/S)$ admits a canonical
commutative group structure.
\end{lem}
\begin{proof}
Let us sketch the construction of the multiplicative structure $M$ and
the $2$-morphisms $\gamma$ and $i$
 in \ref{app:comm}.
The multiplication $m:\sA\times_S\sA\to \sA$, which induces
$$
\xymatrix{T^*(\sA'/S)\times_{\sA'}(\sA'\times_S\sA')\ar[r]^{\ \ \ \ m_d}\ar[d]^{m_p}&T^*(\sA'\times_S\sA'/S)\\
T^*(\sA'/S)}.
$$
Observe that the map
$m_d:T^*(\sA'/S)\times_{\sA'}(\sA'\times_S\sA')\to
T^*(\sA'\times_S\sA'/S)\simeq T^*(\sA'/S)\times_ST^*(\sA'/S)$
induces an isomorphism
\[T^*(\sA'/S)\times_{\sA'}(\sA'\times_S\sA')\simeq T^*(\sA'/S)\times_{\bbT^*_e\sA'}T^*(\sA'/S)\to T^*(\sA'/S)\times_ST^*(\sA'/S).\]
Under this isomorphism $m_p$ becomes the multiplication map $m_S$.
Now the canonical $1$-morphism between $m_S^*\sD_{\sA/S}$ and
$\sD_{\sA/S}\boxtimes\sD_{\sA/S}$ comes from Lemma \ref{BB lemma}.
We have two different factorizations of the multiplicative morphism 
$\sA\times_S\sA\times_S\sA\ra\sA$ and 
the $2$-morphisms $\gamma$ comes from the $2$-morphisms for corresponding equivalences of Lemma \ref{BB lemma}. Finally, the $2$-morphism $i:\sigma^*M\is M$ can be 
constructed by applying Lemma \ref{BB lemma} to the morphism
 $\sA\times_S\sA\stackrel{\sigma}\ra\sA\times_S\sA\stackrel{m}\ra\sA$.
\end{proof}

Now we assume that $\sA$ is a Beilinson's 1-motive and is good when regarded as an algebraic stack.
Let $\sA^\natural:=\sP ic^\natural(\sA)$ be the Picard stack of
multiplicative invertible sheaves on $\sA$ with a connection (cf. \cite{Lau}), and
let $\psi_S:\sA^\natural\ra \bbT^*_e\sA'$ be the $p$-curvature
morphism as given in Proposition \ref{OV stack} (1). By \cite[\S 4.3]{OV},
there is a natural action of
$T^*(\sA'/S)^\vee\is(\sA')^\vee\times_S\bbT^*_e\sA'$ on
$\sA^\natural$. Concretely, for any $b:U\ra \bbT^*_e\sA'$
objects in $\sA^\natural\times_{\bbT^*_e\sA'}U$ consist of
multiplicative line bundles on $\sA\times_SU$ with a connection whose $p$-curvature is equal to $b$. Then for any
$\mL'\in(\sA')^\vee\times_SU\is T^*(\sA'/S)^\vee\times_SU$ and
$(\mL,\nabla)\in\sA^\natural\times_{B_S'}U$ we define
$$\mL'\cdot(\mL,\nabla)
:=(F_{\sA}^*\mL'\otimes\mL,\nabla_{F_\sA^*\mL'}\otimes\nabla),$$ where
$\nabla_{F_\sA^*\mL'}$ is the canonical connection on $F_\sA^*\mL'$ giving
by the Cartier descent. It also follows from the Cartier descent that
$\sA^\natural$ is a $T(\sA'/S)^\vee$-torsor under this action.

On the other hand, recall that for a $\bG_m$-gerbe $\sD$ with
commutative group structure on a Beilinson's 1-motive $\sP$, we
defined the $\sP^\vee$-torsor $\sT_\sD$ of multiplicative splittings of $\sD$ (cf. \ref{DFT}).

\begin{prop} There is a canonical $(T^*(\sA'/S))^\vee$-equivariant
isomorphism $\sA^\natural\to \sT_{\sD_{\sA/S}}$.
\end{prop}
\begin{proof}We sketch the proof.
Write $\sT_{\sD_{\sA/S}}$ by $\sT_\sD$ for simplicity. Recall that for $U\to \bbT^*_e\sA'$,
 $\sT_{D_A}(U)$ is the groupoid of splittings of
$\sD_{\sA/S}$ over $U\times_{\bbT^*_e \sA'}T^*(\sA'/S)$ which are compatible with the commutative group
structure of $\sD_{\sA/S}$. 
Note that
\[U\times_{\bbT^*_e\sA'}T^*(\sA'/S)\is U\times_{\bbT^*_e\sA'}(\bbT^*_e\sA'\times_S\sA')\is \sA'\times_SU,\] and under this isomorphism, the projection of left hand side to the second factor is identified with
\[\sA'\times_SU\ra\sA'\times_S\sA^\natural\ra T^*(\sA'/S).\]
Now by Lemma \ref{OV stack}, the pull back of $\sD_{\sA/S}$ to
$\sA'\times_SU$ has a canonical splitting $\mL_{U,\alpha}$.
Moreover, one can check that this canonical splitting is
compatible with the commutative group structure of $\sD_{\sA/S}$.
Thus the assignment $(U,\alpha)\ra\mL_{U,\alpha}$ defines a map from
$\sA^\natural$ to $\sT_{\sD}$ which is compatible with their
$T(\sA'/S)^\vee$-torsor structures hence an equivalence.
\end{proof}

As a corollary, we obtain the following theorem.
\begin{thm}Let $\sA$ be a good Beilinson's 1-motive. Then
there is a canonical equivalence of categories
\[D^b(\mD\on{-mod}(\sA))\simeq D^b(\on{QCoh}(\sA^\natural)).\]
\end{thm}
\begin{proof}This is the combination of Theorem \ref{twist FM} and Proposition \ref{Azumaya}.
\end{proof}
\begin{remark}Note that in \cite{Lau}, this theorem is proved for abelian schemes over $S$ of characteristic zero. In fact, Laumon's construction applies to any ``good" Beilinson's 1-motive over a locally noetherian base. When $p\mO_S=0$, it is easy to see that Laumon's equivalence and the equivalence constructed above are the same. 
\end{remark}

In particular, let $\theta: \sA'\to \bbT^*\sA'$ be a section obtained by base change $\tau:S\to \bbT_e^*\sA'$. Let $\sD_{\sA/S}(\theta):=\theta^*\sD_{\sA/S}$. Then $\sD_{\sA/S,\theta}$ is a $\bbG_m$-gerbe on $\sA'$ equipped with a canonical commutative group structure, and the ${\sA'}^\vee$-torsor $\sT_{\sD_{\sA/S,\theta}}$ of multiplicative splittings can be identified with $\sA^\natural\times_{\bbT^*_e\sA',\tau}S$.

\subsection{A variant}\label{appen:var}
In the main body of the paper, however, we need a variant of the above construction. Let $k$ be an algebraically closed field of characteristic $p$. For a $k$-scheme $X$, we denote by $X'$ its Frobenius base change along $Fr:k\to k$.
Let $S$ be a smooth $k$-scheme. For an $S$-scheme $X\to S$, we denote by $X^{(S)}$ its base change along $Fr_S:S\to S$. 
Let $\sA\to S$ be a Picard stack with multiplication $m:\sA\times_S\sA\to \sA$. The goal of this subsection is to construct certain multiplicative gerbe $\sD_{\sA}(\theta)$ on $\sA'$ (rather than on $\sA^{(S)}$ as done at the end of the previous subsection).

Let $\theta: \sA'\to T^*\sA'$ be a section, where $T^*\sA'$ is the cotangent bundle of $\sA'$ relative to $k$. We say $\theta$ is multiplicative if the upper right corner of the following diagram is commutative
\[\xymatrix{
T^*\sA'\times T^*\sA'& T^*\sA'\times T^*\sA'|_{\sA'\times_{S'}\sA'}\ar[l]\ar[r]& T^*(\sA'\times_{S'}\sA')\\
\sA'\times\sA'\ar_{\theta\times\theta}[u]&\sA'\times_{S'}\sA'\ar[l]\ar_{\theta\times\theta}[u]\ar^{m^*\theta}[r]\ar^m[d]& T^*\sA'\times_{\sA'}(\sA'\times_{S'}\sA')\ar^{m_p}[d]\ar_{m_d}[u]\\
&\sA'\ar^{\theta}[r]&T^*\sA'.
}\]
Let $\sD_{\sA}(\theta)=\theta^*\sD_{\sA}$ be the pullback of $\sD_{\sA}$ to $\sA'$.  Then by the same argument as in Lemma \ref{comm grp}, we have
\begin{lem}\label{appen:mult} (See also \cite[Lemma 3.16]{BB})
Let $\theta:\sA'\to T^*\sA'$ be a multiplicative section. Then
$\sD_{\sA}(\theta)$ is a $\bbG_m$-gerbe on $\sA'$ with a commutative group structure. 
\end{lem}

\end{document}